\documentclass[a4paper,10pt]{amsart}

\usepackage{color}
\definecolor{Chocolat}{rgb}{0.36, 0.2, 0.09}
\definecolor{BleuTresFonce}{rgb}{0.215, 0.215, 0.36}

\usepackage[colorlinks,final,hyperindex]{hyperref}
\hypersetup{citecolor=BleuTresFonce, linkcolor=Chocolat}

\usepackage{graphicx}
\usepackage{epstopdf}
\usepackage[all]{xy}
\usepackage{amsmath,amsfonts,amssymb,MnSymbol}
\usepackage[centering,vscale=0.68,hscale=0.68]{geometry}
\usepackage{latexsym,amscd}
\usepackage{tikz}
\usetikzlibrary{arrows,decorations.markings}
\tikzset{>=latex}
\usepackage{url}
\usepackage{dsfont}
\usepackage{microtype}
\usepackage{todonotes}
\makeatletter
\providecommand\@dotsep{5}
\renewcommand{\listoftodos}[1][\@todonotes@todolistname]{%
  \@starttoc{tdo}{#1}}
\makeatother

\newtheorem*{theorem*}{Theorem}
\newtheorem{theorem}{Theorem}[section]
\newtheorem{lemma}[theorem]{Lemma}
\newtheorem{proposition}[theorem]{Proposition}
 
\theoremstyle{definition}  
\newtheorem{definition}[theorem]{Definition}
\newtheorem{example}[theorem]{Example}

\newtheorem{conjecture}[theorem]{Conjecture}  
\newtheorem{remark}[theorem]{Remark}

\newcommand{\<}{\langle} 
\renewcommand{\>}{\rangle}

\newcommand{\on}{\operatorname}

\newcommand{\kk}{\mathbf{k}}
\newcommand{\Z}{{\mathbb Z}}
\newcommand{\Q}{{\mathbb Q}}
\newcommand{\C}{{\mathbb C}}
\newcommand{\RR}{{\mathbb{R}}}

\newcommand{\KK}{\mathbf{k}}

\newcommand{\Pa}{\mathbf{Pa}}
\newcommand{\PaB}{\mathbf{PaB}}
\newcommand{\PaCD}{\mathbf{PaCD}}
\renewcommand{\CD}{\mathbf{CD}}

\renewcommand{\t}{{\mathfrak{t}}}

\newcommand{\h}{{\mathfrak{h}}}

\newcommand{\f}{{\mathfrak{f}}}

\renewcommand{\i}{\on{i}}

\newcommand{\PB}{\on{PB}}

\newcommand{\zz}{\mathbf{z}}

\newcommand{\Assoc}{\mathbf{Ass}}
\newcommand{\Ass}{\mathbf{Ass}}
\newcommand{\Ell}{\mathbf{Ell}}

\newcommand{\GT}{\mathbf{GT}}
\newcommand{\GRT}{\mathbf{GRT}}

\newcommand{\cO}{{\mathcal O}}

\newcommand{\cA}{{\mathcal A}}
\newcommand{\cF}{{\mathcal F}}

\newcommand{\CoB}{\tmmathbf{\tmop{CoB}}}

\newcommand{\B}{\on{B}}

\newcommand{\Conf}{\on{Conf}}
\newcommand{\Conff}{\on{Conf}^f}

\tikzstyle cross=[preaction={draw=white, -, line width=4pt}, thick]
\tikzstyle normal=[thick]
\tikzstyle chord=[densely dotted, thick]
\tikzstyle zero=[ultra thick, gray]
\tikzstyle zell=[ultra thick, white]
\tikzstyle zerocross=[preaction={draw=white, -, line width=4pt}, ultra thick, gray]
\tikzstyle point=[draw,circle,inner sep=1,fill=black]
\tikzstyle petitpoint=[draw,circle,inner sep=0.3,fill=black]

\newcommand{\straight}[3][-]{\draw[normal,#1] (#2,-#3) -- (#2,-#3-1);}

\newcommand{\hori}[4][-]{\draw[normal,#1] (#2,-#3)--(#2,-#3-1);\draw[normal,#1] (#2+#4,-#3)--(#2+#4,-#3-1);\draw[chord] (#2,-#3-0.5)--(#2+#4,-#3-0.5);}

\newcommand{\tell}[4][-]{\draw[zell] (0,-#2)--(0,-#2-1);\draw[normal,#1] (#3,-#2)--(#3,-#2-1);\draw[chord] (0,-#2-0.5)--(#3,-#2-0.5); \node[point,label=left:$#4$] at (0,-#2-0.5) {};}

\newcommand{\tik}[1]{\begin{tikzpicture}[baseline=(current bounding box.center)] #1 \end{tikzpicture} }

\catcode`\<=\active \def<{
\fontencoding{T1}\selectfont\symbol{60}\fontencoding{\encodingdefault}}
\catcode`\>=\active \def>{
\fontencoding{T1}\selectfont\symbol{62}\fontencoding{\encodingdefault}}
\newcommand{\assign}{:=}

\newcommand{\tmmathbf}[1]{\ensuremath{\boldsymbol{#1}}}
\newcommand{\mathbbm}{\mathbb}

\newcommand{\tmop}[1]{\ensuremath{\operatorname{#1}}}

\newcommand{\tmtextit}[1]{{\itshape{#1}}}

\title{Surface Drinfeld torsors I : Higher genus associators}
\author{Martin Gonzalez}
\address{Martin GONZALEZ \newline \indent I2M, Universit\'e 
d'Aix-Marseille, 39, rue F. Joliot Curie, 13453, Marseille, France.}
\email{martin.gonzalez@univ-amu.fr}

\begin{document}

\begin{abstract} 
We develop a higher genus version of Drinfeld associators by means of operad theory. 
We start by introducing a framed version of rational associators and Grothendieck--Teichm\"uller 
groups and show that their definition is independent of the framing data. Next, we define 
a framed version of the universal KZ connection and we use it to show that over 
the complex numbers, the rational framed Drinfeld torsor is not empty. Next, we 
concentrate on the higher genus version of this story. We define an operad module of 
framed parenthesized higher genus braidings in prounipotent groupoids and we define its chord diagram
counterpart. We then use these operadic modules to operadicly define higher genus associators and 
Grothendieck--Teichm\"uller groups, which again do not depend on the framing data. 
Finally, we compare our results in the genus $1$ case with those appearing in the litterature.
\end{abstract}

\maketitle

\setcounter{tocdepth}{2}

\tableofcontents

\section*{Introduction}

This article is the first of a series devoted to the study of surface analogs 
of the so-called rational  \textit{Drinfeld torsor}
 which consist, for each $\Q$-ring $\kk$, on the 
 (bi-)torsor $(\on{Ass}(\kk),\widehat{\on{GT}}(\kk), \on{GRT}(\kk))$ where 
$\on{Ass}(\kk)$ is the set of Drinfeld $\kk$-associators acted upon by the $\kk$-prounipotent 
Grothendieck--Teichm\"uller 
group $\widehat{\on{GT}}(\kk)$ and its graded version $\on{GRT}(\kk)$. 
Different authors \cite{DrGal, En2, CEE, En} already constructed analogs of
Drinfeld torsors in the cyclotomic and elliptic cases 
and operadic descriptions of these torsors have become available 
recently in the mentioned cases \cite{Fresse, CaGo2, CaGo3} and 
the twisted elliptic case \cite{CaGo2}.
In this article we will concentrate on the first part of this program. Namely, we 
develop the operadic construction of this torsor in the framed higher 
genus context. A second paper will be devoted to the study of
Drinfeld torsors associated to orbit configuration 
spaces which are finite (possibly ramified) covers over framed configuration spaces of points on 
oriented surfaces.

One motivation for studying them is that Drinfeld torsors consist on a somehow 
``useful fiction'' for finding interesting families of relations for analogs of multiple 
zeta values associated to a wide range of complex curves, 
through the study of algebraic relations 
satisfied by the monodromy of the flat universal KZB connection associated to the 
curve as such monodromy has proven to produce a $\C$-point in the set of Drinfeld associators 
associated to the curve in the already know cases.

Initially, Grothendieck-Teichm\"uller groups and associators were, in the genus 0, 
cyclotomic and genus 1 cases, constructed by using braided monoidal categories, 
braided modules categories and elliptic structures over braided monoidal categories 
respectively. Already in V. Drinfeld's work, 
associators had an implicit operadic nature (made explicit in {\cite{BN}})
which permits to define associators as formality isomorphisms between operads
closely related to the little disks operad $\on{D}_2$. More specifically,  for $\kk$ a $\Q$-ring, 
there are operads in $\kk$-prounipotent groupoids $\widehat{\PaB}(\kk)$, 
encapsulating the combinatorics of parenthesized braidings, and $G\PaCD(\kk)$, 
encapsulating the combinatorics of parenthesized chord diagrams. The former is 
obtained (roughly) by considering a parenthesized groupoid version of 
the pure braid group. The latter is obtained from the so-called 
Kohno-Drinfeld Lie $(\kk)$-algebras $\t_n(\kk)$. In this scope, 
the $\kk$-prounipotent Grothendieck-Teichm\"uller 
group consists on the group of automorphisms of $\widehat{\PaB}(\kk)$ which 
are the identity on objects, the graded Grothendieck-Teichm\"uller group is the 
group of automorphisms of $G\PaCD(\kk)$ which are the identity on objects, and 
the set of $\kk$-associators consists on the set of isomorphisms 
$\widehat{\PaB}(\kk)\longrightarrow G\PaCD(\kk)$ of operads in $\kk$-prounipotent 
groupoids which are the identity on objects. It can be shown that these operadic point 
of view is compatible with the classic one, namely that there is a one-to-one 
correspondence between the operadic definition of these objects and the objects 
defined in the literature in terms of elements satisfying certain equations.
Let us also mention that in \cite{Fresse}, B. Fresse developped very powerful 
tools to study rational 
homotopy theory of operads in order to understand, from a homotopical viewpoint, 
a deep relationship between operads and Grothendieck-Teichm\"uller groups which 
was first foreseen by M. Kontsevich in his work on deformation quantization process 
in mathematical physics. 

\medskip

Let us explain the general approach for constructing Drinfeld torsors 
in the framed higher genus context.

Let $n \geqslant 1$ and  let $M$ be a closed smooth manifold of dimension $2$.
Consider the configuration space of $M$
\[ \tmop{Conf} (M, n) = \{  \tmmathbf{x}=(x_1, \ldots, x_n) \in M^n ; x_i \neq x_j \tmop{if}
   i \neq j \} . \]
The spaces $\tmop{Conf} (M, n)$ are weakly equivalent to their
Axelrod--Singer--Fulton--MacPherson (ASFM) compactification $
\overline{\Conf}(M, n)$. These spaces are acted on by the symmetric group 
$\mathfrak{S}_n$ by relabelling the marked points and the collection 
$\overline{\Conf}(M, -):=\{\overline{\Conf}(M, n)\}_{n\geq0}$ is actually 
a $\mathfrak{S}$-module.
When $M$ is parallelizable, $\overline{\Conf}(M, -)$ forms a right
$\overline{\on{C}}(\RR^2,-)$-module $\overline{\Conf}(M, -)$. Otherwise, in order to 
endow $\overline{\Conf}(M, -)$ with a well defined operadic structure, we need to introduce 
framed versions of the above considered configuration spaces. This consists on setting a choice 
of trivialization of the tangent bundle of $M$ in order to specify in which direction 
we will insert the disks on $M$ constructed by the ASFM compactification.

Let $M$ be a Riemannian closed oriented\footnote{In the case of non-oriented 
manifolds one can only consider the bundle projection $\tmop{O} (M) \rightarrow M$.} 
compact $2$-manifold and consider the
bundle projection $\pi_M : \tmop{SO} (M) \rightarrow M$, where $\tmop{SO} (M)$
is the principal $\tmop{GL}_2$-bundle of special orthogonal linear frames on
$M$.
The framed configuration space $\Conf^f (M, n)$ of $n$ distinct points in $M$ is
  \[ \Conf^f (M, n) := \{ (\tmmathbf{x},
     f_1, \ldots, f_n) \in \tmop{Conf} (M, n) \times \tmop{SO} (M)^{\times n}
     | f_i \in \pi_M^{- 1} (x_i) \} . \]
This is the same to define $\Conf^f (M, n)$ as
the pullback of the diagram
$$ \xymatrix{ 
   & \on{SO}(M)^{\times n} \ar[d] \\
   \on{Conf}(M,n)\ar[r] & M^{\times n}   
   } 
$$
so $\Conf^f (M, n) \longrightarrow \tmop{Conf}
(M, n)$ is a principal $\tmop{SO} (2)^{\times n}$-bundle. If $M$ is
parallelizable, then $\Conf^f (M, n)$ is isomorphic to
$\tmop{Conf} (M, n) \times \tmop{SO} (2)^{\times n}$. For instance, this is the case when $M
=\mathbbm{R}^2$ (by considering its reduced version) or when $M=\mathbb{T}$.  

Let $\overline{\on{C}}^f(\RR^2,n)$
be the ASFM compactification of $
\on{C}^f(\RR^2,n)= \tmop{C} (\RR^2, n) \times \tmop{SO} (2)^{\times n}$. Now, if $M$ is 
an oriented $2$-manifold, then the collection of its framed
ASFM compactifications forms a right $
\overline{\on{C}}^f(\RR^2,-)$-module denoted $ \overline{\Conf}^f(M, -)$ where each space
$\overline{\Conf}^f(M, n)$ is a principal $\tmop{SO}
(2)^{\times n}$-bundle over $\overline{\Conf}(M, n)$.

In general, if $M=\Sigma_g$ has genus $g$, the $\mathfrak{S}$-module
$\on{D}^f_{2, g}$ of framed little 2-disks on $\Sigma_g$ can be endowed with a 
well-defined operadic module structure over the framed
little 2-disks operad $\on{D}^f_{2}$. 
In particular, if $g = 1$, as $\mathbbm{T}$ is pararellizable so each space
$\on{D}^f_{2,1} (n)$ is isomorphic to $\on{D}_{2,1} (n) \times \tmop{SO}
(2)^{\times n}$.

Then we also have
\begin{equation*}
\xymatrix{
\on{D}^f_M(n) \ar[r]^-{\simeq} \ar[d] & \Conf^f(M,n) & \ar[l]_-{\simeq} \overline{\Conf}^f(M, n) \ar[d] \\
\on{D}_M(n) \ar[r]^-{\simeq} & \on{Conf}(M,n)  & \ar[l]_-{\simeq} \overline{\Conf}(M, n)
}
\end{equation*}
where again the horizontal maps are $\mathfrak{S}_n$-equivariant homotopy
equivalences.

 If $M$ is parallelizable, then the semi-direct product in the
below spaces becomes an usual product and we get a square of $\mathfrak{S}$-modules
\begin{equation*}
\xymatrix{
\on{D}^f_M  \ar[d] & \ar[l]_-{\simeq} \overline{\Conf}^f(M, -) \ar[d] \\
\on{D}_M & \ar[l]_-{\simeq} \overline{\Conf}(M, -)
}
\end{equation*}
If $M$ is not
parallelizable the second line of this square doesn't enhance into an operadic 
module morphism but we still have a weak equivalence
$\overline{\Conf}^f(M, -) \overset{\simeq}{\longrightarrow}  \on{D}^f_M$
of modules over $\overline{\on{C}}^f(\RR^2, -)\overset{\simeq}{\longrightarrow}\on{D}^f_2 
$.

\medskip

\noindent\textbf{Plan of the paper.}
After briefly recalling the categorical and operadical language we use to define 
rigourously rational Drinfeld torsors in Section \ref{Preliminaries}, we introduce in 
Section \ref{framed associators} a full suboperad $\PaB^f \subset \pi_1 (\on{D}^f_{2})$ 
of framed parenthesized braidings. 
We do so by restricting the object
sets of the groupoid so that $B (\PaB^f)
\overset{\sim}{\longrightarrow} B (\pi_1(\on{D}^f_{2}))$. We then use a framed version 
of the Kohno-Drinfeld Lie $\kk$-algebra and construct an
operad $G\PaCD^f(\kk)$ of parenthesized framed (group-like) horizontal chord 
diagrams. These two operads will allow us to operadicly define the bitorsor consisting of framed 
associators and framed Grothendieck--Teichm\"uller groups. After showing that 
such torsor is not empty over $\C$ by means of the monodromy of a framed 
version of the universal KZ connection, we show that this torsor is isomorphic 
to the unframed rational Drinfeld torsor. As an application of this, we relate 
associators and Grothendieck--Teichm\"uller groups to the rational homotopy 
theory of the framed little disks operad.

We then turn in Section \ref{g framed associators} to the genus $g$ 
situation and we introduce a full submodule
$\PaB^f_g \subset \pi_1 (\on{D}^f_{2,g})$ of genus $g$ framed 
parenthesized braidings by restricting the object
sets of the groupoid so that $B (\PaB^f_g)
\overset{\sim}{\longrightarrow} B (\pi_1(\on{D}^f_{2,g}))$ and we give a presentation 
of this operadic module by using the presentation of surface frame braids from \cite{BeG}.
We then construct a framed version $\t_{g,n}^f(\kk)$ of the genus $g$ Kohno-Drinfeld 
Lie $\kk$-algebra contained in \cite{En4} and use it to define a $G\PaCD^f(\kk)$-module 
$G\PaCD^f_g(\kk)$ of genus $g$ parenthesized (group-like) framed chord diagrams.

The main result of this article is then the following

\begin{theorem*}[Theorem \ref{Assg}]
There is an isomorphism between the following two sets:
\begin{itemize}
\item  the set $\mathbf{A}\mathbf{s}\mathbf{s}^f_g (\kk)$
   of couples $(F,G)$, where
  $F$ is an operad isomorphism $\widehat{\PaB}^f(\kk) \to G\PaCD^f(\kk)$ 
  and $G$ is an isomorphism between
  the $\widehat{\PaB}^f(\kk)$-module
  $\widehat{\PaB}^f_g(\kk)$ and the
  $G\mathbf{P}\mathbf{a}\mathbf{C}\mathbf{D}^f(\kk)$-module
  $G\mathbf{P}\mathbf{a}\mathbf{C}\mathbf{D}^f_g(\kk)$ which is the
  identity on objects and which is compatible with $F$; 
\item the set $\on{Ass}_g(\kk)$ consisting on tuples $(\mu,
 \varphi,A^{1,2}_{1,\pm}, \ldots, A^{1,2}_{g,\pm})$ where $(\mu,
 \varphi) \in \on{Ass}(\kk)$ and $ A^{1,2}_{a,\pm} \in \exp
  (\hat{\mathfrak{t}}^f_{g, 2}(\kk))$, for $a=1,...,g$, satisfying equations 
  \eqref{def:g:ass:0}, \eqref{def:g:ass:1}, \eqref{def:g:ass:2}, \eqref{def:g:ass:3} and \eqref{def:g:ass:4}.
\end{itemize}
\end{theorem*}
Next, we operadicly define genus $g$ (graded) Grothendieck--Teichm\"uller groups, 
extract from them descriptions \textit{à la Drinfeld}.
We finish this section by making a conjecture on the existence of a 
genus $g$ $\C$-associator by means of a yet-to-be-defined framed extension of the genus $g$ 
universal KZB connection which was constructed in \cite{En4}.

Finally, in Section \ref{genus 1 associators} we compare different genus 1 
version of the Drinfeld torsor associated to configuration spaces of the 2-torus, 
depending on the framed/unframed and reduced/non-reduced versions of it.

\medskip

It should be interesting to relate the Lie algebra of our genus $g$ graded Grothendieck--Teichm\"uller 
group to the higher genus Kashiwara--Vergne Lie algebra $\mathfrak{krv}^{(g,n+1)}$ 
which is being studied in the recent work \cite{KVg}.
Finally, we should point out that the recent paper \cite{CIW} makes a complementary construction 
of higher genus associators which intersects ours in the genus 0 case. It should be interesting to 
link their construction and ours via the study of higher genus version of the Arnold-Kohno 
isomorphism $\kappa_n : \on{C}^*_{\on{CE}}(\t_n)  \longrightarrow  \on{H}^*(\on{C}(\C,n))$.

\medskip

\noindent\textbf{Acknowledgements.} The author is grateful to Damien Calaque, Benjamin 
Enriquez, Benoit Fresse and Geoffroy Horel for numerous conversations and 
suggestions about this project.
This paper is grew as a part of the author's doctoral thesis at Sorbonne Universit\'e, 
and part of this work has been done while the author was visiting the 
\textit{Institut Montpelli\'erain Alexander Grothendieck}, thanks to the financial 
support of the \textit{Institut Universitaire de France}. The author warmly thanks the 
Max-Planck Institute of Mathematics in Bonn, and Université d'Aix-Marseille for its hospitality 
and excellent working conditions that made it possible to finish this work.

\section{Preliminaries on rational Drinfeld torsors}
\label{Preliminaries}
In this section we fix the categorical and operadical notation that will be used 
later on and which is used consistently in \cite{CaGo2}. 
We also recall the definitions 
and results concerning the operadicly defined rational Drinfeld 
torsor which is the triple consisting of the
Grothendieck-Teichm\"uller groups, its graded version and the
 set  rational Drinfeld associators. Let $\kk$ be a $\Q$-ring.

\subsection{Pointed modules over an operad}
Consider a symmetric monoidal category $(\mathcal C,\otimes,\mathbf{1})$ 
having small colimits and such that $\otimes$ commutes with these. 
We make use of the following notation (we refer the reader to \cite{CaGo2} for further details):
\begin{itemize}
\item Let $\mathfrak S$-mod be the category of $\mathfrak S$-modules 
in $\mathcal C$, endowed with
\begin{itemize}
\item the 
symmetric monoidal product $\otimes$ defined by 
$$
(S\otimes T)(n):=\coprod_{p+q=n}\left(S(p)\otimes T(q)\right)_{\mathfrak S_p
\times\mathfrak S_q}^{\mathfrak S_n}\,,
$$
where, for each group inclusion $H\subset G$, $(-)_H^G$ is left adjoint 
to the restriction functor 
from the category of objects carrying a $G$-action to the category of 
objects carrying an $H$-action;
\item the monoidal unit defined by 
$$
\mathbf{1}_{\otimes}(n):=
\begin{cases} 
\mathbf{1} & \mathrm{if}~n=0. \\
\emptyset & \mathrm{else}
\end{cases}
$$
\end{itemize}
\item An \textit{operad in $\mathcal C$} is a unital monoid in 
$(\mathfrak S\textrm{-mod},\circ,\mathbf{1}_{\circ})$, where 
\begin{itemize}
\item $\circ$ is the (non-symmetric) monoidal product $\circ$ on $\mathfrak S$-mod defined by
$$
(S\circ T)(n):=\coprod_{k\geq0}T(k)\underset{\mathfrak S_k}{\otimes}\left(S^{\otimes k}(n)\right)\,.
$$
\item the monoidal unit $\mathbf{1}_{\circ}$ for $\circ$ is given by 
$$
\mathbf{1}_{\circ}(n):=
\begin{cases} 
\mathbf{1} & \mathrm{if}~n=1 \\
\emptyset & \mathrm{else}
\end{cases}.
$$ 
\end{itemize}
The category of operads in $\mathcal C$ will be denoted $\tmop{Op}\mathcal C$.
\item A \textit{module over an operad $\mathcal O$} (in $\mathcal C$) is a 
left $\mathcal O$-module in $(\mathfrak S\textrm{-mod},\circ,\mathbf{1}_{\circ})$. 
\item All of our operads (resp. operad modules) will be pointed in the sense of 
\cite[Subsection 1.8]{CG}. In particular, there are operations that decrease the 
arity and, in the case of modules, we have a distinguished morphism 
$\mathcal O\to\mathcal P$ of $\mathfrak{S}$-modules. 
\item Let $\mathcal{P} \to \mathcal{Q}$ be a morphism between operads in 
$\mathcal{C}$, let $\mathcal{M}$ 
be a module over $\mathcal{P}$, and let $\mathcal{N}$ be a module over 
$\mathcal{Q}$. Operadic module mophisms $\mathcal{M} \to \mathcal{N}$ 
are considered to lie in the category of $\mathcal{P}$-modules 
(\textit{via} the restriction functor), and will simply be refered to as module 
morphisms. 
\item We write $\tmop{OpR}\mathcal C$ for the category of 
pairs $(\mathcal P,\mathcal M)$, where $\mathcal P$ 
is an operad and $\mathcal M$ is a right $\mathcal O$-module, in $\mathcal C$. 
A morphism $(\mathcal P,\mathcal M)\to (\mathcal Q,\mathcal N)$ in $\tmop{OpR}\mathcal C$ is a pair $(f,g)$, where 
$f:\mathcal{P} \to \mathcal{Q}$ is a morphism between operads and $g:\mathcal{M} \to \mathcal{N}$ is a morphism 
of $\mathcal{P}$-modules. 
\item Let $\mathcal P, \mathcal Q$ be two operads (resp.~modules) in groupoids. 
If we are given a morphism $f:\on{Ob}(\mathcal P)\to\on{Ob}(\mathcal Q)$ between 
the operads (resp. operad modules) of objects of $\mathcal P$ and $\mathcal Q$, 
then (following \cite{Fresse}), the \textit{fake pull-back} operad 
(resp. operad module) $ f^\star \mathcal Q$ is defined by 
\begin{itemize}
\item $\on{Ob}(f^\star \mathcal Q) := \on{Ob}(\mathcal P)$,
\item $\on{Hom}_{(f^\star \mathcal Q)(n)}(p,q):=\on{Hom}_{\mathcal Q(n)}(f(p),f(q))$.
\end{itemize}
\item We denote by $\bf{CoAlg_{\KK}}$ the 
symmetric monoidal category of complete filtered topological coassociative 
cocommutative counital $\KK$-coalgebras, with monoidal product given by 
the completed tensor product $\hat{\otimes}_{\KK}$ over $\KK$. 
\item Let $\bf{Cat(CoAlg_{\kk})}$ be the symmetric monoidal category of 
small $\bf{CoAlg_{\KK}}$-enriched categories, with symmetric monoidal 
product $\otimes$ given by 
\begin{itemize}
\item $\on{Ob}(C \otimes C'):=\on{Ob}(C) \times \on{Ob}(C')$,
\item $\on{Hom}_{C \otimes C'}\big((c,c'),(d,d')\big):=
\on{Hom}_{C}(c,d) \hat\otimes_{\KK} \on{Hom}_{C'}(c',d')$. 
\end{itemize}
\item Let $grLie_\KK$ be the category of positively graded finite dimensional 
Lie $\kk$-algebras, with symmetric monoidal strucure is given by the direct sum $\oplus$. 
There is a lax symmetric monoidal functor 
$$
\hat{\mathcal{U}}:grLie_\kk\longrightarrow \mathbf{Cat(CoAlg_\KK)}
$$
sending a positively graded Lie algebra to the degree completion of its universal 
envelopping algebra, which is a complete filtered cocommutative Hopf algebra 
when viewed as a $\mathbf{CoAlg_\KK}$-enriched category with only one object. 
\item There is a functor that goes  from the category of surjective morphisms
 $G\to S$ with finitely generated kernel and with $S$ a finite group to the 
 category of groupoids. It sends  $\varphi:G\to S$ to the groupoid 
 $\mathcal{G}(\varphi)$ defined by $\on{Ob}(\mathcal{G}(\varphi))=S$ 
 and, for $s,s'\in S$, 
\[
\mathrm{Hom}_{\mathcal{G}(\varphi)}(s,s'):=\{g\in G|\varphi{g}=s^{-1}s'\}
\]
with multiplication of arrows in $\mathcal{G}(\varphi)$ identical to the one in $G$. 
\item $\on{Top}$ will denote the category of topological spaces endowed 
with the cartesian product as symmetric monoidal product.
\item Consider the symmetric monoidal category $\bf{Grpd}$ of groupoids, 
with symmetric monoidal structure given by the cartesian product. There is a 
\textit{$\KK$-prounipotent completion} functor 
$$\mathcal G\mapsto \hat{\mathcal G}(\KK):=G\big(\mathcal G(\KK)\big)$$ for 
operads (resp.~modules) in groupoids. Consider the symmetric monoidal 
category $\mathbf{Grpd}_\kk$ of $\kk$-prounipotent groupoids (being the 
image of the completion functor $\mathcal G\mapsto \hat{\mathcal G}(\kk)$). 
\item For $\mathcal C$ being $\mathbf{Grpd}$, $\mathbf{Grpd}_\kk$, or 
$\bf{Cat(CoAlg_{\kk})}$, the notation 
$$
\on{Aut}_{\tmop{Op}\mathcal C}^+\quad(\mathrm{resp.}~\on{Iso}_{\tmop{Op}\mathcal C}^+)
$$ 
refers to those automorphisms (resp.~isomorphisms) which are the identity on 
objects. Likewise, in the case of operadic modules, the superscript ``$+$'' still 
indicates that we consider couples of morphisms 
that are both the identity on objects. 
\end{itemize}

\subsection{Rational Drinfeld torsors}
\label{Rators}

Consider the inclusions of topological operads
$$
\mathbf{Pa}(-)\,\subset\,\overline{\textrm{C}}(\mathbb{R},-)\,
\subset\,\overline{\textrm{C}}(\mathbb{C},-)\,
$$
where, for any finite set $I$,
\begin{itemize}
\item $\mathbf{Pa}(I)$ is the set of ordered maximal parenthesizations of 
$\underbrace{\bullet\cdots\bullet}_{|I|~{\rm times}}$,
\item $\overline{\textrm{C}}(\mathbb{R},I)$ (resp. $\overline{\textrm{C}}(\mathbb{C},I)$) 
is the Axelrod--Singer--Fulton--MacPherson 
(ASFM) compactification of the reduced configuration space of points indexed 
by $I$ in $\RR$ (resp. in $\C$).
\end{itemize}
As a pointed operad in groupoids having $\mathbf{Pa}$ as operad of objects, 
$$\mathbf{PaB}:=\pi_1\left(\overline{\textrm{C}}(\mathbb{C},-),\mathbf{Pa}\right)$$
is freely generated by
\begin{center}
$R^{1,2}:=$
\begin{tikzpicture}[baseline=(current bounding box.center)]
\tikzstyle point=[circle, fill=black, inner sep=0.05cm]
 \node[point, label=above:$1$] at (1,1) {};
 \node[point, label=below:$2$] at (1,-0.25) {};
 \node[point, label=above:$2$] at (2,1) {};
 \node[point, label=below:$1$] at (2,-0.25) {};
  \draw[->,thick] (1,1) .. controls (1,0.25) and (2,0.5).. (2,-0.20);
 \node[point, ,white] at (1.5,0.4) {};
 \draw[->,thick] (2,1) .. controls (2,0.25) and (1,0.5).. (1,-0.20); 
\end{tikzpicture}
$\in \on{Hom}_{\PaB(2)}(12,21)$ ;
$\Phi^{1,2,3}:=$
\begin{tikzpicture}[baseline=(current bounding box.center)]
\tikzstyle point=[circle, fill=black, inner sep=0.05cm]
 \node[point, label=above:$(1$] at (1,1) {};
 \node[point, label=below:$1$] at (1,-0.25) {};
 \node[point, label=above:$2)$] at (1.5,1) {};
 \node[point, label=below:$(2$] at (3.5-1,-0.25) {};
 \node[point, label=above:$3$] at (4-1,1) {};
 \node[point, label=below:$3)$] at (4-1,-0.25) {};
 \draw[->,thick] (1,1) .. controls (1,0) and (1,0).. (1,-0.20); 
 \draw[->,thick] (1.5,1) .. controls (1.5,0.25) and (3.5-1,0.5).. (3.5-1,-0.20);
 \draw[->,thick] (4-1,1) .. controls (4-1,0) and (4-1,0).. (4-1,-0.20);
\end{tikzpicture}
$\in\on{Hom}_{\PaB(3)}((12)3,1(23))$
\end{center}
together with the following relations: 
\begin{flalign}
& \Phi^{\emptyset,2,3}=\Phi^{1,\emptyset,3}=\Phi^{1,2,\emptyset}=\on{Id}_{1,2}									\quad 
\Big(\mathrm{in}~\mathrm{Hom}_{\mathbf{PaB}(2)}\big(12,12\big)\Big)\,, 				\tag{R} \\
& R^{1,2}\Phi^{2,1,3}R^{1,3}=\Phi^{1,2,3}R^{1,23}\Phi^{2,3,1} 											\quad 
\Big(\mathrm{in}~\mathrm{Hom}_{\mathbf{PaB}(3)}\big((12)3,2(31)\big)\Big)\,, 				\tag{H1} \\
& (R^{2,1})^{-1}\Phi^{2,1,3}(R^{3,1})^{-1}=\Phi^{1,2,3}(R^{23,1})^{-1}\Phi^{2,3,1}	\quad 
\Big(\mathrm{in}~\mathrm{Hom}_{\mathbf{PaB}(3)}\big((12)3,2(31)\big)\Big)\,, 				\tag{H2} \\
& \Phi^{12,3,4}\Phi^{1,2,34}=\Phi^{1,2,3}\Phi^{1,23,4}\Phi^{2,3,4} 									\quad
\Big(\mathrm{in}~\mathrm{Hom}_{\mathbf{PaB}(4)}\big(((12)3)4,1(2(34))\big)\Big)\,.&	\tag{P}
\end{flalign}
We will denote $\tilde R^{1,2}:=(R^{2,1})^{-1}$ and we make use of the same notation as in \cite{CaGo2}.

For each $n\geq 1$ and each object $\mathbf{p}\in \PaB(n)$, the group $\on{Aut}_{\PaB(n)}(\mathbf{p})$ is 
isomorphic to the fundamental group of $\on{Conf}(\C,n)$, also known as the pure 
braid group with $n$ strands.
One can easily check that $\PB_n$ is generated by elementary pure braids 
$x_{ij}$, $1\leq i<j\leq n$, 
which satisfy a prescribed family of relations.
We will depict the generator $x_{ij}$ in the following two equivalent ways: 
\begin{center}
\begin{tikzpicture}[baseline=(current bounding box.center)]
\tikzstyle point=[circle, fill=black, inner sep=0.05cm]
	\node[point, label=above:$1$] at (0,1) {};
	\node[point, label=below:$1$] at (0,-1.5) {};
	\draw[->,thick, postaction={decorate}] (0,1) -- (0,-1.45);
	\node[point, label=above:$i$] at (1,1) {};
	\node[point, label=below:$i$] at (1,-1.5) {};
	\node[point, label=above:$...$] at (2,1) {};
	\node[point, label=below:$...$] at (2,-1.5) {};
	\draw[->,thick, postaction={decorate}] (2,1) -- (2,-1.45);
	\node[point, ,white] at (2,0.4) {};
	\node[point, ,white] at (2,-0.9) {};
	\node[point, label=above:$j$] at (3,1) {};
	\node[point, label=below:$j$] at (3,-1.5) {};
	\draw[thick] (1,1) .. controls (1,0.25) and (3.5,0.5).. (3.5,-0.25);
	\node[point, ,white] at (3,0.2) {};
	\draw[->,thick, postaction={decorate}] (3,1) -- (3,-1.45);
	\draw[->,thick, postaction={decorate}] (3,0) -- (3,-1.45);
	\node[point, ,white] at (3,-0.7) {};
	\draw[->,thick, postaction={decorate}] (3.5,-0.25) .. controls (3.5,-1) and (1,-0.75).. (1,-1.45);
	\node[point, label=above:$n$] at (4,1) {};
	\node[point, label=below:$n$] at (4,-1.5) {};
	\draw[->,thick, postaction={decorate}] (4,1) -- (4,-1.45);
\end{tikzpicture}
$\qquad\longleftrightarrow\quad\sphericalangle\quad$
\begin{tikzpicture}[baseline=(current bounding box.center)] 
\tikzstyle point=[circle, fill=black, inner sep=0.05cm]
	\draw[loosely dotted] (-1.5,1.5) -- (1.5,-1.5); 
	\node[point, label=below:1] at (-1.5,1.5) {}; 
	\node[point, label=above:$i$] at (-0.5,0.5) {};
	\node[point, label=above:$j$] at (0.5,-0.5) {};
	\node[point, label=above:$n$] at (1.5,-1.5) {};
	\draw[thick, postaction={decorate}] (-0.5,0.5) .. controls (0,-0.5) and (-0.25,-0.2).. (0.6,-0.4); 
	\draw[->,thick] (0.6,-0.4) .. controls (0.7,-0.5)  .. (0.6,-0.6) ;
	\draw[thick, postaction={decorate}] (0.6,-0.6) .. controls (0,-0.75) .. (-0.5,0.5);
\end{tikzpicture}
\end{center}

The holonomy Lie algebra of 
the configuration space $\on{Conf}(\C,n)$ is isomorphic to the so-called 
\textit{Kohno-Drinfeld} graded Lie $\mathbb{C}$-algebra $\t_n$ 
generated by $t_{ij}$, $1\leq i\neq j\leq n$, with relations
\begin{flalign}
& t_{ij}=t_{ji}\,,                			                 			\tag{S} \label{eqn:S} \\
& [t_{ij},t_{kl}]=0\quad\textrm{if }\#\{i,j,k,l\}=4 \,,				\tag{L} \label{eqn:L} \\
& [t_{ij},t_{ik}+t_{jk}]=0\quad\textrm{if }\#\{i,j,k\}=3\,. &	\tag{4T} \label{eqn:4T} 
\end{flalign}
The collection of Kohno-Drinfeld Lie $\kk$-algebras $\t_n(\KK)$, defined likewise, is 
provided with the structure of an operad in the 
category $grLie_\KK$.
The center of $\t_3(\kk)$ is $c_3=t_{12}+t_{13}+t_{23}$, the quotient of $\t_3(\kk)$ by $c_3$ 
is the free Lie algebra $\f_2(\kk)$. Along this paper we consider the inclusion $\hat\f_2(\kk)\subset \hat\t_3(\kk)$ 
sending $x$ to $t_{12}$ and $y$ to $t_{23}$.

We then consider the operad of \textit{chord diagrams} $\CD(\KK) := 
\hat{\mathcal{U}}(\t({\kk}))$ in $\mathbf{Cat(CoAlg_\KK)}$. 
It has only one object in each arity. The terminal 
morphism of operads $\omega_1:\Pa=\on{Ob}(\Pa(\KK))\to\on{Ob}(\CD(\kk))$ 
allow us to consider the fake pull-back operad
$$
\PaCD(\kk):=\omega_1^\star \CD(\kk)
$$
of \textit{parenthesized chord diagrams}. 
More explicitely, we have $\on{Ob}(\PaCD(\KK)):=\Pa$ and for all 
$p,q\in \PaCD(\KK)(n)$,
$
\on{Mor}_{\PaCD(\KK)(n)}(p,q)
:=\on{Mor}_{\CD(\KK)(n)}(pt,pt)=\hat{\mathcal{U}}(\hat\t_n({\kk}))\,
$. 
As is shown in \cite[Theorem 10.3.4]{Fresse}, the operad ${\PaCD}(\kk)$ does not have a presentation in terms 
of generators and relations (as is the case for $\PaB$) but has, nevertheless, a universal property with respect to generators 
$H^{1,2}, X^{1,2}$ and $ a^{1,2,3}$ depicted as follows
\begin{center}
$H^{1,2}:=$
\begin{tikzpicture}[baseline=(current bounding box.center)]
\tikzstyle point=[circle, fill=black, inner sep=0.05cm]
 \node[point, label=above:$1$] at (1,1) {};
 \node[point, label=below:$1$] at (1,-0.25) {};
 \node[point, label=above:$2$] at (2,1) {};
 \node[point, label=below:$2$] at (2,-0.25) {};
 \draw[->,thick] (2,1) to (2,-0.20); 
  \draw[densely dotted, thick] (1,0.5) to (2,0.5); 
 \draw[->,thick] (1,1) to (1,-0.20);
\end{tikzpicture}
\qquad
$X^{1,2}=1 \cdot$ 
\begin{tikzpicture}[baseline=(current bounding box.center)]
\tikzstyle point=[circle, fill=black, inner sep=0.05cm]
 \node[point, label=above:$1$] at (1,1) {};
 \node[point, label=below:$2$] at (1,-0.25) {};
 \node[point, label=above:$2$] at (2,1) {};
 \node[point, label=below:$1$] at (2,-0.25) {};
 \draw[->,thick] (2,1) to (1,-0.20); 
 \draw[->,thick] (1,1) to (2,-0.20);
\end{tikzpicture}
\qquad
$a^{1,2,3}=1 \cdot$
\begin{tikzpicture}[baseline=(current bounding box.center)]
\tikzstyle point=[circle, fill=black, inner sep=0.05cm]
 \node[point, label=above:$(1$] at (1,1) {};
 \node[point, label=below:$1$] at (1,-0.25) {};
 \node[point, label=above:$2)$] at (1.5,1) {};
 \node[point, label=below:$(2$] at (3.5,-0.25) {};
 \node[point, label=above:$3$] at (4,1) {};
 \node[point, label=below:$3)$] at (4,-0.25) {};
 \draw[->,thick] (1,1) to (1,-0.20); 
 \draw[->,thick] (1.5,1) to (3.5,-0.20);
 \draw[->,thick] (4,1) to (4,-0.20);
\end{tikzpicture}
\end{center}
These elements satisfy the following relations:
\begin{itemize}
\item $X^{2,1}=(X^{1,2})^{-1}$,
\item $a^{12,3,4}a^{1,2,34} = a^{1,2,3}a^{1,23,4}a^{2,3,4}$,
\item $X^{12,3}=a^{1,2,3}X^{2,3}(a^{1,3,2})^{-1}X^{1,3}a^{3,1,2}$,
\item $H^{1,2}=X^{1,2}H^{2,1}(X^{1,2})^{-1}$,
\item $H^{12,3}=a^{1,2,3}\big(H^{2,3}+X^{2,3}(a^{1,3,2})^{-1}H^{1,3}a^{1,3,2}X^{3,2}\big)(a^{1,2,3})^{-1}$.
\end{itemize}
\begin{definition}
We call \textit{rational Drinfeld torsor over} $\kk$
the bi-torsor\linebreak $(\widehat{\GT}(\kk),\Assoc(\kk),\GRT(\kk))$ where
\begin{eqnarray}
\widehat{\GT}(\kk)&:=& \on{Aut}_{\on{Op}\mathbf{Grpd}_\kk}^{+}\big(\widehat{\PaB}(\kk)\big), \\
\Ass(\kk)& := & \on{Iso}^+_{\on{Op} \mathbf{Grpd}_{\kk}}(\widehat{\PaB}(\KK),G \PaCD(\KK)), \\
\GRT(\kk)&:=&\on{Aut}_{\on{Op} \mathbf{Grpd}_{\kk}}^+(G\PaCD(\KK)).
\end{eqnarray}
\end{definition}
There is a bi-torsor isomorphism
\begin{equation}\label{bitorsor:cl}
(\widehat{\GT}(\kk),\Assoc(\kk),\GRT(\kk)) \longrightarrow 
(\widehat{\on{GT}}(\kk),\on{Ass}(\kk),\on{GRT}(\kk)),
\end{equation}
where 
\begin{itemize}
\item  $\on{Ass}(\kk)$ is the set of couples 
$(\mu,\varphi)\in\kk^\times \times \on{exp}(\hat\f_2(\kk))$ 
such that 
\begin{itemize}
\item $\varphi^{3,2,1}=(\varphi^{1,2,3})^{-1}, \quad $ in $\on{exp}(\hat\t_{3}(\kk))$,
\item $\varphi^{1,2,3}e^{\mu t_{23}/2}\varphi^{2,3,1}e^{\mu t_{31}/2}\varphi^{3,1,2}e^{\mu t_{12}/2}
=e^{\mu(t_{12}+t_{13}+t_{23})/2}, \quad $ in $\on{exp}(\hat\t_{3}(\kk))$,
\item $\varphi^{1,2,3}\varphi^{1,23,4}\varphi^{2,3,4}=
\varphi^{12,3,4}\varphi^{1,2,34}, \quad$ in $\on{exp}(\hat\t_{4}(\kk))$.
\end{itemize}
\item $\widehat{\on{GT}}(\kk)$ is the group of pairs 
$
(\lambda,f) \in \KK^{\times} \times\widehat{\on{F}}_2(\KK)
$
which satisfy the following equations:
\begin{itemize}
\item $f(x,y)=f(y,x)^{-1}$, in $\widehat{\on{F}}_2(\KK)$, 
\item $x_1^{\nu}f(x_1,x_2)x_2^{\nu}f(x_2,x_3)x_3^{\nu}f(x_3,x_1)=1$, in 
$\widehat{\on{F}}_2(\KK) \qquad (x_1x_2x_3=1 \text{, } \nu=\frac{\lambda-1}{2})$, 
\item $f(x_{13}x_{23}, x_{34})f(x_{12}, x_{23}x_{24}) =f(x_{12}, x_{23}) 
f(x_{12}x_{13}, x_{23}x_{34})f(x_{23}, x_{34})$, in $\widehat{\on{PB}}_4(\KK)$,
\end{itemize}
with multiplication law given by
\begin{equation*}\label{GT:LCI}
(\lambda_1, f_1)(\lambda_2, f_2)=
(\lambda_1\lambda_2, 
 f_1(x^{\lambda_2},f_2(x, y)y^{\lambda_2} f_2(x, y)^{-1}) f_2(x, y)) .
\end{equation*}
\item $\on{GRT}(\kk):=\on{GRT}_1 \rtimes \kk^{\times}$ where $\on{GRT}_1$ is the 
group of elements $g\in\on{exp}(\hat\f_2(\kk))$ such that 
\begin{itemize}
\item $g^{3,2,1}=g^{-1}$ and $g^{1,2,3}g^{2,3,1}g^{3,1,2}=1$, in $\on{exp}(\hat\t_{3}(\kk))$,
\item $t_{12}+\on{Ad}(g^{1,2,3})(t_{23})+ \on{Ad}(g^{2,1,3})(t_{13}) = t_{12}+t_{13}+t_{23}$, in $\hat\t_{3}(\kk)$, 
\item $g^{1,2,3}g^{1,23,4}g^{2,3,4}=g^{12,3,4}g^{1,2,34}$, in $\on{exp}(\hat\t_{4}(\kk))$,
\end{itemize}
with multiplication law given by
$$
(g_1*g_2)(t_{12},t_{23})= g_1(t_{12},\on{Ad}(g_2(t_{12},t_{23}))(t_{23}))g_2(t_{12},t_{23})\,
$$
and where $\kk^\times$ acts on $\on{\GRT}_1$ by $\lambda\cdot g(x,y)=g(\lambda x,\lambda y)$.
\end{itemize}

\section{Operads associated to framed configuration spaces (framed associators)}
\label{framed associators}
\subsection{The operad of parenthesized framed braidings}

\subsubsection{Compactified framed configuration spaces of the plane}

To any finite set $I$ we associate  framed configuration space 
$$
\tmop{Conf}^f (\C,I) \assign \tmop{Conf} (\C, I) \times \tmop{SO} (2)^{\times I}.
$$
We also consider its reduced version 
$$
 \tmop{C}^f (\C,I) \assign \tmop{Conf}^f (\C,I)/\C \rtimes \RR_{>0}.
$$
The symmetric group $\mathfrak{S}_I$ acts on $\on{Conf}^f(\C,I)$ by relabelling 
the indices of the marked points and the map 
$\Conf^f(\C,[I]):=\Conf^f(\C,I)/\mathfrak{S}_I \to \on{Conf}(\C,[I])$ is a 
locally trivial bundle with fiber $\on{SO}(2)^{ I}$. 

We then consider the ASFM compactification $\overline{\on{C}}^f(\C,I)$ 
of the reduced framed 
configuration space $\tmop{C}^f (\C,I)$.
The boundary $\partial \overline{\on{C}}^f(\C,I) =\overline{\on{C}}^f(\C,I) - 
\tmop{C}^f (\C, I)$ of $\overline{\on{C}}^f(\C,I)$ is
made of the following irreducible components: for any decomposition $I = J_1  \coprod
\cdots \coprod J_k$ there is a component
\[ \partial_{J_1, \cdots, J_k} \overline{\on{C}}^f(\C,I) \cong \prod_{i =
   1}^k \overline{\on{C}}^f(\C,J_i) \times \overline{\on{C}}^f(\C,k)
   \hspace{0.17em} . \]
The collection of spaces $\overline{\on{C}}^f(\C,I)$ for all finite sets $I$ 
assemble into an $\mathfrak{S}$-module denoted $\overline{\on{C}}^f(\C,-)$, 
and the inclusion of boundary components  with respect to the direction of
the frame data provides $\overline{\on{C}}^f(\C,-)$ with the structure 
of an operad in topological spaces. 
This operad will be called the \textit{framed
ASFM operad}. It turns out to be weakly equivalent to the framed little $2$-disks 
operad. Partial operadic composition morphisms can be pictured as follows:
\begin{center}
\includegraphics[scale=1]{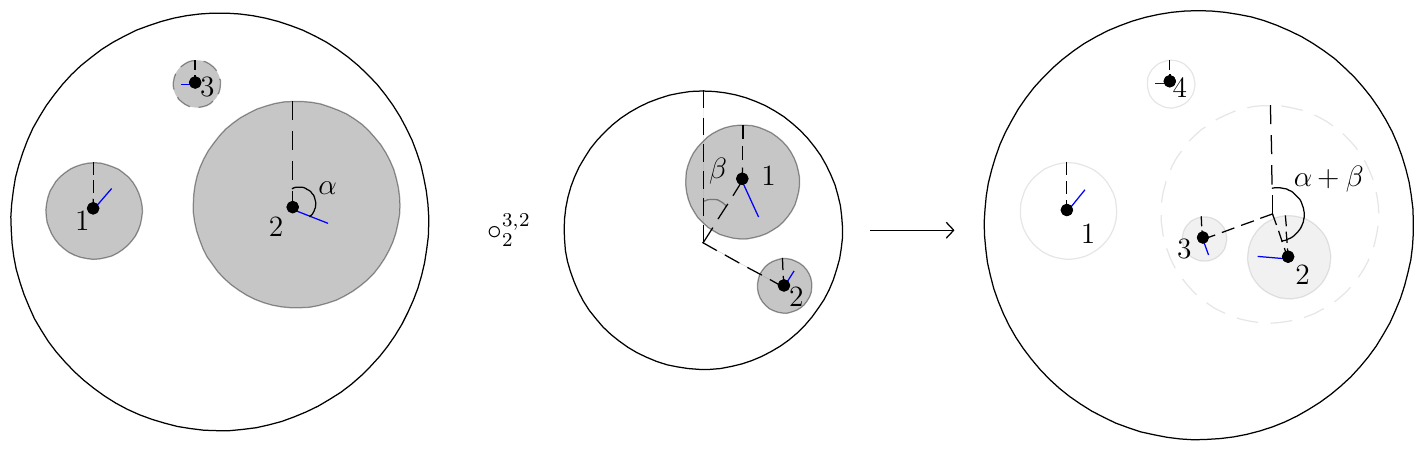}
\end{center}

\subsubsection{The operad of framed parenthesized braidings}

We have inclusions  of
topological operads
\[ \Pa \subset
\tmop{C} (\mathbbm{R}, -) \hspace{0.17em}
   \subset \hspace{0.17em} \overline{\on{C}}^f(\C,-) . \]
where the right inclusion is given by setting all framing elements pointing to the right.
Then it makes sense to define the operad in groupoids of 
\textit{framed parenthesized braidings}
  \[ \PaB^f \assign \pi_1 (
     \overline{\on{C}}^f(\C,-),  \Pa
     ). \]
\begin{example}[Description of $\PaB^f(1)$]
Recall that $ \overline{\on{C}}^f(\C,1)=\{(z_1,f_1)\}\simeq \mathbb{S}^1$. 
Besides the identity morphism in $\PaB^f(1)$, there is an element denoted $F^1\in \on{End}_{\PaB^f
(1)}(1)$ corresponding to the 360° clockwise twisting of the framing and which can be depicted as follows:
\begin{center}
\begin{tikzpicture}[baseline=(current bounding box.center)]
\tikzstyle point=[circle, fill=black, inner sep=0.05cm]
\tikzstyle point2=[circle, fill=black, inner sep=0.08cm]
\node[point, label=above:$1$] at (1,1) {};
\node[petitpoint, label=above:$\overrightarrow{1}$] at (1.5,1) {};
\node[point, label=below:$1$] at (1,0) {};
\node[petitpoint, label=below:$\overrightarrow{1}$] at (1.5,0) {};
\draw[-,thick] (1,1) .. controls (1,0.5) .. (1,0.5);
\node[point, ,white] at (1,0.7) {};
\draw[-,thin] (1.5,1) .. controls (1.5,0.75) and (0.75,0.75) .. (0.75,0.5); 
\draw[->,thin] (0.75,0.5) .. controls (0.75,0.35) and (1.5,0.35) .. (1.5,0.05);
\node[point, ,white] at (1,0.37) {};
\draw[->,thick] (1,0.5) .. controls (1,0.5) .. (1,0.05);
  \draw[-,thin] (1,1) .. controls (1,1) and (1.5,1) .. (1.5,1); 
    \draw[-,thin] (1,0) .. controls (1,0) and (1.5,0) .. (1.5,0); 
\end{tikzpicture}
\qquad\qquad\qquad
\begin{tikzpicture}[baseline=(current bounding box.center)]
\tikzstyle point=[circle, fill=black, inner sep=0.05cm]
 \node[point, label=left:$1$] at (0,0) {};
 \node[petitpoint, label=right:$\overrightarrow{1}$] at (0.5,0) {};
 \draw[->,thin] (0.5,0) .. controls (-1,-1) and (-1,1) .. (0.5,0); 
   \draw[-,thin] (0,0) .. controls (0,0) and (0.5,0) .. (0.5,0); 
\end{tikzpicture}
\\ \text{Two incarnations of $F^{1}$.}
\end{center}
\end{example}

\begin{example}[Description of $\PaB^f(2)$]\label{Ex1}
Recall that $ \overline{\on{C}}(\C,2)\simeq \mathbb{S}^1$. Then 
$ \overline{\on{C}}^f(\C,2)\simeq (\mathbb{S}^1)^3$. We have two 
arrows $F^{1,2}$ and $\tilde{F}^{1,2}$ in $\PaB^f(2)$ going from 
$(12)$ to $(12)$ which can be depicted as follows:
\begin{center}
\begin{tikzpicture}[baseline=(current bounding box.center)]
\tikzstyle point=[circle, fill=black, inner sep=0.05cm]
\tikzstyle point2=[circle, fill=black, inner sep=0.08cm]
\node[point, label=above:$1$] at (1,1) {};
\node[petitpoint, label=above:$\overrightarrow{1}$] at (1.5,1) {};
\node[point, label=below:$1$] at (1,0) {};
\node[petitpoint, label=below:$\overrightarrow{1}$] at (1.5,0) {};
\draw[-,thick] (1,1) .. controls (1,0.5) .. (1,0.5);
\node[point, ,white] at (1,0.7) {};
\draw[-,thin] (1.5,1) .. controls (1.5,0.75) and (0.75,0.75) .. (0.75,0.5); 
\draw[->,thin] (0.75,0.5) .. controls (0.75,0.35) and (1.5,0.35) .. (1.5,0.05);
\node[point, ,white] at (1,0.37) {};
\draw[->,thick] (1,0.5) .. controls (1,0.5) .. (1,0.05);
  \draw[-,thin] (1,1) .. controls (1,1) and (1.5,1) .. (1.5,1); 
    \draw[-,thin] (1,0) .. controls (1,0) and (1.5,0) .. (1.5,0); 
    \node[point, label=above:$2$] at (2.5,1) {};
    \node[point, label=below:$2$] at (2.5,0) {};
    \draw[->,thick] (2.5,1) .. controls (2.5,0.5) .. (2.5,0.05);
        \node[petitpoint, label=above:$\overrightarrow{2}$] at (3,1) {};
    \node[petitpoint, label=below:$\overrightarrow{2}$] at (3,0) {};
    \draw[->,thin] (3,1) .. controls (3,0.5) .. (3,0.05);
      \draw[-,thin] (2.5,1) .. controls (2.5,1) and (3,1) .. (3,1); 
       \draw[-,thin] (2.5,0) .. controls (2.5,0) and (3,0) .. (3,0); 
\end{tikzpicture}
\qquad  \qquad \qquad 
\begin{tikzpicture}[baseline=(current bounding box.center)]
\tikzstyle point=[circle, fill=black, inner sep=0.05cm]
\tikzstyle point2=[circle, fill=black, inner sep=0.08cm]
\node[point, label=above:$2$] at (1,1) {};
\node[petitpoint, label=above:$\overrightarrow{2}$] at (1.5,1) {};
\node[point, label=below:$2$] at (1,0) {};
\node[petitpoint, label=below:$\overrightarrow{2}$] at (1.5,0) {};
\draw[-,thick] (1,1) .. controls (1,0.5) .. (1,0.5);
\node[point, ,white] at (1,0.7) {};
\draw[-,thin] (1.5,1) .. controls (1.5,0.75) and (0.75,0.75) .. (0.75,0.5); 
\draw[->,thin] (0.75,0.5) .. controls (0.75,0.35) and (1.5,0.35) .. (1.5,0.05);
\node[point, ,white] at (1,0.37) {};
\draw[->,thick] (1,0.5) .. controls (1,0.5) .. (1,0.05);
  \draw[-,thin] (1,1) .. controls (1,1) and (1.5,1) .. (1.5,1); 
    \draw[-,thin] (1,0) .. controls (1,0) and (1.5,0) .. (1.5,0); 
    \node[point, label=above:$1$] at (-0.5,1) {};
    \node[point, label=below:$1$] at (-0.5,0) {};
    \draw[->,thick] (-0.5,1) .. controls (-0.5,0.5) .. (-0.5,0.05);
        \node[petitpoint, label=above:$\overrightarrow{1}$] at (0,1) {};
    \node[petitpoint, label=below:$\overrightarrow{1}$] at (0,0) {};
    \draw[->,thin] (0,1) .. controls (0,0.5) .. (0,0.05);
      \draw[-,thin] (-0.5,1) .. controls (-0.5,1) and (0,1) .. (0,1); 
       \draw[-,thin] (-0.5,0) .. controls (-0.5,0) and (0,0) .. (0,0); 
\end{tikzpicture}
\\ \text{The arrows $F^{1,2}$ (left) and $\tilde{F}^{1,2}$ (right).}
\end{center}
Next, the image of $F^1$ in $\PaB^f(2)$ will be denoted $F^{12}$. 
It can be depicted as follows:
\begin{center}
 \includegraphics[scale=1]{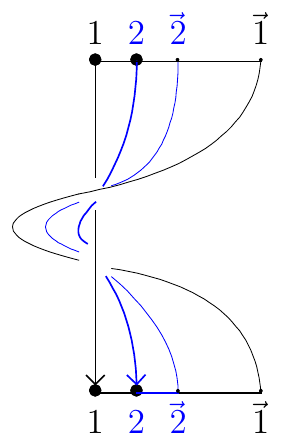}$\quad\qquad 
 \qquad$ \includegraphics[scale=1]{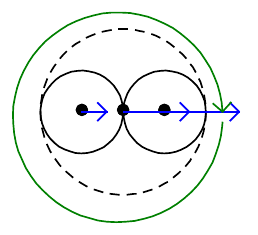}
\\ \text{Two incarnations of $F^{12}$.}
\end{center}
The element $F^{12}$ consits on a single ribbon, with second a ribbon glued along its surface, 
being twisted 360 degrees
and the blue strand is the transport of the glued ribbon lying in the surface of
this ribbon.

Finally, we have arrows $R^{1,2}$ and $\tilde{R}^{1,2}$ in $\mathbf{PaB}^f(2)$ 
going from $(12)$ to $(21)$ which can be depicted as follows:
\begin{center}
\includegraphics[scale=1]{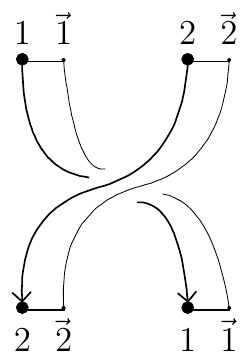} $\quad \qquad \qquad$ 
\includegraphics[scale=1]{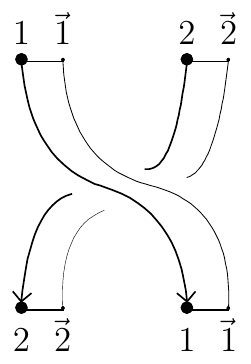}
\\ \text{The arrows $R^{1,2}$ (left) and $\tilde{R}^{1,2}$ (right).}
\end{center}
\end{example}

\begin{remark}
On can actually see that $\tilde{R}^{1,2}$, $F^{12}$ and $\tilde{F}^{1,2}$ 
are obtained from the arrows $F^{1,2}$ and $R^{1,2}$ via the following identities:
\begin{itemize}
\item $ \tilde{R}^{1,2}=(R^{2,1})^{-1}$,
\item $\tilde{F}^{1,2}=F^{12}(F^{1,2})^{-1}(R^{1,2}R^{2,1})^{-1}$,
\item $F^{12}=F^{1,\emptyset} \circ_1 \on{Id}^{1,2}=F^{12,\emptyset}$.
\end{itemize}
\end{remark}

\begin{example}[Notable arrows in $\PaB^f (3)$]
We have an arrow $\Phi^{1,2,3}$ from $(12)3$ to $1(23)$ in $\PaB^f(3)$. 
It can be depicted as follows: 

\begin{center}
\includegraphics[scale=1]{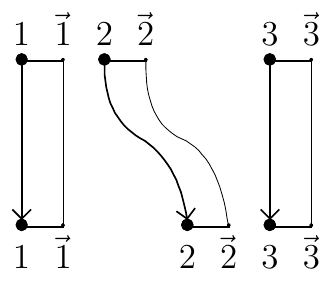} 
\\ \text{The arrow $\Phi^{1,2,3}$ in $\PaB^f (3)$.}
\end{center}
\end{example}

Recall the definition of the operad $\CoB$ of coloured braids from 
\cite[Subsection 5.2.8]{Fresse} and of its framed version $\CoB^f$ 
contained in \cite{Ho2}.
As in the case of the operad $\PaB$, the operad $\PaB^f$ can be
defined as the fake pullback of the framed version $\CoB^f$ of $\CoB$ 
along the operad map $\Pa \to \mathfrak{S}$ and we have a
presentation of $\PaB^f$ in terms of generators and
relations. Namely, we have the following theorem which is an 
straightforward corollary of \cite[Lemma 7.4]{Ho2}. 
\begin{theorem}\label{th:assf}
As a pointed operad in groupoids having $\Pa$ as operad of objects, $\PaB^f$ is freely generated by $F:=F^{1,2} \in 
  \textbf{\tmop{\PaB}}^f (2)$, $R:=R^{1, 2} \in \PaB^f
  (2)$ and $\Phi:=\Phi^{1, 2, 3} \in \PaB^f (3)$ together
  with the following relations:
  \begin{flalign}
  & F^{\emptyset,2}=\on{Id}^1 \quad 
\Big(\mathrm{in}~\mathrm{End}_{\mathbf{PaB}^f(1)}\big(1\big)\Big)\,, 		\tag{R1}\label{eqn:R1} \\
& \Phi^{\emptyset,2,3}=\Phi^{1,\emptyset,3}=\Phi^{1,2,\emptyset}=\on{Id}_{1,2}									\quad 
\Big(\mathrm{in}~\mathrm{Hom}_{\mathbf{PaB}^f(2)}\big(12,12\big)\Big)\,, 		\tag{R2}\label{eqn:R2} \\
& F^{1,2}R^{1,2}F^{2,1}R^{2,1}=F^{12} \quad 
\Big(\mathrm{in}~\mathrm{Hom}_{\mathbf{PaB}^f(2)}\big(12,12\big)\Big)\,,      \tag{F}\label{eqn:F} \\
& R^{1,2}\Phi^{2,1,3}R^{1,3}=\Phi^{1,2,3}R^{1,23}\Phi^{2,3,1} 											\quad 
\Big(\mathrm{in}~\mathrm{Hom}_{\mathbf{PaB}^f(3)}\big((12)3,2(31)\big)\Big)\,, 				\tag{H1}\label{eqn:H1} \\
& \tilde R^{1,2}\Phi^{2,1,3}\tilde R^{1,3}=\Phi^{1,2,3}\tilde R^{1,23}\Phi^{2,3,1} 			\quad 
\Big(\mathrm{in}~\mathrm{Hom}_{\mathbf{PaB}^f(3)}\big((12)3,2(31)\big)\Big)\,, 				\tag{H2}\label{eqn:H2} \\
& \Phi^{12,3,4}\Phi^{1,2,34}=\Phi^{1,2,3}\Phi^{1,23,4}\Phi^{2,3,4} 									\quad
\Big(\mathrm{in}~\mathrm{Hom}_{\mathbf{PaB}^f(4)}\big(((12)3)4,1(2(34))\big)\Big)\,. & 	\tag{P}\label{eqn:P} 
\end{flalign}
\end{theorem}

\begin{remark}
Combining relations \eqref{eqn:R1} and \eqref{eqn:F} we obtain $ F^{1,\emptyset}=F^1$.
\end{remark}

\subsubsection{The non-symmetric operad $\on{PB}^f$ of framed braidings}\label{PBf}

Let us now introduce two non-symmetric operads that will be of 
use in Lemma \ref{prop:gfra} and Theorem \ref{Thm:PaBfg}.

The collection $\on{PB}^f:=\{\on{PB}^f_{n}\}_{n\geq 0}$ can be 
endowed with the structure of a non-symmetric operad given by partial compositions 
\begin{eqnarray*}
\circ_i :  \on{PB}^f_{n} \times  \on{PB}^f_{m} & \longrightarrow &  \on{PB}^f_{n+m-1} \\
(b,b') & \longmapsto & b \circ_i b'
\end{eqnarray*}
where $b \circ_i b'$ is defined by replacing the $i$-labelled strand 
in $b$ by the braid $b'$ made very thin.
Via the homotopy equivalence between framed little disks and framed 
configuration spaces we presented in the last section, one checks that 
the above operadic composition for $\on{PB}^f$ is induced by that on $\on{D}_2^f$.

The fundamental group of the unordered framed configuration space 
$\Conf^f(\C,[n])$ was studied in \cite{KS} and is isomorphic to the 
\textit{framed} (also called \textit{ribbon}) \textit{braid group} $\on{B}^f_n$ 
generated by elements $\sigma_1,\sigma_2,\ldots,\sigma_{n-1},f_1,f_2,\ldots,f_n$ 
together with relations
\begin{flalign}
& \sigma_i \sigma_{i + 1} \sigma_i = \sigma_{i + 1} \sigma_i \sigma_{i + 1}\,,   
\quad \text{ for all } 1 \leq i \leq n-2,  \tag{B1} \label{eqn:B1} \\
& (\sigma_i, \sigma_j)= 1\,,      		\quad \text{ if } |i - j| > 1,  \tag{B2} \label{eqn:B2} \\
& f_if_j=f_jf_i\,,      		\quad \text{ for all }1 \leq i,j \leq n,  \tag{FB1} \label{eqn:FB1} \\
& \sigma_i f_j = f_{\sigma_i(j)}\sigma_i\,,       \quad \text{ for all }1 \leq i,j \leq n. &  \tag{FB2} \label{eqn:FB2} 
\end{flalign}

For convenience, we will rather think of the framed braid group $\on{B}^f_n$ 
as a subgroup of $\on{B}_{2n}$ with two generating elements $\tau_i$ and $f_i$ such that
\begin{enumerate}
\item $\tau_i=\sigma_{2i}\sigma_{2i-1}\sigma_{2i+1}\sigma_{2i}$
\item $f_i=\sigma_{2i-1}^{2}$
\end{enumerate}
Geometrically, if we denote $(z_1,\vec{z}_1,\ldots,z_n,\vec{z}_n)$ a point in 
$\Conf^f(\C,[n])$, where  then $\tau_i$ exchanges $(z_i,\vec{z}_i)$ and 
$(z_{i+1},\vec{z}_{i+1})$ in clockwise direction and $f_i$ makes a 360 degrees 
twist of $\vec{z}_i$ around $z_1$ in the clockwise direction.

The space $\Conf^f(\C,[n])$ is an Eilenberg--Maclane space of type 
$\on{K}(\on{B}^f_n,1)$ and the group $\on{B}^f_n$ is identified with the semidirect 
product $\Z^n \rtimes \B_n$ where the action of the braid group $\B_n$ on $\Z^n$ 
is given by $\sigma(r_i , ...., r_n) = (r_{\sigma(1)}, r_{\sigma(2)}, ... , r_{\sigma(n)})$. 
An element of $\on{B}_n^f$ is written as $f_1^{r_1},f_2^{r_2},\cdots, f_n^{r_n},\alpha\in \on{B}^f_n$ 
with $\alpha\in \on{B}_n$. The $r_i$'s are called \textit{framings}. 
The group composition law in this notation is given by
$$
(f_1^{r_1}f_2^{r_2}\cdots f_n^{r_n}\alpha)(f_1^{s_1}f_2^{s_2}\cdots f_n^{s_n}\beta)=
f_1^{r_1+s_{\alpha(1)}}f_2^{r_2+s_{\alpha(2)}}\cdots f_n^{r_n+s_{\alpha(n)}}\alpha\beta.
$$

Let $n\geq 0$, and $p$ be the object $(\cdots((12)3)\cdots\cdots)n$ of $\PaB^f(n)$. 
Then $\on{Aut}_{\mathbf{PaB}^f(n)}(p)$ identifies with the fundamental group 
$\on{PB}^f_{n}:=\pi_1\left(\overline{\textrm{C}}^f(\C,n),p\right)$, which is is isomorphic 
to the direct product $\on{PB}^f_{n}=\Z^n \times \on{PB}_n$.

In the same way, one can construct a non-symmetric operad in 
groupoids $\on{B}^f$ in the following way : 
\begin{itemize}
\item The objects of $\on{B}^f(n)$ are unnumbered maximal parenthesizations of 
lenght $n$. In particular, this means that for every object $p$ of $\mathbf{Pa}(n)$, 
there is a corresponding object $[p]$ in $\on{B}^f(n)$, and $[p]=[q]$ if $p$ and $q$ 
only differ by a permutation (but have the same underlying parenthesization). 
\item $\on{B}^f$ is freely generated by $F:=F^{\bullet, \bullet}\in 
  \on{B}^f (2)$, $R:=R^{\bullet, \bullet} \in \on{B}^f
  (2)$ and $\Phi:=\Phi^{\bullet, \bullet, \bullet} \in \on{B}^f (3)$ together
  with relations (H1), (H2), (P) and the following relation:
  \begin{flalign}
  & R^{\bullet,\bullet}R^{\bullet,\bullet}F^\bullet F^\bullet=F^{(\bullet\bullet)} \quad 
  \text{in } \on{End}_{\on{B}^f (2)}(\bullet\bullet), & \tag{F}
  \end{flalign}
  \item $\on{B}^f$ is the image of $\PaB^f$ via the forgetful map $\bf{Op} \to \bf{NsOp}$ 
  sending an operad to a non-symmetric operad.
\item The operad of coloured framed braids is nothing but 
$$
\CoB^f=\mathcal{G}(\on{B}^f \to \mathfrak{S}).
$$
\item It follows that there are group morphisms $\on{B}_n^f \tilde{\to} 
\on{Aut}_{\on{B}^f(n)}(p)\to\mathfrak S_n$, the left one being an isomorphism.
\end{itemize}

For example,  arrows in $\on{Aut}_{\on{B}^f(3)}((\bullet\bullet)\bullet)$  can be 
depicted as follows (we neglect for simplicity the framing data):
\begin{align}
\begin{tikzpicture}[baseline=(current bounding box.center)]
\tikzstyle point=[circle, fill=black, inner sep=0.05cm]
 \node[point, label=above:$(\bullet$] at (1,1) {};
 \node[point, label=below:$(\bullet$] at (1,-0.25) {};
 \node[point, label=above:$\bullet)$] at (2,1) {};
 \node[point, label=below:$\bullet)$] at (2,-0.25) {};
  \node[point, label=above:$\bullet$] at (3,1) {};
 \node[point, label=below:$\bullet$] at (3,-0.25) {};
 \draw[->,thick] (2,1) .. controls (1,0.25) and (1,0.5).. (2,-0.20); 
 \node[point, ,white] at (1.5,0.66) {};
 \draw[->,thick] (1,1) .. controls (2,0.25) and (2,0.5).. (1,-0.20);
   \node[point, ,white] at (1.5,0.15) {}; \draw[->,thick] (3,1) .. controls (3,0) and (3,0).. (3,-0.20);
    \draw[thick] (1.41,0.2) .. controls (1.41,0.2) and (1.41,0.2).. (2,-0.20); 
\end{tikzpicture}  ; \qquad 
\begin{tikzpicture}[baseline=(current bounding box.center)]
\tikzstyle point=[circle, fill=black, inner sep=0.05cm]
 \node[point, label=above:$(\bullet$] at (1,1) {};
 \node[point, label=below:$(\bullet$] at (1,-0.25) {};
 \node[point, label=above:$\bullet)$] at (2,1) {};
 \node[point, label=below:$\bullet)$] at (2,-0.25) {};
  \node[point, label=above:$\bullet$] at (3,1) {};
 \node[point, label=below:$\bullet$] at (3,-0.25) {};
  \draw[->,thick] (2,1) .. controls (2,0.25) and (1,0.5).. (1,-0.20); 
 \node[point, ,white] at (1.5,0.4) {};
 \draw[->,thick] (1,1) .. controls (1,0.25) and (2,0.5).. (2,-0.20);
   \node[point, ,white] at (1.5,0.15) {}; \draw[->,thick] (3,1) .. controls (3,0) and (3,0).. (3,-0.20);
\end{tikzpicture} 
\end{align}
We let the reader depict the generators $F\in 
  \on{B}^f (2)$, $R \in \on{B}^f
  (2)$ and $\Phi \in \on{B}^f (3)$ accordingly.

\subsection{Horizontal framed chord diagrams and rational framed associators}

\subsubsection{The operad of framed chord diagrams}

Let $\mathfrak{t}^f_{n}(\kk)$ denote the graded Lie algebra over $\kk$ generated by 
$t_{ij}$, $1\leq i, j\leq n$ with relations 
\begin{flalign}
& t_{ij}=t_{ji},\quad\textrm{ for  }1\leq i, j\leq n,     			                 			\tag{FS} \label{eqn:fS} \\
& [t_{ij},t_{kl}]=0,\quad\textrm{if }\{i,j\} \cap \{k,l\}=\emptyset \,,				\tag{FL} \label{eqn:fL} \\
& [t_{ij},t_{ik}+t_{jk}]=0,\quad\textrm{if }\{i,j\} \cap \{k\}=\emptyset\,. &	\tag{F4T} \label{eqn:f4T} 
\end{flalign}
It is easy to see that we have a decomposition 
$\mathfrak{t}^f_{n}(\kk)=\bigoplus_{i=1}^n \kk t_{ii} \oplus \t_n(\KK)$. 

\begin{remark}\label{remfra}
The above definition coincides with that appearing in \cite{BeG}, indeed it is 
isomorphic to the graded Lie algebra over $\kk$ generated by 
$t_{ij}$, $1\leq i\neq j\leq n$ and $t_k$, $1\leq k \leq n$, with relations 
\eqref{eqn:S}, \eqref{eqn:L}, \eqref{eqn:4T} and 
\begin{flalign}
& [t_{i},t_{j}]=0\quad\textrm{if }1\leq i, j\leq n,     			                 			\tag{FL'} \label{eqn:FL'} \\
& [t_{i},t_{jk}]=0\quad\textrm{if }1\leq i, j,k\leq n.    			   &              			\tag{FL''} \label{eqn:FL''}
\end{flalign}
\end{remark}

The Lie algebra $\t^f_n(\KK)$ is acted on by the symmetric group $\mathfrak{S}_n$, 
and one can show that the $\mathfrak{S}$-module in $grLie_\kk$ 
$$
\t^f({\KK}):=\{\t^f_n(\KK) \}_{n \geq 0}
$$
is provided with the structure of an operad in $grLie_\KK$. Partial compositions are defined as follows: 
for $I,J$ a finite sets and $k\in I$, 
 $$
  \begin{array}{cccccc}
\circ_k : & \t^f_{I}(\KK) \oplus \t^f_J(\KK)  & \longrightarrow & \t^f_{J\sqcup I-\{i\}}(\KK) \\
    & (0,t_{\alpha \beta}) & \longmapsto & t_{\alpha\beta} \\
 & (t_{ij},0) & \longmapsto & 
 \begin{cases}
  \begin{tabular}{ccccc}
  $t_{ij}$ & if & $ k\notin\{i,j\} $ \\
  $\sum\limits_{p\in J} t_{pj}$ & if & $k=i$ \\
  $\sum\limits_{p\in J} t_{ip}$ & if & $j=k$ 
  \end{tabular}
  \end{cases}
\end{array}
$$    
In particular, this has a translation into insertion-coproduct morphisms. 
We call $\t^f({\KK})$  the operad of \textit{infinitesimal framed braids}. 
We then consider the operad of \textit{framed chord diagrams}
$\mathbf{C}\mathbf{D}^f(\kk) \assign \hat{\mathcal{U}}
(\hat{\mathfrak{t}}^f (\kk))$ in $\mathbf{Cat(CoAlg_\KK)}$.

\begin{remark}
Morphisms in $\CD^f(\KK)(n)$ can be represented as linear combinations of diagrams 
of chords on $n$ vertical strands, where the chord diagram corresponding 
to $t_{ij}$ can be represented as in the unframed case, the chord corresponding 
to $t_{ii}$ as
\begin{align*}
\tik{  \node[point, label=above:$i$] at (0,-1) {}; \node[point, label=above:$n$] at (1,-1) {};
\node[point, label=above:$1$] at (-1,-1) {}; 
\node[point, label=below:$1$] at (-1,-2) {}; 
\node[point, label=below:$i$] at (0,-2) {}; \node[point, label=below:$n$] at (1,-2) {};
 \draw[->,thick] (-1,-1) to (-1,-2);  \draw[->,thick] (0,-1) to (0,-2); \draw[->,thick] (1,-1) to (1,-2);
  \draw[-,thick] (-0.1,-1.5) to (0.1,-1.5)  ;
  }
\end{align*}
and the composition is given by vertical concatenation of diagrams. 
Relations \eqref{eqn:S}, \eqref{eqn:L} and \eqref{eqn:4T} can be described as in 
the in the unframed case and the remaining relations defining each $\t^f_n(\KK)$
 can be represented as follows:
\begin{align}\tag{\ref{eqn:FL'}}
\tik{ 
\straight[->]{0}{1}; 
\straight[->]{1}{1};
\straight{0}{0};
\straight{1}{0};
 \draw[-,thick] (-0.1,-0.5) to (0.1,-0.5) ;
  \draw[-,thick] (0.9,-1.5) to (1.1,-1.5) ;
\node[point, label=above:$j$] at (1,0) {};
\node[point, label=above:$i$] at (0,0) {}; 
\node[point, label=below:$i$] at (0,-2) {}; 
\node[point, label=below:$j$] at (1,-2) {}; }
=
\tik{
\straight[->]{0}{1}; 
\straight[->]{1}{1};
\straight{0}{0};
\straight{1}{0} ;
\draw[-,thick] (-0.1,-1.5) to (0.1,-1.5) ;
  \draw[-,thick] (0.9,-0.5) to (1.1,-0.5) ;
\node[point, label=above:$j$] at (1,0) {}; 
\node[point, label=above:$i$] at (0,0) {}; 
\node[point, label=below:$i$] at (0,-2) {}; 
\node[point, label=below:$j$] at (1,-2) {}; }
\end{align}

\begin{align}\tag{\ref{eqn:FL''}}
\tik{ 
 \hori{0}{0}{1};
\straight[->]{0}{1}; 
\straight[->]{1}{1};
\straight{0}{0};
\straight{1}{0};
 \draw[-,thick] (-0.1,-1.5) to (0.1,-1.5) ;
\node[point, label=above:$j$] at (1,0) {};
\node[point, label=above:$i$] at (0,0) {}; 
\node[point, label=below:$i$] at (0,-2) {}; 
\node[point, label=below:$j$] at (1,-2) {}; }
=
\tik{\hori[->]{0}{1}{1};
\straight[->]{0}{1}; 
\straight[->]{1}{1};
\straight{0}{0};
\straight{1}{0} ;
\draw[-,thick] (-0.1,-0.5) to (0.1,-0.5) ;
\node[point, label=above:$j$] at (1,0) {}; 
\node[point, label=above:$i$] at (0,0) {}; 
\node[point, label=below:$i$] at (0,-2) {}; 
\node[point, label=below:$j$] at (1,-2) {}; }
\qquad ; \qquad 
\tik{ \hori{1}{0}{1};
\straight[->]{0}{1}; 
\straight[->]{1}{1};
\straight[->]{2}{1};
\straight{0}{0};
 \draw[-,thick] (-0.1,-1.5) to (0.1,-1.5) ;
\node[point, label=above:$j$] at (1,0) {}; 
\node[point, label=above:$k$] at (2,0) {};
\node[point, label=above:$i$] at (0,0) {}; 
\node[point, label=below:$i$] at (0,-2) {}; 
\node[point, label=below:$j$] at (1,-2) {}; 
\node[point, label=below:$k$] at (2,-2) {};}
=
\tik{
\hori[->]{1}{1}{1};
\straight{0}{0};
\straight{1}{0};
\straight{2}{0};
\straight[->]{0}{1};
 \draw[-,thick] (-0.1,-0.5) to (0.1,-0.5) ;
\node[point, label=above:$j$] at (1,0) {}; 
\node[point, label=above:$k$] at (2,0) {};
\node[point, label=above:$i$] at (0,0) {}; 
\node[point, label=below:$i$] at (0,-2) {}; 
\node[point, label=below:$j$] at (1,-2) {}; 
\node[point, label=below:$k$] at (2,-2) {};}
\end{align}
\end{remark}

\subsubsection{The operad $\PaCD^f(\kk)$ of parenthesized framed chord diagrams}\label{sec-pacdf}

As the operad $\CD^f(\KK)$ has only one object in each arity, we have an obvious 
terminal morphism of operads $\omega_1:\Pa=\on{Ob}(\Pa(\KK))\to\on{Ob}(\CD^f(\kk))$, 
and thus we can consider the operad 
$$
\PaCD^f(\KK):=\omega_1^\star \CD^f(\kk)
$$
in $\mathbf{Cat(CoAss_\KK)}$ of \textit{parenthesized framed chord diagrams}. 
More explicitely we have:
\begin{itemize}
\item $\on{Ob}(\PaCD^f(\KK)):=\Pa$,
\item $\on{Mor}_{\PaCD^f(\KK)(n)}(p,q):=\CD^f(\kk)$.
\end{itemize}

\begin{example}[Notable arrows in $\mathbf{PaCD}^f(\kk)$]
We have the following arrow $P^1$ in $\mathbf{PaCD}^f(\kk)(1)$ 
\begin{center}
$P^1=t_{11} \cdot$
\begin{tikzpicture}[baseline=(current bounding box.center)]
\tikzstyle point=[circle, fill=black, inner sep=0.05cm]
 \node[point, label=above:$1$] at (1,1) {};
 \node[point, label=below:$1$] at (1,-0.25) {};
 \draw[->,thick] (1,1) .. controls (1,0) and (1,0).. (1,-0.20); 
\end{tikzpicture}$:=$
\begin{tikzpicture}[baseline=(current bounding box.center)]
\tikzstyle point=[circle, fill=black, inner sep=0.05cm]
 \node[point, label=above:$1$] at (1,1) {};
 \node[point, label=below:$1$] at (1,-0.25) {};
 \draw[->,thick] (1,1) .. controls (1,0) and (1,0).. (1,-0.20); 
   \draw[-,thick] (0.9,0.5) to (1.1,0.5)  ;
\end{tikzpicture}
\end{center} 
as well as the following arrows in $\mathbf{PaCD}^f(\kk)(2)$ and $\mathbf{PaCD}^f(\kk)(3)$
\begin{center}
$P^{1,2}:=t_{11} \cdot$
\begin{tikzpicture}[baseline=(current bounding box.center)]
\tikzstyle point=[circle, fill=black, inner sep=0.05cm]
 \node[point, label=above:$1$] at (1,1) {};
 \node[point, label=below:$1$] at (1,-0.25) {};
 \node[point, label=above:$2$] at (2,1) {};
 \node[point, label=below:$2$] at (2,-0.25) {};
 \draw[->,thick] (2,1) to (2,-0.20); 
 \draw[->,thick] (1,1) to (1,-0.20);
\end{tikzpicture}$:= $
\begin{tikzpicture}[baseline=(current bounding box.center)]
\tikzstyle point=[circle, fill=black, inner sep=0.05cm]
 \node[point, label=above:$1$] at (1,1) {};
 \node[point, label=below:$1$] at (1,-0.25) {};
 \node[point, label=above:$2$] at (2,1) {};
 \node[point, label=below:$2$] at (2,-0.25) {};
 \draw[->,thick] (2,1) to (2,-0.20); 
 \draw[->,thick] (1,1) to (1,-0.20);
   \draw[-,thick] (0.9,0.5) to (1.1,0.5)  ;
\end{tikzpicture}
 \qquad 
$H^{1,2}:=t_{12} \cdot$
\begin{tikzpicture}[baseline=(current bounding box.center)]
\tikzstyle point=[circle, fill=black, inner sep=0.05cm]
 \node[point, label=above:$1$] at (1,1) {};
 \node[point, label=below:$1$] at (1,-0.25) {};
 \node[point, label=above:$2$] at (2,1) {};
 \node[point, label=below:$2$] at (2,-0.25) {};
 \draw[->,thick] (2,1) to (2,-0.20); 
 \draw[->,thick] (1,1) to (1,-0.20);
\end{tikzpicture} =
\begin{tikzpicture}[baseline=(current bounding box.center)]
\tikzstyle point=[circle, fill=black, inner sep=0.05cm]
 \node[point, label=above:$1$] at (1,1) {};
 \node[point, label=below:$1$] at (1,-0.25) {};
 \node[point, label=above:$2$] at (2,1) {};
 \node[point, label=below:$2$] at (2,-0.25) {};
 \draw[->,thick] (2,1) to (2,-0.20); 
  \draw[densely dotted, thick] (1,0.5) to (2,0.5); 
 \draw[->,thick] (1,1) to (1,-0.20);
\end{tikzpicture}
\\
$X^{1,2}=1 \cdot$ 
\begin{tikzpicture}[baseline=(current bounding box.center)]
\tikzstyle point=[circle, fill=black, inner sep=0.05cm]
 \node[point, label=above:$1$] at (1,1) {};
 \node[point, label=below:$2$] at (1,-0.25) {};
 \node[point, label=above:$2$] at (2,1) {};
 \node[point, label=below:$1$] at (2,-0.25) {};
 \draw[->,thick] (2,1) to (1,-0.20); 
 \draw[->,thick] (1,1) to (2,-0.20);
\end{tikzpicture}
\quad 
$a^{1,2,3}=1 \cdot$
\begin{tikzpicture}[baseline=(current bounding box.center)]
\tikzstyle point=[circle, fill=black, inner sep=0.05cm]
 \node[point, label=above:$(1$] at (1,1) {};
 \node[point, label=below:$1$] at (1,-0.25) {};
 \node[point, label=above:$2)$] at (1.5,1) {};
 \node[point, label=below:$(2$] at (3.5,-0.25) {};
 \node[point, label=above:$3$] at (4,1) {};
 \node[point, label=below:$3)$] at (4,-0.25) {};
 \draw[->,thick] (1,1) to (1,-0.20); 
 \draw[->,thick] (1.5,1) to (3.5,-0.20);
 \draw[->,thick] (4,1) to (4,-0.20);
\end{tikzpicture} 
\\
$\tilde{P}^{1,2}:=t_{22} \cdot$
\begin{tikzpicture}[baseline=(current bounding box.center)]
\tikzstyle point=[circle, fill=black, inner sep=0.05cm]
 \node[point, label=above:$1$] at (1,1) {};
 \node[point, label=below:$1$] at (1,-0.25) {};
 \node[point, label=above:$2$] at (2,1) {};
 \node[point, label=below:$2$] at (2,-0.25) {};
 \draw[->,thick] (2,1) to (2,-0.20); 
 \draw[->,thick] (1,1) to (1,-0.20);
\end{tikzpicture}$:=$
\begin{tikzpicture}[baseline=(current bounding box.center)]
\tikzstyle point=[circle, fill=black, inner sep=0.05cm]
 \node[point, label=above:$1$] at (1,1) {};
 \node[point, label=below:$1$] at (1,-0.25) {};
 \node[point, label=above:$2$] at (2,1) {};
 \node[point, label=below:$2$] at (2,-0.25) {};
 \draw[->,thick] (2,1) to (2,-0.20); 
 \draw[->,thick] (1,1) to (1,-0.20);
    \draw[-,thick] (0.9+1,0.5) to (1.1+1,0.5)  ;
\end{tikzpicture}
 \qquad 
$P^{12}:=(t_{11}+t_{22}) \cdot$
\begin{tikzpicture}[baseline=(current bounding box.center)]
\tikzstyle point=[circle, fill=black, inner sep=0.05cm]
 \node[point, label=above:$1$] at (1,1) {};
 \node[point, label=below:$1$] at (1,-0.25) {};
 \node[point, label=above:$2$] at (2,1) {};
 \node[point, label=below:$2$] at (2,-0.25) {};
 \draw[->,thick] (2,1) to (2,-0.20); 
 \draw[->,thick] (1,1) to (1,-0.20);
\end{tikzpicture} $:=$
\begin{tikzpicture}[baseline=(current bounding box.center)]
\tikzstyle point=[circle, fill=black, inner sep=0.05cm]
 \node[point, label=above:$1$] at (1,1) {};
 \node[point, label=below:$1$] at (1,-0.25) {};
 \node[point, label=above:$2$] at (2,1) {};
 \node[point, label=below:$2$] at (2,-0.25) {};
 \draw[->,thick] (2,1) to (2,-0.20); 
 \draw[->,thick] (1,1) to (1,-0.20);
    \draw[-,thick] (0.9,0.5) to (1.1,0.5)  ;
       \draw[-,thick] (0.9+1,0.5) to (1.1+1,0.5)  ;
\end{tikzpicture} 
\end{center}
\end{example}

\begin{remark}
The elements $a^{1,2,3}$, $X^{1,2}$, $H^{1,2}$ and $P^{1,2}$ are generators of 
$\mathbf{PaCD}^f(\kk)$ and satisfy the following relations:
\begin{itemize}
\item $X^{2,1}=(X^{1,2})^{-1}$,
\item $ \tilde{P}^{1,2}P^{1,2}=P^{12} $,
\item $P^{12}=P^{1,\emptyset} \circ_1 \on{Id}^{1,2}$,
\item $a^{12,3,4}a^{1,2,34} = a^{1,2,3}a^{1,23,4}a^{2,3,4}$,
\item $X^{12,3}=a^{1,2,3}X^{2,3}(a^{1,3,2})^{-1}X^{1,3}a^{3,1,2}$,
\item $H^{1,2}=X^{1,2}H^{2,1}(X^{1,2})^{-1}$,
\item $H^{12,3}=a^{1,2,3}H^{2,3}(a^{1,2,3})^{-1} + 
(X^{2,1})^{-1}a^{2,1,3}H^{1,3}(a^{2,1,3})^{-1}X^{2,1}$,
\end{itemize}
\end{remark}

\subsubsection{Rational framed associators}

\begin{definition}
A framed $\kk$-associator is an isomorphism between the operads 
$\widehat{\PaB}^f(\KK)$ and $G \PaCD^f(\KK)$ in 
$ \mathbf{Grpd}_{\kk}$ which is the identity on objects. We denote
$$
\Ass^f(\kk):=\on{Iso}^+_{\on{Op} \mathbf{Grpd}_{\kk}}(\widehat{\PaB}^f(\KK),G \PaCD^f(\KK))
$$
 the set of framed $\kk$-associators.
\end{definition}

\begin{proposition}
There is a one-to-one correspondence between the set of framed $\kk$-associators 
$\Ass^f(\kk)$ and the set $\on{Ass}(\kk)$ of $\kk$-associators.
\end{proposition}

\begin{proof}

A morphism $\tilde{H}:\widehat{\PaB}^f(\KK) \longrightarrow G \PaCD^f(\KK)$ is uniquely determined 
by a morphism $H:{\PaB}^f\longrightarrow G \PaCD^f(\KK)$. Such a morphism is uniquely determined 
by two scalar parameters $\mu,\lambda\in\kk$ and $\varphi\in\on{exp}(\hat\t^f_2(\kk))$ 
such that we have the following assignment in the morphism sets of the parenthesized 
chord diagram operad $G\PaCD^f(\kk)$:
\begin{itemize}
\item $H(F^{1,2})=e^{\lambda t_{1}} \cdot \on{Id}^{1,2}$,
\item $H(R^{1,2})=e^{\mu t_{12}/2} \cdot X^{1,2}$,
\item $H(\Phi^{1,2,3})=\varphi \cdot a^{1,2,3}$\,,
\end{itemize}
where $F^{1,2},R^{1,2}$ and $\Phi^{1,2,3}$ are the generators of ${\PaB}^f$.
The triples$(\lambda, \mu,\varphi)$ then satisfy 
\begin{itemize}
\item $(\mu,\varphi)\in\on{Ass(\kk)}$,
\item $e^{\lambda (t_1 + t_2 + 2t_{12})} = e^{\lambda (t_1+ t_2) + \mu t_{12} }$.
\end{itemize}
From the last equation one can easily deduce (by using the map $\t^f_2 \to \t_2$ 
sending the $t_i$ to 0) that $2\mu=\lambda$, which in turn implies that condition 
$e^{\lambda (t_1 + t_2 + 2t_{12})} = e^{\lambda (t_1+ t_2 +2t_{12}) }$ is trivially satisfied as the $t_{i}$ are central. 
This finishes the proof.
\end{proof}

\begin{theorem}\label{fKZ}
The set $\Ass^f(\C)$ is non empty. 
\end{theorem}

We will prove this statement in the following subsection by using the 
regularized monodromy of a framed version of the universal KZ connection.

\subsubsection{The framed universal KZ connection}

We use the conventions for principal bundles and monodromy actions from 
\cite[Appendix A]{CG}.  Define the framed universal KZ connection on the trivial $\exp
(\hat{\mathfrak{t}}^f_n)$-principal bundle over $\tmop{Conf}^f (\C, n)$
as the connection given by the holomorphic 1-form
\[ w^{f\tmop{KZ}}_n : = \underset{1 \leqslant i  \leqslant n}{\sum} 
t_{i i}\on{d}log(\lambda_i) + \underset{1 \leqslant i < j \leqslant n}{\sum}
   \frac{d z_i - d z_j}{z_i - z_j} t_{i j} \in \Omega^1 (\tmop{Conf}^f
   (\C, n), \mathfrak{t}^f_n), \]
which takes its values in $\mathfrak{t}^f_n$ and where the $\lambda_i\in \C^\times$, for all $1 \leq i \leq n$, are the
fiber coordinates.

\begin{proposition}
The connection $\nabla^{f\tmop{KZ}}_n:= \on{d} - w^{f\tmop{KZ}}_n$ is flat.
\end{proposition}

\begin{proof}
Let $w_1:= \underset{1 \leqslant i  \leqslant n}{\sum} t_i \on{d}log(\lambda_i)$ and 
$w_2:=\underset{1 \leqslant i < j \leqslant n}{\sum}
   \frac{d z_i - d z_j}{z_i - z_j} t_{i j}$.
We want to show that $[w_1+w_2,w_1+w_2]=0$.
We have 
\begin{align*}
[w_1+w_2,w_1+w_2] & =  [w_1,w_1]+[w_2,w_2]+[w_1,w_2]+[w_2,w_1] \\
& =  2[w_1,w_2]
\end{align*}
since $[w_1,w_1]=0$ because the relation (FT1), $[w_2,w_2]=0$ because of 
flatness of the unframed KZ connection, and $[w_2,w_1]+[w_2,w_1]=2[w_1,w_2]$. 
Next, because of relation (FT2), we have 
$$
[w_1,w_2]= [t_i \on{d}log(\lambda_i),\frac{d z_i - d z_j}{z_i - z_j} t_{i j}]+
\underset{1 \leqslant i < j \leqslant n}{\sum}
   [t_j \on{d}log(\lambda_i),\frac{d z_i - d z_j}{z_i - z_j} t_{i j}].
$$
And finally,
$$
\underset{1 \leqslant i < j \leqslant n}{\sum}
   [t_i \on{d}log(\lambda_i),\frac{d z_i - d z_j}{z_i - z_j} t_{i j}]+
   \underset{1 \leqslant i < j \leqslant n}{\sum}
   [t_j \on{d}log(\lambda_i),\frac{d z_i - d z_j}{z_i - z_j} t_{i j}] =0.
$$
This concludes the proof.
\end{proof}
In particular, we get morphism of 
splitting short exact sequences
\begin{equation}
\xymatrix{
1 \ar[r] & \kk^n \ar[r]\ar[d] & \widehat{\on{PB}}^f_{n}(\kk) \ar[d] \ar[r] 
& \widehat{\on{PB}}_{n}(\kk) \ar[d] \ar[r] & 1 \\
1 \ar[r] & \kk^n \ar[r] & \on{exp}(\hat{\t}^f_{n}(\kk)) \ar[r] 
& \on{exp}(\hat{\t}_{n}(\kk))  \ar[r] & 1
}
\end{equation}
showing that $ \widehat{\on{PB}}^f_{n}(\kk) \to \on{exp}(\hat{\t}^f_{n}(\kk))$ is a 
$\kk$-pro-unipotent group isomorphism. Similarly we get an isomorphism 
$$ \widehat{\on{B}}^f_{n}(\kk) \to \on{exp}(\hat{\t}^f_{n}(\kk))\rtimes \mathfrak{S}_n.
$$

\begin{proof}[Proof of Theorem \ref{fKZ}]
Let $x \in \tmop{Conf}^f (\mathbbm{C}, n)$ and let $T_x^{f, \tmop{KZ}}$ be the
parallel transport morphism associated to $\omega_{f, n}^{\tmop{KZ}}$. Then
\[ T_x^{f, \tmop{KZ}} (\lambda_i) = e^{2 \i \pi t_{ii}} \in \exp
   (\hat{\mathfrak{t}}_n^f)  \]
so that $T_x^{f, \tmop{KZ}} (f_i)=(T_x^{f, \tmop{KZ}} (\sigma_i))^2$.
\end{proof}

\subsection{The group $\GT^f$ and homotopy theory of the framed little disks operad}

\begin{definition}
The \textit{framed Grothendieck--Teichm\"uller group} is defined as the group 
$$
\GT^f:= \on{Aut}_{\on{Op} \mathbf{Grpd}}^{+}(\PaB^f)
$$
of automorphisms of the operad in groupoids $\PaB^f$ which are the identity of objects. 
One defines similarly the $\kk$-pro-unipotent version 
$$
\widehat{\GT}^f(\kk):= \on{Aut}_{\on{Op}\mathbf{Grpd}_\kk}^{+}\big(\widehat{\PaB}^f(\kk)\big)
$$
There are also pro-$\ell$ and profinite versions, denoted $\GT^f_{\ell}$ and $\widehat \GT^f$ 
respectively, defined by replacing the $\kk$-pro-unipotent completion of $\PaB$ by its 
pro-$\ell$ and profinite completions. 
\end{definition}

\begin{definition}
The \textit{graded framed Grothendieck--Teichm\"uller group} is the group 
$$
\GRT^f(\kk):=\on{Aut}_{\on{Op} \mathbf{Grpd}_{\kk}}^+(G\PaCD^f(\KK))
$$ 
of automorphisms of $G\PaCD^f(\KK)$ that are the identity on objects.
\end{definition}

By \cite[Lemma 7.7]{Ho2}, there is a group isomorphism
$$
\widehat{\GT}(\kk) \simeq \widehat{\GT}^f(\kk):=\on{Aut}_{\on{Op} 
\mathbf{Grpd}_{\kk}}^{+}(\widehat{\PaB}^f(\kk))
$$
and the fact that $\mathfrak{t}^f_{n}(\kk)=\bigoplus_{i=1}^n \kk t_i \oplus \t_n(\KK)$ 
gives us a further isomorphism
$$
\GRT(\kk) \simeq \GRT^f(\kk):=\on{Aut}_{\on{Op} \mathbf{Grpd}_{\kk}}^{+}(G\PaCD^f(\kk)).
$$

The above results permit to extend the results in \cite{Fresse2} in the following manner.

Consider the diagram
$$
\xymatrix{
\on{D}^f_2(n) \ar[r]^-{\simeq} \ar[d] & \Conf^f({\mathbb R}^2,n) \ar[d]  & 
\ar[l]_-{\simeq}\overline{\on{C}}^f(\RR^2,n) \ar[d] \\
\on{D}_2(n) \ar[r]^-{\simeq} & \on{Conf}({\mathbb R}^2,n)  & \ar[l]_-{\simeq} 
\overline{\on{C}}(\RR^2,n)
}
$$
where the horizontal arrows are $\mathfrak{S}_n$-equivariant homotopy equivalences
and the vertical arrows are $\tmop{SO} (2)^{\times n}$-principal bundles. 
This diagram does not enhance into an operad map. Nevertheless, in \cite{EH}, an 
operad morphism $\phi : \overline{\on{C}}(\RR^2,-) \longrightarrow \on{D}_2$ 
was constructed and it is easy to verify that $\phi $ is equivariant for the action of 
$\on{SO}(2)$ on these two operads and by construction, the data of the framings 
are compatible with this map (since the rotation of a disk will preserve that disk). 
Thus, we can construct a square
\begin{equation}\label{unoff}
\xymatrix{
\on{D}^f_2 \ar[d] & \ar[l]_-{\simeq}\overline{\on{C}}^f(\RR^2,-) \ar[d] \\
\on{D}_2  & \ar[l]_-{\simeq} \overline{\on{C}}(\RR^2,-) 
}
\end{equation}
where the horizontal arrows are weak equivalences of operads in topological
spaces (see \cite{EH} for details).

On the one hand, let $\on{C}^{\ast}_{\tmop{CE}} (\mathfrak{t}^f_{ n})$ be
the Chevalley-Eilenberg cochain complex of $\mathfrak{t}^f_{n}$. It is a
quasi-free commutative dg-algebra generated by the module $(\mathfrak{t}^f_{n})^{\vee}$ 
in degree 1. Now $\mathfrak{t}^f_{n}$ is equipped with a weight grading such 
that each $t_{ij}, 1\leq i,j\leq n$ is homogeneous of weight 1.

On the other hand, let $\Omega^{\ast} (\tmop{Conf}^f (\C, n))$ be
the de-Rham complex of $\tmop{Conf}^f (\C, n)$.
Now, as $\on{Conf}^f(\C,n)$ is a trivial $\on{SO}(2)^{\times n}$-principal bundle over $\on{Conf}(\C,n)$, 
then using the Kunneth formula and the cohomology of $\mathbb{S}^1$, one can then 
show that the dg-algebra quasi-isomorphism
$\kappa_n : \on{C}^*_{\on{CE}}(\t_n)  \longrightarrow  \on{H}^*(\on{C}(\C,n))$
extends into a dg-algebra quasi-isomorphism
\begin{eqnarray*}
\kappa^f_n : \on{C}^*_{\on{CE}}(\t_n^f) & \longrightarrow & \on{H}^*(\on{C}^f(\C,n)) \\
 t_{ij}^{\vee} & \longmapsto & [\omega_{ij}]:=[\on{d}log(z_i-z_j)] \\
  t_{ii}^{\vee} & \longmapsto & [\omega_{i}]:=[\on{d}log(\lambda_i)],
\end{eqnarray*}
and $\kappa^f_n(\alpha^\vee)=0$ when $\alpha$ has weight $>1$.
The collection of cochain complexes $\{\on{C}^*_{\on{CE}}(\t_n^f)\}_{n \geq 0}$ 
forms a Hopf dg-cooperad, with coproducts induced by the operadic compositions 
of $\mathfrak{t}^f$ and, as $\on{H}^*$ is lax monoidal, the collection 
$\{\on{H}^*(\on{C}^f(\C,n))\}_{n \geq 0}$ inherits a Hopf dg-cooperads structure.

One can show that the collection of quasi-isomorphisms $\kappa^f_n$ is 
compatible with cooperadic partial compositions and can be promoted into a 
quasi-isomorphism of Hopf dg-cooperads 
$$
\kappa^f_n : \on{C}^*_{\on{CE}}(\t^f)  \longrightarrow  \on{H}^*(\on{D}_2^f).
$$

Now let $G$ be the left adjoint functor of Sullivan's functor $A$ of piecewise linear differential forms, 
let $G_{\bullet}$ be its operadic enhancement (taking arguments in Hopf dg-cooperads)  and 
let $ L G_{\bullet}$ denote the derived functor of $G_{\bullet}$.

Then, the definition of the derived functor $L G_{\bullet}$, Maurer-Cartan theory, 
the construction of framed associators and the weak-equivalence 
$B(\PaB^f)\simeq \on{D}_2^f$ induce a sequence of operad morphisms
\begin{equation}\label{eqn:rat}
L G_{\bullet}\on{H}^*(\on{D}_2^f)=G_{\bullet}\on{C}^*_{\on{CE}}(\t^f)
\simeq B(G\CD^f) \leftarrow B(\PaB^f) \simeq \on{D}_2^f.
\end{equation}
Thus, as $\Assoc^f(\Q)\neq \emptyset$, $L G_{\bullet}\on{H}^*(\on{D}_2^f)$ represents a 
rationalization of the framed little 2-disc operad $\on{D}_2^f$.
Next, Fresse introduced an operadic replacement $A_\sharp$ of Sullivan's functor, 
showed that the the couple $(G_\bullet,A_\sharp)$ is a Quillen pair and showed that 
if the components of an operad $\mathcal{O}$ in simplicial sets have a degree-wise 
finitely generated cohomology, then we have a weak equivalence 
$A_\sharp(\mathcal{O})(n)\simeq A(\mathcal{O}(n))$ for each arity $n$ so that the assignment 
$\mathcal{O} \to \mathcal{O}_{\mathbbm{Q}}^{\wedge}:=L G_{\bullet}A(\mathcal{O})$ 
is equivalent to Sullivan's rationalization of $\mathcal{O}(n)$ arity-wise. 

Then equation \eqref{eqn:rat} induces a rational weak equivalence 
\begin{equation}
(\on{D}^f_{2})_{\mathbbm{Q}}^{\wedge} \longrightarrow L G_{\bullet} 
\on{H}^{\ast} (\on{D}^f_{ 2}),
\end{equation}
which in turn induces a weak-equivalence of Hopf dg-cooperads 
$A_\sharp(\on{D}_2^f) \simeq \on{H}^\ast(\on{D}_2^f)$.

Let $\tmop{Ho} (\on{OpTop})$ be the homotopy category of the category of operads 
in topological spaces. We can then sumarize the results of this section as follows.
\begin{theorem}
There is a torsor isomorphism 
 \begin{equation}\label{bitorsor:cl:fr}
(\widehat{\GT}^f(\mathbbm{Q}),\Assoc^f(\mathbbm{Q}),\GRT^f(\mathbbm{Q})) 
\longrightarrow (\widehat{\GT}(\mathbbm{Q}),\Assoc(\mathbbm{Q}),\GRT(\mathbbm{Q}))
\end{equation}
 and the following maps are bijections
   \begin{eqnarray}\label{Hom:fr:1}
   \Assoc^f(\mathbbm{Q}) & \longrightarrow & \tmop{Iso}_{\tmop{Ho}
    (\on{OpTop})} ((\on{D}^f_{2})_{\mathbbm{Q}}^{\wedge}, L G_{\bullet}
    \on{H}^{\ast} (\on{D}^f_{ 2})),
  \end{eqnarray}
   \begin{eqnarray}\label{Hom:fr:2}
    \widehat{\GT}^f(\mathbbm{Q}) & \longrightarrow &
    \tmop{Aut}_{\tmop{Ho} (\on{OpTop})} ((\on{D}^f_{
    2})_{\mathbbm{Q}}^{\wedge})
  \end{eqnarray}
\end{theorem}

\begin{proof}
The fact that the map \eqref{bitorsor:cl:fr} is a torsor isomorphism is a 
straigthforward consequence of the fact that the set of complex associators 
is not empty, fact proven in Theorem \ref{fKZ}, the fact that 
$(\widehat{\GT}^f(\C),\Assoc^f(\C),\GRT^f(\C))$ has a natural torsor 
structure so that $\Assoc^f(\mathbbm{Q}$ is not empty and the fact 
that we have group isomorphisms $\widehat{\GT}^f(\C) \to
 \widehat{\GT}(\C)$ and $\GRT^f(\C)\to\GRT(\C)$.

The proof of the fact that \eqref{Hom:fr:1} and \eqref{Hom:fr:2} are bijections 
comes from the fact that the set $\Assoc^f(\mathbbm{Q})$ is not empty 
thus we have a chains of set morphisms
$$
\Assoc^f(\mathbbm{Q}) \to \Assoc(\mathbbm{Q}) \to \tmop{Iso}_{\tmop{Ho}
    (\on{OpTop})} ((\on{D}_{2})_{\mathbbm{Q}}^{\wedge}, L G_{\bullet}
    \on{H}^{\ast} (\on{D}_{ 2})) \to \tmop{Iso}_{\tmop{Ho}
    (\on{OpTop})} ((\on{D}^f_{2})_{\mathbbm{Q}}^{\wedge}, L G_{\bullet}
    \on{H}^{\ast} (\on{D}^f_{ 2}))
$$
and
$$
   \widehat{\GT}^f(\mathbbm{Q}) \to   \widehat{\GT}(\mathbbm{Q}) \to 
   \tmop{Aut}_{\tmop{Ho} (\on{OpTop})} ((\on{D}_{
    2})_{\mathbbm{Q}}^{\wedge}) \to
    \tmop{Aut}_{\tmop{Ho} (\on{OpTop})} ((\on{D}^f_{
    2})_{\mathbbm{Q}}^{\wedge}).
$$
The central morphisms are those constructed and proven to be isomorphisms by 
Fresse in \cite{Fresse2}, and the leftmost bijections are those coming 
from the isomorphism \eqref{bitorsor:cl:fr}. The rightmost maps are constructed as sections of the restriction maps 
\begin{eqnarray*}
\tmop{Iso}_{\tmop{Ho}
    (\on{OpTop})} ((\on{D}^f_{2})_{\mathbbm{Q}}^{\wedge}, L G_{\bullet}
    \on{H}^{\ast} (\on{D}^f_{ 2})) & \to & \tmop{Iso}_{\tmop{Ho}
    (\on{OpTop})} ((\on{D}_{2})_{\mathbbm{Q}}^{\wedge}, L G_{\bullet}
    \on{H}^{\ast} (\on{D}_{ 2})) \\
\tmop{Aut}_{\tmop{Ho} (\on{OpTop})} ((\on{D}^f_{
    2})_{\mathbbm{Q}}^{\wedge}) & \to &  \tmop{Aut}_{\tmop{Ho} (\on{OpTop})} ((\on{D}_{
    2})_{\mathbbm{Q}}^{\wedge})
\end{eqnarray*}
since $\PaB^f$ and $\t^f$ are quotients of $\PaB$ and $\t$.
\end{proof}

\section{Modules associated to framed configuration spaces (genus $g$ associators)}
\label{g framed associators}
\subsection{The module of parenthesized genus $g$ framed braidings}

\subsubsection{Compactified configuration spaces of surfaces}\label{Bgn}

Let $g \geq 0$ and $n
> 0$ be two integers and consider a compact topological oriented surface $\Sigma_g$ of
genus $g$.

The boundary $\partial \overline{\on{Conf}}^f(\Sigma_g,I) =
\overline{\on{Conf}}^f(\Sigma_g,I) - {  \tmop{Conf}}^f (\Sigma_g, I)$ is
made of the following irreducible components: for any decomposition $I = J_1 \coprod
\cdots \coprod J_k$ there is a component
\[ \partial_{J_1, \cdots, J_k} \overline{\on{Conf}}^f(\Sigma_g,I) \cong
   \prod_{i = 1}^k \overline{\on{C}}^f(\C, J_i) \times 
   \overline{\on{Conf}}^f(\Sigma_g,k) \hspace{0.17em} . \]
The inclusion of boundary components with respect to the direction of
the frame data provide $\overline{\on{Conf}}^f(\Sigma_g,-)$
with the structure of a module over the operad $\overline{\on{C}}^f(\C,-)$ in
topological spaces. 

We can represent the action of
$\overline{\on{C}}^f(\C,-)$ on $\overline{\on{Conf}}^f(\Sigma_g,-)$ as follows (in the case $g =
2$):
\begin{center}
\includegraphics[scale=1]{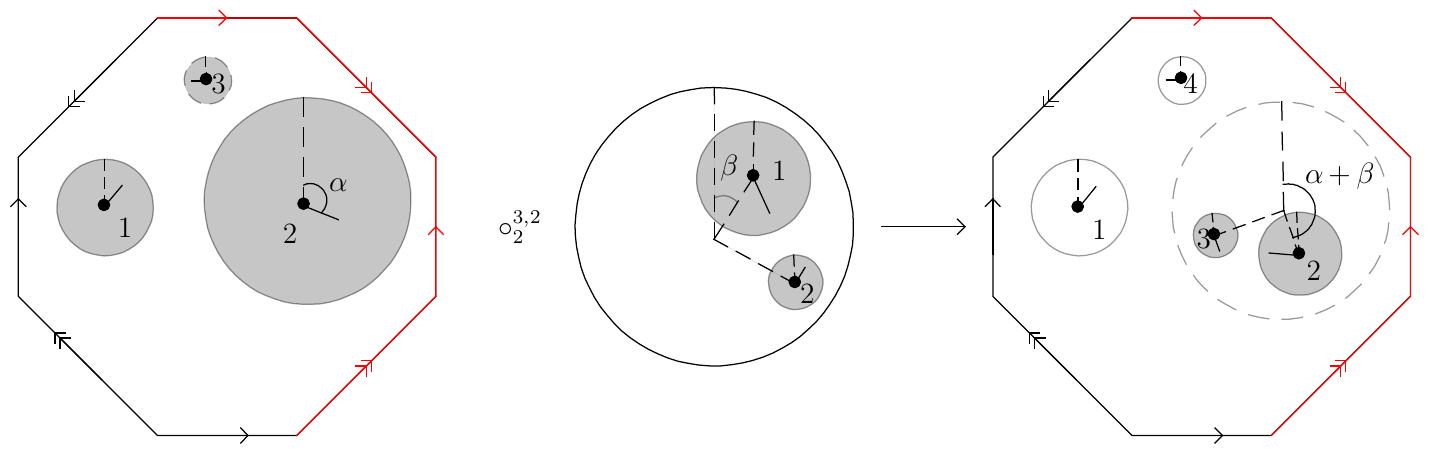}
\end{center}

\subsubsection{A presentation for surface framed braid groups}
\label{PBgnf}

We recall that composition of paths are read from left to right. 
In particular the commutator of two elements $A,B$, is $(A,B)=ABA^{-1}B^{-1}$.
For $n\geq1$, we denote $\on{PB}^f_{g,n}$ the fundamental group 
of $\on{Conf}^f(\Sigma_g,n)$ and we define the framed braid group 
on $\Sigma_g$ as the group 
$\on{B}^f_{g,n}$ generated by $$X_1^1, Y_1^1, \ldots, X_1^g, Y_1^g, \tau_1, 
\ldots, \tau_{n-1}, f_{1}, \ldots, f_{n},$$ together with the following relations 
\begin{flalign}
& \tau_i \tau_{i + 1} \tau_i = \tau_{i + 1} \tau_i \tau_{i + 1}\,,      		 
\quad \text{ for all } 1 \leq i \leq n-2,  \tag{T1} \label{eqn:T1} \\
& (\tau_i, \tau_j)= 1\,,      		\quad \text{ if } |i - j| > 1,  \tag{T2} \label{eqn:T2} \\
& f_if_j=f_jf_i\,,      		\quad \text{ for all }1 \leq i,j \leq n,  \tag{FT1} \label{eqn:FT1} \\
& \tau_i f_j = f_{j}\tau_i\,,       \quad \text{ for all }j \neq i,i+1   \tag{FT2} \label{eqn:FT2} \\
& \tau_i f_i = f_{i+1}\tau_i,\, \quad f_{i}\tau_i=\tau_i f_{i+1}  
\quad \text{ for all } 1 \leq i \leq n-2,   \tag{FB3} \label{eqn:FB3} \\
& (X^a_1,\tau_i)=(Y^a_1,\tau_i)=1, 
\quad \text{ for all } i=2,\ldots,n-1, \text{ and }1\leq a \leq g,   \tag{FBG1} \label{eqn:FBG1} \\
& (X_1^a,X_2^a)=(Y_1^a,Y_2^a)=1, \quad \text{ for } 1\leq a \leq g,  
\tag{FBG2} \label{eqn:FBG2} \\
& (X^a_2,Y^a_1)=\tau_1^2, \quad \text{ for } 1\leq a \leq g, 
\tag{FBG3} \label{eqn:FBG3} \\
&  (X_1^a,X_2^b)=(X_1^a,Y_2^b)=(Y_1^a,X_2^b)=(Y_1^a,Y_2^b)=1
 \quad \text{ for } 1 \leq b<a\leq g,  \tag{FBG4} \label{eqn:FBG4}  \\
& \prod_{a=1}^{g} ((X_1^{a})^{-1},Y^{a}_1)= \tau_1\cdots \tau_{n-2} 
\tau_{n-1}^2 \tau_{n-2} \cdots \tau_1 f_1^{2(g-1)}. & \tag{FBG5} \label{eqn:FBG5}
\end{flalign}
Here $X^a_{i+1} = \tau_iX_i^a \tau_i$ and $Y^a_{i+1} = \tau_iY^a_i\tau_i$
for $i=1,\ldots ,n-1$.

The corresponding geometric configuration of points and paths for 
the above presentation is the same that the one used in \cite{BeG} 
which we now recall. Let $\on{B}_{g,2n}$ be the fundamental group 
of $\tmop{Conf} (\Sigma_g, [2n])$ based at the point $\underline{p}=(p_1,\ldots,p_n)$ 
where the $p_i$ are aligned in the right-most $A$-generating cycle of $\Sigma_g$
 (see the picture below). It is generated by paths 
$\tilde X_1^1, \tilde Y_1^1, \ldots, \tilde X_1^g, \tilde Y_1^g, \sigma_1, 
\ldots, \sigma_{2n-1}$ corresponding geometrically in particular to 
the following paths:
\begin{center}
\includegraphics[scale=1]{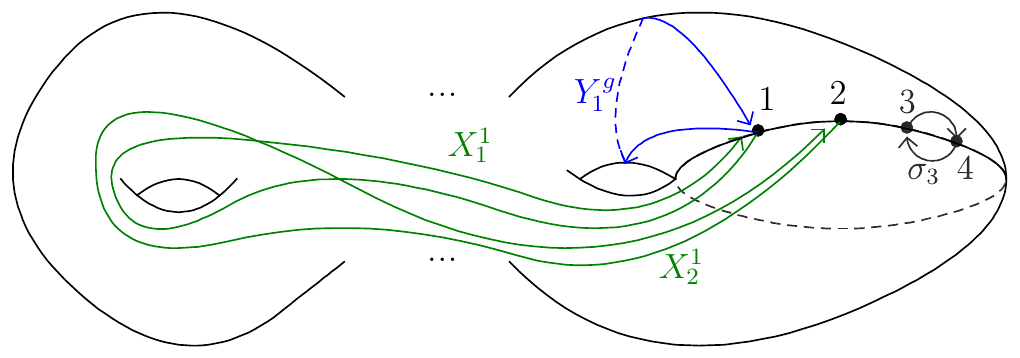}
\\ \text{The geometric configuration of $B_{g,2n}$, for $n=2$.}
\end{center}
There is a morphism $\on{B}_{g, n} \to \mathfrak{S}_n$ given by 
$\tilde X_1^a, \tilde Y_1^a \mapsto 1$, $\sigma_i
\mapsto s_i \assign (i, i + 1)$. It is proved in {\cite{Bell}} that the 
fundamental group $\pi_1 (\tmop{Conf} (\Sigma_g, n))$ is isomorphic 
to the genus $g$ pure braid group $\on{PB}_{g, n}$ which
is the kernel of this map and is generated by $\tilde X_i^a, \tilde Y_i^a$ 
($1 \leq i \leq n$, $1 \leq a \leq g$), where $Z_{i+1}^a =
\sigma_{i}  Z_i^a \sigma_{i}$ for $Z$ any of the letters $X, Y$.

Then, $\on{B}^f_{g,n}$ is seen as a subgroup of $\on{B}_{g,2n}$ and 
its generators can be written in terms of the generators of $\on{B}_{g,2n}$:
$$
X_1^a=\tilde X_1^a\tilde X_2^a \sigma_1^2, \quad Y_1^a=
\tilde Y_1^a\tilde Y_2^a \sigma_1^2.
$$

Let us assume that $g>1$. In \cite{BeG}, the authors showed that 
$\on{PB}^f_{g, n}$ can be exhibed as a non-splitting central extension
\begin{equation}\label{eq:pure framed braid and pure braid}
 1\longrightarrow \Z^{n}\longrightarrow \on{PB}^f_{g,n}\stackrel{\beta_{n}}{\longrightarrow}  \on{PB}_{g,n}\longrightarrow 1,
\end{equation}
where $\beta_{n}$ is the morphism induced by the projection map 
$\on{Conf}^f(\Sigma_{g},n)\to \on{Conf}(\Sigma_{g},n)$ 
({\em i.e.} $\beta_{n}$ consists in forgetting the framing). 
$\on{Conf}^f(\Sigma_{g},n)$ is an Eilenberg--Maclane space of type 
$(\on{PB}^f_{g,n},1)$.
This short exact sequence extends to the following non-split 
short exact sequence
\begin{equation}\label{eq:framed braid and braid}
1\longrightarrow \Z^{n}\longrightarrow \on{B}^f_{g,n}\stackrel{\widehat{\beta}_{n}}{\longrightarrow}  
\on{B}_{g,n}\longrightarrow 1\; ,
\end{equation}
where again $\widehat{\beta}_{n}$ consists in forgetting the framing. 
$\on{Conf}^f(\Sigma_{g},[n])$ is an Eilenberg--Maclane space 
of type $(\on{B}^f_{g,n},1)$.

\begin{definition}
  Let $\CoB^f_g$ the $
  \CoB^f$-module in groupoids with $\mathfrak{S}$-module of 
  objects $\mathfrak{S}$ and
  where, for $n \geqslant 1$, the morphisms of $
  \CoB^f_g (n)$ consists of isotopy classes of genus $g$ framed braids (i.e. elements
  of the braid group $\on{B}^f_{g, n}$) $\alpha$ together with a colouring bijection
  $i \mapsto \alpha_i$ between the index set $i \in \{ 1, \ldots, n \}$ which
  leaves the last strand uncoloured and the strands $\alpha_i \in \{ \alpha_1,
  \ldots, \alpha_n \}$ of our braid $\alpha$ and the data of a special braid
  corresponding to the framing.
\end{definition}

\subsubsection{The $\PaB^f$-module of parenthesized framed genus $g$ braids}

Let us choose an embedding $\mathbbm{S}^1 \hookrightarrow \Sigma_g$. 
To any finite set $I$ we associate the ASFM compactification
 $\overline{\textrm{Conf}}(\mathbb{S}^1,I)$ of the configuration space 
 $\textrm{Conf}(\mathbb{S}^1,I)$ of $\mathbb{S}^1$. 
 To any finite set $I$, we associate the framed ASFM compactification $
\overline{\on{Conf}}^f(\Sigma_g,I)$ of ${ \tmop{Conf}}^f (\Sigma_g, I)$.
The inclusion of boundary components provide $\overline{\rm Conf}(\mathbb{S}^1,-)$ 
with the structure of a module over the operad 
$\overline{\rm C}(\mathbb{R},-)$ in $\mathbf{Top}$.

 We
have inclusions of topological modules
\[ \Pa\subset 
   \overline{ \on{Conf}} (\mathbbm{S}^1, -) \hspace{0.17em} \subset \hspace{0.17em}
   \overline{\on{Conf}}^f(\Sigma_g,-), \]
   over the topological operads 
\[ \Pa\subset 
   \overline{ \on{C}} (\RR, -) \hspace{0.17em} \subset \hspace{0.17em}
   \overline{\on{C}}^f(\C,-) . \]
We then define
\[ \PaB^f_g \assign \pi_1
   (\overline{\on{Conf}}^f(\Sigma_g,-),
   \Pa) \hspace{0.17em}, \]
which is a $\PaB^f$-module in
groupoids.

As all our modules are considered to be pointed, there is a map of 
$\mathfrak{S}$-modules $\PaB^f \longrightarrow\PaB^f_g$ ane we 
abusively denote $R^{1,2}$, $\tilde R^{1,2}$, $\Phi^{1,2,3}$ and 
$F^{1,2}$ the images in $\PaB^f_g$ of the corresponding arrows in 
$\PaB^f$. Notice that in this case, $\PaB^f_g(1)$ is not the trivial groupoid so the choice 
of the pointing here is not functorial (appart from the reduced elliptic case studied in \cite{CaGo2}).

\begin{example}[Structure of $\PaB^f_g(1)$]\label{Exbgn}

As opposed to the unframed reduced genus 1 case studied in \cite{CaGo2}, 
we have non trivial arrows in arity 1. More precisely, we have $2 g$ 
automorphisms, $A^1_a$ and $B^1_a\in\on{Aut}_{\PaB^f_g(1)}(1)$,  for all 
$1 \leqslant a \leqslant g$, corresponding to the inverse generating loops in 
$ \overline{\on{Conf}}^f(\Sigma_g,1)$. Here is a picture for $A^1_1$ and $B^1_1$ for $g=2$:
\begin{center}
\includegraphics[scale=1]{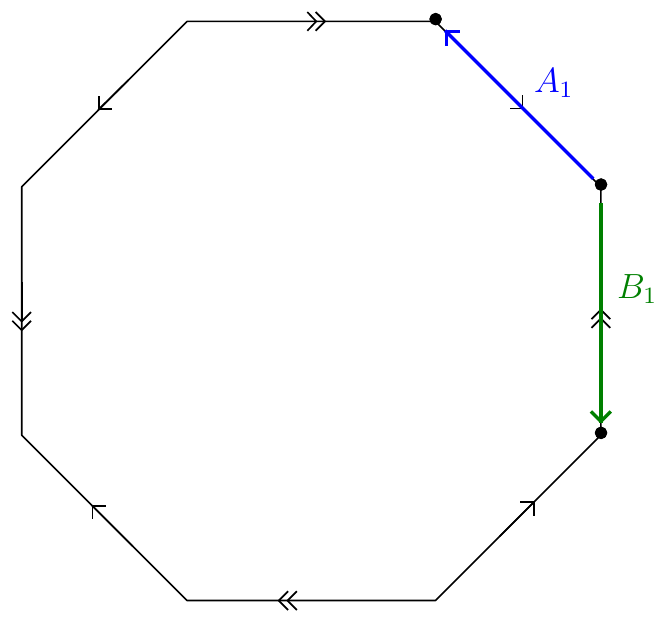}
\end{center}
All other $A^1_a$ and $B^1_a$ are depicted in the same way. We will 
formally depict these arrows as \textit{pic diagrams} pointing to the left, expressing the fact that 
the paths considered go in the opposite direction of the generating paths considered in $\pi_1(\Sigma_g)$:
\begin{center}
\begin{tikzpicture}[baseline=(current bounding box.center)]
\tikzstyle point=[circle, fill=black, inner sep=0.05cm] 
\node[point, label=above:$1$] at (1,1) {};
 \node[point, label=below:$1$] at (1,0) {};
 \draw[-,thick] (1,1) .. controls (1,0.5) and (1,0.5).. (0.5,0.5); 
 \draw[->,thick] (0.5,0.5) .. controls (1,0.5) and (1,0.5).. (1,0.05); 
\node[point, white, label=left:$A_a$] at (0.5,0.5) {};
\end{tikzpicture} $\qquad  \qquad \qquad $
\begin{tikzpicture}[baseline=(current bounding box.center)]
\tikzstyle point=[circle, fill=black, inner sep=0.05cm] 
\node[point, label=above:$1$] at (1,1) {};
 \node[point, label=below:$1$] at (1,0) {};
 \draw[-,thick] (1,1) .. controls (1,0.5) and (1,0.5).. (0.5,0.5); 
 \draw[->,thick] (0.5,0.5) .. controls (1,0.5) and (1,0.5).. (1,0.05); 
\node[point, white, label=left:$B_a$] at (0.5,0.5) {};
\end{tikzpicture} 
\end{center}

\end{example}

\begin{remark}
One has to be careful with the above notation. Indeed, taking into acount 
the framing data for, say, $A^1_1$, the above geometrical picture corresponds 
to the following
\begin{center}
\includegraphics[scale=1]{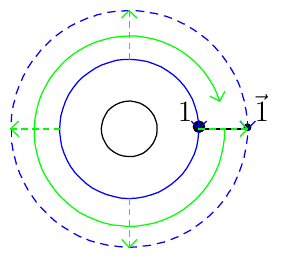}
\end{center}
In other words, seeing $\on{Aut}_{\PaB^f_g(1)}(1)=\on{PB}_{g,1}$ as a 
subgroup of $\on{B}_{g,2}$, the element $A_1$ is identified with the composite 
$X_1^1X_2^1 \sigma_1^2$.
\end{remark}

\begin{example}[Structure of $\PaB^f_g
(2)$]\label{Exbgn2}
We have $4 g$ automorphisms, $A_a^{1,2}, \tilde A_a^{1,2}, B_a^{1,2}$ and 
$\tilde B_a^{1,2}\in\on{End}_{\PaB^f_g(2)}(12)$, for all $1 \leqslant a \leqslant g$ 
corresponding, for the case of $A_1^{1,2}$ and $\tilde B_1^{1,2}$  to the 
following paths in $\overline{\on{Conf}}^f(\Sigma_g,2)$ (we neglect for 
simplicity the framing data):
\begin{center}
\includegraphics[scale=1]{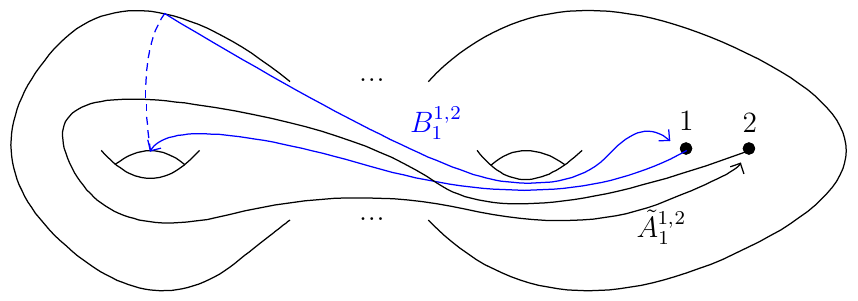}
\end{center}
We will represent the paths $A_a^{1,2}$ and $B_a^{1,2}$ as follows:
\begin{center}
\begin{align}
\begin{tikzpicture}[baseline=(current bounding box.center)]
\tikzstyle point=[circle, fill=black, inner sep=0.05cm] 
\node[point, label=above:$1$] at (1,1) {};
 \node[point, label=below:$1$] at (1,0) {};
 \node[point, label=above:$2$] at (1.5,1) {};
 \node[point, label=below:$2$] at (1.5,0) {};
 \draw[-,thick] (1,1) .. controls (1,0.5) and (1,0.5).. (0.5,0.5); 
 \draw[->,thick] (0.5,0.5) .. controls (1,0.5) and (1,0.5).. (1,0.05); 
 \draw[-,thick] (1.5,1) .. controls (1.5,0.5) and (1.5,0.5).. (1.5,0.5); 
 \draw[->,thick] (1.5,0.5) .. controls (1.5,0.5) and (1.5,0.5).. (1.5,0.05); 
\node[point, white, label=left:$A_a$] at (0.5,0.5) {};
\end{tikzpicture} \qquad 
\begin{tikzpicture}[baseline=(current bounding box.center)]
\tikzstyle point=[circle, fill=black, inner sep=0.05cm] 
\node[point, label=above:$1$] at (1,1) {};
 \node[point, label=below:$1$] at (1,0) {};
 \node[point, label=above:$2$] at (1.5,1) {};
 \node[point, label=below:$2$] at (1.5,0) {};
 \draw[-,thick] (1,1) .. controls (1,0.5) and (1,0.5).. (0.5,0.5); 
 \draw[->,thick] (0.5,0.5) .. controls (1,0.5) and (1,0.5).. (1,0.05); 
 \draw[-,thick] (1.5,1) .. controls (1.5,0.5) and (1.5,0.5).. (1.5,0.5); 
 \draw[->,thick] (1.5,0.5) .. controls (1.5,0.5) and (1.5,0.5).. (1.5,0.05); 
\node[point, white, label=left:$B_a$] at (0.5,0.5) {};
\end{tikzpicture} 
 \end{align}
\end{center}
All other $A_a^{1,2}$ and $B_a^{1,2}$ are depicted along the same 
representation as that for $B_1^{1,2}$.

Moreover, $\tilde A_a^{1,2}$ and $\tilde B_a^{1,2}$ can also  be depicted as follows
$$\begin{tikzpicture}[baseline=(current bounding box.center)]
\tikzstyle point=[circle, fill=black, inner sep=0.05cm] 
\node[point, label=above:$1$] at (1,1) {};
 \node[point, label=below:$1$] at (1,0) {};
 \node[point, label=above:$2$] at (2,1) {};
 \node[point, label=below:$2$] at (2,0) {};
 \draw[-,thick] (1,1) .. controls (1,0.5) and (1,0.5).. (1,0.5); 
 \draw[->,thick] (1,0.5) .. controls (1,0.5) and (1,0.5).. (1,0.05); 
 \draw[-,thick] (1+1,1) .. controls (1+1,0.5) and (1+1,0.5).. (1.5+1,0.5); 
 \draw[->,thick] (1.5+1,0.5) .. controls (1+1,0.5) and (1+1,0.5).. (1+1,0.05); 
\node[point, white, label=left:$\tilde A_a$] at (2,0.5) {};
\end{tikzpicture} \qquad \begin{tikzpicture}[baseline=(current bounding box.center)]
\tikzstyle point=[circle, fill=black, inner sep=0.05cm] 
\node[point, label=above:$1$] at (1,1) {};
 \node[point, label=below:$1$] at (1,0) {};
 \node[point, label=above:$2$] at (2,1) {};
 \node[point, label=below:$2$] at (2,0) {};
 \draw[-,thick] (1,1) .. controls (1,0.5) and (1,0.5).. (1,0.5); 
 \draw[->,thick] (1,0.5) .. controls (1,0.5) and (1,0.5).. (1,0.05); 
 \draw[-,thick] (1+1,1) .. controls (1+1,0.5) and (1+1,0.5).. (1.5+1,0.5); 
 \draw[->,thick] (1.5+1,0.5) .. controls (1+1,0.5) and (1+1,0.5).. (1+1,0.05); 
\node[point, white, label=left:$\tilde B_a$] at (2,0.5) {};
\end{tikzpicture} $$
\end{example}

\begin{remark}
Doubling the braid $A_a\in\PaB^f_g(1)$ amounts to taking 
$\circ_1(A_a,\on{id}_{12})\in\PaB^f_g(2)$, we get an arrow $A_a^{12}$ depicted as follows:
$$
\begin{tikzpicture}[baseline=(current bounding box.center)]
\tikzstyle point=[circle, fill=black, inner sep=0.05cm] 
\node[point, label=above:$1$] at (1,1) {};
 \node[point, label=below:$1$] at (1,0) {};
 \node[point, label=above:$2$] at (1.5,1) {};
 \node[point, label=below:$2$] at (1.5,0) {};
 \draw[-,thick] (1,1) .. controls (1,0.5) and (1,0.5).. (0.5,0.5); 
 \draw[->,thick] (0.5,0.5) .. controls (1,0.5) and (1,0.5).. (1,0.05); 
 \draw[-,thick] (1.5,1) .. controls (1.5,0.5) and (1.5,0.5).. (1,0.5); 
 \draw[->,thick] (1,0.5) .. controls (1.5,0.5) and (1.5,0.5).. (1.5,0.05); 
\node[point, white, label=left:$A_a$] at (0.5,0.5) {};
\end{tikzpicture}
$$
As both braids represent actually ribbon braids, then it is then a fact that 
\begin{center}
\begin{align}
A_a^{12}(A_a^{1,2})^{-1} (R^{1,2} R^{2,1})^{-1}=
   \begin{tikzpicture}[baseline=(current bounding box.center)]
\tikzstyle point=[circle, fill=black, inner sep=0.05cm] 
\node[point, label=above:$1$] at (1,1) {};
 \node[point, label=below:$1$] at (1,0) {};
 \node[point, label=above:$2$] at (2.5,1) {};
 \node[point, label=below:$2$] at (2.5,0) {};
 \draw[-,thick] (1,1) .. controls (1,0.5) and (1,0.5).. (1,0.5); 
 \draw[->,thick] (1,0.5) .. controls (1,0.5) and (1,0.5).. (1,0.05); 
 \draw[-,thick] (2.5,1) .. controls (2.5,0.5) and (2.5,0.5).. (2,0.5); 
 \draw[->,thick] (2,0.5) .. controls (2.5,0.5) and (2.5,0.5).. (2.5,0.05); 
\node[point, white, label=left:$A_a$] at (2,0.5) {};
\end{tikzpicture} 
 \end{align}
\end{center}
This means that, contrary to the reduced genus 1 case, 
$A_a^{1,2}$ and $\tilde A_a^{1,2}$ are not equal.
Nevertheless, one can retrieve the latter arrow from the first one : 
 $$\tilde A_a^{1,2}=(A_a^{12})^{-1}A_a^{1,2} R^{1,2} R^{2,1}.$$
\end{remark}

The following theorem can be undestood as a rephrasing of the
MacLane-Joyal-Street coherence theorem for framed genus $g$
$\on{D}_2$-modules.
\begin{theorem}\label{Thm:PaBfg}
  As a $\PaB^f$-module in
    groupoids having $ \textbf{\tmop{\Pa}}$ as
    $ \textbf{\tmop{\Pa}}$-module of objects, $\PaB^f_g$ is
freely generated by
$A_a^{1, 2}$ and $B_a^{1, 2}$, for $1
    \leqslant a \leqslant g$, in $\on{Aut}_{\PaB^f_g(2)}(12)$,
together with the following relations, for all $1 \leq a \leq g$ and $Z$ any of the letters $A,B$:
\begin{flalign}
&  Z_a^{\emptyset,1}=\on{Id}^{1}, \quad\Big(\textrm{in }\mathrm{End}_{\PaB_g^f(1)}\big(1\big)\Big)\,,			\tag{R$_g$}\label{eqn:gR} \\
& Z_a^{(12)3}=\Phi^{1,2,3}Z_a^{1,23} R^{1,23}\Phi^{2,3,1}Z_a^{2,31} R^{2,31}\Phi^{3,1,2}Z_a^{3,12} R^{3,12},				\tag{D$_g$}\label{eqn:gD} \\
& \on{Id}^{12,3} =\left(\Phi^{1,2,3}Z_b^{1,23}(\Phi^{1,2,3})^{-1}, R^{1,2}\Phi^{2,1,3}Z_a^{2,13}(\Phi^{2,1,3})^{-1} R^{2,1}  \right)\,,	 \quad \text{for } (1\leq a<b \leq g),
\tag{N$_g$}\label{eqn:gN}  \\
& R^{1,2}R^{2,1}=\left(R^{1,2}\Phi^{2,1,3}A_a^{2,13}(\Phi^{2,1,3})^{-1} R^{2,1},\Phi^{1,2,3}B_a^{1,23}(\Phi^{1,2,3})^{-1} \right)\,,		&
\tag{E1$_g$}\label{eqn:gE1} 
\end{flalign}
as relations holding in the automorphism group of $(12)3$ in $\PaB_g^f(3)$, and
\begin{flalign}
&R^{1,2}R^{2,1}(F^{1,2})^{2(g-1)}=\prod_{a=1}^g\left((A_{a}^{1,2})^{-1},B_{a}^{1,2}\right), &	\tag{E2$_g$}\label{eqn:gE2}
\end{flalign}
as a relation holding in the automorphism group of $(12)$ in $\PaB_g^f(2)$.
\end{theorem}

\begin{remark}
Some direct consequences of the above theorem:
\begin{itemize}
\item By removing the third strand in \eqref{eqn:gD} 
and using $Z_a^{\emptyset,1}=\on{Id}^{1}$, 
one deduces that $A_a^{1, \emptyset}=A^1_a$ and $ B_a^{1, \emptyset}=B^1_a$, 
where $A^1_a,B^1_a$ are the elements introduced in Example \ref{Exbgn};
\item $\PaB^f_g$ identifies with the fake pull-back $\omega^{\star}  \CoB^f_g$ 
of the $\CoB^f$-module $$\CoB^f_g:=\mathcal{G}(\on{B}_g^f \to \mathfrak{S}),$$
  along the forgetful functor $\omega : \Pa
  \longrightarrow \mathfrak{S}$.
\item Relation \eqref{eqn:gN} can be understood as a naturality condition 
between couples of elements $(A_a,B_b)$ suggesting that $\PaB_g^f$ 
is a $(\PaB_1^f \otimes \ldots \otimes \PaB_1^f)$-module in the category of right 
$\PaB^f$-modules.
\end{itemize}
\end{remark}
Before proving this theorem let us state a fact that will be useful later on.

\begin{lemma}\label{LE1}
Relation \eqref{eqn:gD} is equivalent to 
\begin{equation}\label{eqn:Nbis}
Z_a^{12, 3}= \Phi^{1, 2, 3} Z_a^{1, 23} (\Phi^{1, 2, 3})^{- 1} R^{1, 2}
   \Phi^{2, 1, 3} Z_a^{2, 13}  (\Phi^{2, 1, 3})^{- 1} R^{2, 1}.
\end{equation}
\end{lemma}

\begin{proof}[Proof of Lemma \ref{LE1}]
On the one hand, as $Z_a^{1,
\emptyset} = \tmop{Id}_1$, erasing the third strand both in relation \eqref{eqn:Pab:1:1} and in \eqref{eqn:Nbis},
implies $$Z_a^{1, 2} R^{1, 2} Z_a^{2, 1} R^{2, 1} =
Z_a^{12}.$$ Then, by doubling the first braid, this is equivalent to
\[ (Z_a^{3, 12})^{- 1} = R^{3, 12} Z_a^{12, 3} R^{12, 3} . \]
By using the above equation and the hexagon $R^{1, 23} \Phi^{2, 3, 1} = 
(\Phi^{1, 2, 3})^{- 1} R^{1, 2} \Phi^{2,
1, 3} R^{1, 3}$, then equation \eqref{eqn:Pab:1:1} reads
\begin{eqnarray*}
  Z_a^{(12) 3} & = & \Phi^{1, 2, 3} Z_a^{1, 23}  R^{1, 23} \Phi^{2, 3, 1}
  Z_a^{2, 31} R^{2, 31} \Phi^{3, 1, 2} Z_a^{3, 12}  R^{3, 12}\\
  & = & \Phi^{1, 2, 3} Z_a^{1, 23} (\Phi^{1, 2, 3})^{- 1}  R^{1, 2} \Phi^{2, 1,
  3} R^{1, 3} Z_a^{2, 31} R^{2, 31} \Phi^{3, 1, 2} Z_a^{3, 12}  R^{3, 12}\\
  & = & \Phi^{1, 2, 3} Z_a^{1, 23} (\Phi^{1, 2, 3})^{- 1}  R^{1, 2} \Phi^{2, 1,
  3} R^{1, 3} Z_a^{2, 31} R^{2, 31} \Phi^{3, 1, 2}  ( R^{12, 3})^{- 1} (Z_a^{12,
  3})^{- 1} Z_a^{(12) 3}.
\end{eqnarray*}
Next, as we have $R^{1, 3} Z_a^{2, 31} = Z_a^{2, 13} R^{1, 3} $ and $R^{1, 3}
R^{2, 31} = R^{2, 13} R^{1, 3}$, equation \eqref{eqn:Pab:1:1} is equivalent to
\[ Z_a^{12, 3} = \Phi^{1, 2, 3} Z_a^{1, 23} (\Phi^{1, 2, 3})^{- 1} R^{1, 2}
   \Phi^{2, 1, 3} Z_a^{2, 13} R^{2, 13} R^{1, 3} \Phi^{3, 1, 2}  (R^{12, 3})^{-
   1} . \]
Now, by applying the permutation $(123) \mapsto (312)$, the second hexagon
relation yields
\[ (R^{12, 3})^{- 1} = (\Phi^{3, 1, 2})^{- 1} (R^{1,3})^{- 1} \Phi^{1, 3, 2}
   (R^{2,3})^{- 1} (\Phi^{1, 2, 3})^{- 1} . \]
Thus, equation \eqref{eqn:Pab:1:1}  is equivalent to
\[ Z_a^{12, 3} = \Phi^{1, 2, 3} Z_a^{1, 23} (\Phi^{1, 2, 3})^{- 1} R^{1, 2}
   \Phi^{2, 1, 3} Z_a^{2, 13} R^{2, 13} \Phi^{1, 3, 2}  (R^{2,3})^{- 1}
   (\Phi^{1, 2, 3})^{- 1} . \]
By applying the permutation $(12) \mapsto (21)$, the first hexagon relation
yields
\[ R^{2, 13} = (\Phi^{2, 1, 3})^{- 1} R^{2, 1} \Phi^{1, 2, 3} R^{2, 3}
   (\Phi^{1, 3, 2})^{- 1} . \]
Thus, equation \eqref{eqn:Pab:1:1}  is equivalent to
\begin{equation}
Z_a^{12, 3}= \Phi^{1, 2, 3} Z_a^{1, 23} (\Phi^{1, 2, 3})^{- 1} R^{1, 2}
   \Phi^{2, 1, 3} Z_a^{2, 13}  (\Phi^{2, 1, 3})^{- 1} R^{2, 1}.
\end{equation}
\end{proof}

\begin{proof}[Proof of Theorem \ref{Thm:PaBfg}]

Let $\mathcal Q$ be the $\PaB^f$-module with the above presentation. 
We first show that there is a morphism of $\PaB^f$-modules $\mathcal Q\to\PaB_g^f$. 
We have already seen that there are $2g$ automorphisms $A_a^{1,2},B_a^{1,2}$ 
of $(12)$ in $\mathbf{PaB}_g^f(2)$ (see Example \ref{Exbgn2}). We have to 
prove that they indeed satisfy the relations \eqref{eqn:gR}, \eqref{eqn:gD}, 
\eqref{eqn:gN}, \eqref{eqn:gE1}, and \eqref{eqn:gE2}. 

\medskip

\noindent\underline{Relation \eqref{eqn:gR} is satisfied:} 
This is straightforwardly satisfied as it corresponds topologically to 
removing the first brand to the paths $A_a^{1,2},B_a^{1,2}$ that move the first strand, leaving the 
second strand untouched.

\medskip

\noindent\underline{Relation \eqref{eqn:gD} is satisfied:} 
The \textit{decagon relation} \eqref{eqn:gD} can be depicted as follows 
(for simplicity we abusively neglect picturing the framing data): 
\begin{center}
\begin{align}\tag{$D_g$}
\begin{tikzpicture}[baseline=(current bounding box.center)]
\tikzstyle point=[circle, fill=black, inner sep=0.05cm] 
\node[point, label=above:$(1$] at (1,1) {};
 \node[point, label=below:$(1$] at (1,0) {};
 \node[point, label=above:$2)$] at (1.5,1) {};
 \node[point, label=below:$2)$] at (1.5,0) {};
 \node[point, label=above:$3$] at (3,1) {};
 \node[point, label=below:$3$] at (3,0) {};
 \draw[-,thick] (1,1) .. controls (1,0.5) and (1,0.5).. (0.5,0.5); 
 \draw[->,thick] (0.5,0.5) .. controls (1,0.5) and (1,0.5).. (1,0.05); 
 \draw[-,thick] (1.5,1) .. controls (1.5,0.5) and (1.5,0.5).. (1,0.5); 
 \draw[->,thick] (1,0.5) .. controls (1.5,0.5) and (1.5,0.5).. (1.5,0.05); 
\node[point, white, label=left:$Z_a^{(12)3}$] at (0.5,0.5) {};
 \draw[-,thick] (3,1) .. controls (3,0.5) and (3,0.5).. (2.5,0.5); 
 \draw[->,thick] (2.5,0.5) .. controls (3,0.5) and (3,0.5).. (3,0.05); 
\end{tikzpicture}
\qquad = \qquad
\begin{tikzpicture}[baseline=(current bounding box.center)] 
\tikzstyle point=[circle, fill=black, inner sep=0.05cm]
 \node[point, label=above:$(1$] at (1,4) {};
 \node[point, label=below:$(1$] at (1,-2) {};
 \node[point, label=above:$2)$] at (1.5,4) {};
 \node[point, label=below:$2)$] at (1.5,-2) {};
 \node[point, label=above:$3$] at (3,4) {};
 \node[point, label=below:$3$] at (3,-2) {};
 \draw[-,thick] (1,4) .. controls (1,3.75) and (1,3.75).. (1,3.5);
 \draw[-,thick] (1.5,4) .. controls (1.5,3.75) and (2.5,3.75).. (2.5,3.5);
 \draw[-,thick] (1,3.5) .. controls (1,3.25) and (1,3.25).. (0.5,3.25); 
\node[point, white, label=left:$Z_a^{1,23}$] at (0.5,3.25) {};
 \draw[-,thick] (0.5,3.25) .. controls (1,3.25) and (1,3.25).. (1,3); 
 \draw[-,thick] (3,4) .. controls (3,3.75) and (3,3.75).. (3,3);
 \draw[-,thick] (3,3) .. controls (3,3.75) and (3,3.75).. (3,3);
 \draw[-,thick] (2.5,3.5) .. controls (2.5,3.25) and (2.5,3.25).. (2.5,3); 
 \draw[-,thick] (1,3) .. controls (1.1,2.5) and (2.9,2.5).. (3,2);
\node[point, ,white] at (2.15,2.45) {};
\node[point, ,white] at (1.85,2.55) {};
\draw[-,thick] (2.5,3) .. controls (2.5,2.5) and (1,2.5).. (1,2);
\draw[-,thick] (3,3) .. controls (3,2.5) and (1.5,2.5).. (1.5,2);
\draw[-,thick] (1,2) .. controls (1,1.75) and (1,1.75).. (1,1.5);
\draw[-,thick] (1.5,2) .. controls (1.5,1.75) and (2.5,1.75).. (2.5,1.5);
 \draw[-,thick] (1,1.5) .. controls (1,1.25) and (1,1.25).. (0.5,1.25); 
\node[point, white, label=left:$Z_a^{2,31}$] at (0.5,1.25) {};
 \draw[-,thick] (0.5,1.25) .. controls (1,1.25) and (1,1.25).. (1,1); 
 \draw[-,thick] (3,2) .. controls (3,1.75) and (3,1.75).. (3,1);
 \draw[-,thick] (3,1) .. controls (3,1.75) and (3,1.75).. (3,1);
 \draw[-,thick] (2.5,1.5) .. controls (2.5,1.25) and (2.5,1.25).. (2.5,1); 
 \draw[-,thick] (1,1) .. controls (1.1,0.5) and (2.9,0.5).. (3,0);
\node[point, ,white] at (2.15,0.45) {};
\node[point, ,white] at (1.85,0.55) {};
\draw[-,thick] (2.5,1) .. controls (2.5,0.5) and (1,0.5).. (1,0);
\draw[-,thick] (3,1) .. controls (3,0.5) and (1.5,0.5).. (1.5,0);
\draw[-,thick] (1,0) .. controls (1,-0.25) and (1,-0.25).. (1,-0.5);
\draw[-,thick] (1.5,0) .. controls (1.5,-0.25) and (2.5,-0.25).. (2.5,-0.5);
 \draw[-,thick] (1,-0.5) .. controls (1,-0.75) and (1,-0.75).. (0.5,-0.75); 
\node[point, white, label=left:$Z_a^{3,12}$] at (0.5,-0.75) {};
 \draw[-,thick] (0.5,-0.75) .. controls (1,-0.75) and (1,-0.75).. (1,-1); 
 \draw[-,thick] (3,0) .. controls (3,-0.25) and (3,-0.25).. (3,-1);
 \draw[-,thick] (3,-1) .. controls (3,-0.25) and (3,-0.25).. (3,-1);
 \draw[-,thick] (2.5,-0.5) .. controls (2.5,-0.75) and (2.5,-0.75).. (2.5,-1); 
 \draw[->,thick] (1,-1) .. controls (1.1,-1.5) and (2.9,-1.5).. (3,-1.95);
\node[point, ,white] at (2.15,-1.55) {};
\node[point, ,white] at (1.85,-1.45) {};
\draw[->,thick] (2.5,-1) .. controls (2.5,-1.5) and (1,-1.5).. (1,-1.95);
\draw[->,thick] (3,-1) .. controls (3,-1.5) and (1.5,-1.5).. (1.5,-1.95);
\end{tikzpicture} \end{align}
\end{center}
It is satisfied in $\PaB_{g}^f$, expressing the fact that when all (here, three) 
points with their associated framing data move along a generating loop on $\Sigma_g$ (in the opposite direction), 
this corresponds to the path in the framed configuration space of points on 
$\Sigma_g$ moving and twisting simultaneously the three points. 
Thus, the number of twists in the l.h.s. and r.h.s. of \eqref{eqn:gD} are equal and cancel out. 
\medskip

\noindent\underline{Relation \eqref{eqn:gN}  is satisfied:} 
This is straightforwardly satisfied as the braids corresponding to the l.h.s. 
and r.h.s. of the comutator are independent.

\medskip

\noindent\underline{Relation \eqref{eqn:gE1} is satisfied:} 
One can interpret the path in the r.h.s. of \eqref{eqn:gE1} 
as follows. Consider the path
$$K=\left(R^{1,2}\Phi^{2,1,3}A_a^{2,13}(\Phi^{2,1,3})^{-1} R^{2,1},
\Phi^{1,2,3}B_a^{1,23}(\Phi^{1,2,3})^{-1} \right).$$
This can be depicted as follows:
\begin{center}
\includegraphics[scale=1]{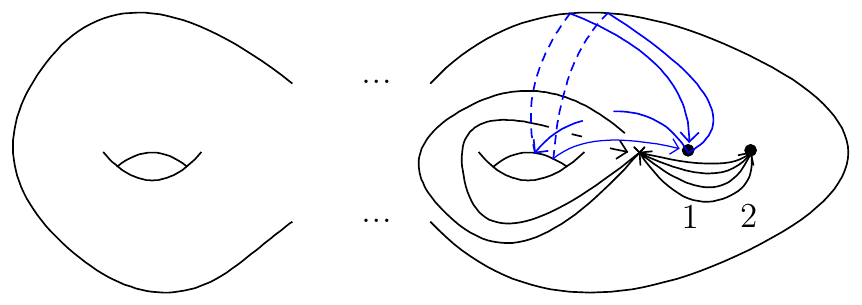}
\end{center}
By noticing that the the braid numbered by 1 passes above the one numbered by 2
 first, one can see that $K$ is homotopic to the braid $R^{1,2}R^{2,1}$.

\noindent\underline{Relation \eqref{eqn:gE2} is satisfied:} 
Relation \eqref{eqn:gE2} is more difficult to draw so we sketch the way to think 
of the right-hand-side. Align the points in a generating cycle of the genus $g$ 
surface (this means that they are in the boundary of the compactified framed 
configuration space). Then if a point travels through a cycle, its corresponding 
framing will naturally start to spin (in clockwise direction) as one can see in the 
following picture, for $g=2$ and for $g=4$
\begin{center}
\includegraphics[scale=1]{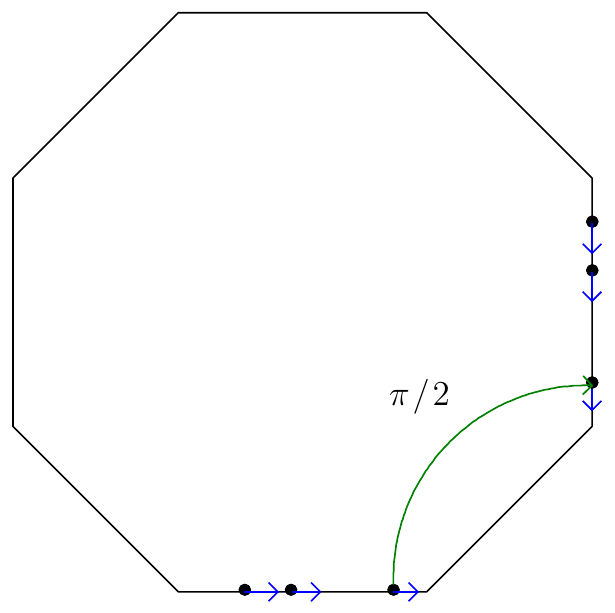} $\quad$ \includegraphics[scale=1]{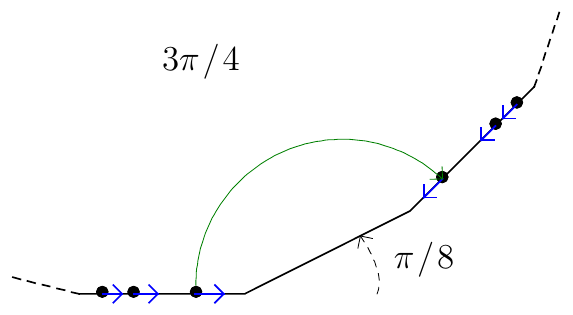}
\end{center}
If we consider a polygon with $4g$ sides corresponding to a genus $g$ surface, 
then for each marked point travelling through the generating cycles, the framing 
attached to that point will be twisted by an angle of $\pi - \frac{\pi}{g}$.

If we suppose that the marked points were chosen to be in the $A_1$-cycle of 
$\Sigma_g$, the right hand side of \eqref{eqn:gE2} can be drawn as follows, for $g=2$:
\begin{center}
\includegraphics[scale=1]{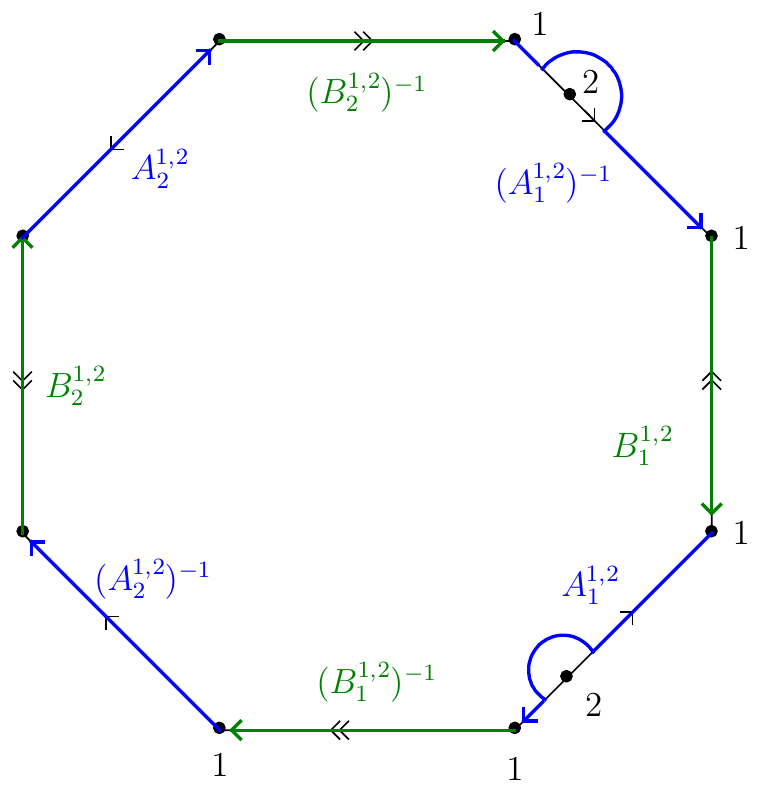}
\end{center}
In conclusion, one can then easily see that if, for a 2-point configuration, the first point travels
around all the generating cycles concerned in the right-hand-side of relation 
\eqref{eqn:gE2}, its corresponding framing data will make $2g \times \frac{(g-1)}{g} = 2(g-1)$ 
complete spins and the first point will have done a complete loop around 
the second point. This is exactly the left-hand-side of equation \eqref{eqn:gE2}.

\medskip

Thus, by the universal property of $\mathcal Q$, there is a morphism of 
$\PaB^f$-modules $\mathcal Q \to \mathbf{PaB}_g^f$, which is the identity 
on objects.
To show that this map is in fact an isomorphism, it suffices to show that it is 
an isomorphism at the level of automorphism groups of objects arity-wise, 
as all groupoids are connected. 
Let $n\geq 0$, and $p$ be the object $1\Big(\cdots(\big((n-2)\big((n-1) n\big)\big)\cdots\cdots\Big)$ of 
$\mathcal Q(n)$ and $\mathbf{PaB}_g^f(n)$. We want to show that the 
induced morphism 
$$
\on{Aut}_{\mathcal Q(n)}(p)  \longrightarrow \on{Aut}_{\mathbf{PaB}_g^f(n)}(p)
=\pi_1\left(\overline{\textrm{Conf}}^f(\Sigma_g,n),p\right)
$$
is an isomorphism. 

On the one hand, as $\overline{\textrm{Conf}}^f(\Sigma_g,n)$ is a manifold 
with corners, we are allowed to move the basepoint $p$ to the point $p_{reg}$ 
in which is based the fundamental group in subsection \ref{PBgnf}. 
We then have an isomorphism of fundamental groups 
$\pi_1(\overline{\textrm{Conf}}^f(\Sigma_g,n),p) \simeq 
\pi_1(\textrm{Conf}^f(\Sigma_g,n),p_{reg})$. 

\medskip

On the other hand, one can construct a non-symetric module $\tilde{Q}$ 
in groupoids over $\B^f$ carrying  an action of the (algebraic version of the) 
framed braid group $\on{B}^f_{g,n}$ on $\Sigma_g$ in the following sense: 
\begin{itemize}
\item for each $n\geq1$, $\tilde{Q}(n)$ is a groupoid with maximal 
parenthesizations of unnumbered elements as objects. We will make abuse of notation on still numbering these elements in order to count them.
\item $\tilde{Q}$ is freely generated by $A_a^{\bullet, \bullet}:=A_a^{1, 2}$ 
and $B_a^{\bullet, \bullet}:=B_a^{1, 2}$ in $\tilde{Q}(2)$, for all $1
    \leqslant i \leqslant g,$ satisfying relations \eqref{eqn:gR}, 
    \eqref{eqn:gD}, \eqref{eqn:gN}, \eqref{eqn:gE1} 
    and \eqref{eqn:gE2}.
\end{itemize}

In the following lemma we show that there are group morphisms 
$\on{B}_{g,n}^f \tilde{\to} \on{Aut}_{\tilde{Q}(n)}(p)\to\mathfrak S_n$, 
the left one being unique. 

\begin{lemma} \label{prop:gfra}
Let $\tilde{Q}$ be the operadic $\on{B}^f$-module with unnumbered maximal 
paranthesizations as objects and with generators $A_i^{1, 2}:=A_i^{\bullet, \bullet}$ 
and $B_i^{1, 2}:=B_i^{\bullet, \bullet}$, for all $1
    \leqslant i \leqslant g$, in $\tilde{Q}(2)$ satisfying relations 
    \eqref{eqn:gE1} and \eqref{eqn:gE2}.
Let $p$ be the object in $\tilde{Q}(n)$ given by the $n$-lenght rightmost 
maximal parenthesization $$p:=(\bullet(\bullet(\bullet(\ldots((\bullet\bullet))\ldots).$$
Then there is a unique group isomorphism 
$$\phi_n : \on{B}^f_{g,n}\to \on{Aut}_{\tilde{Q}(n)}(p),$$ 
such that, for $\Phi_i:= \Phi^{1...i-1,i,i+1...n}...\Phi^{1...n-2,n-1,n}$
\begin{itemize}
\item $X_1^a\mapsto A_a^{1,2\ldots n}$, for all $1
    \leqslant a \leqslant g$ ;
\item $Y_1^a\mapsto B_a^{1,2\ldots n}$, for all $1
    \leqslant a \leqslant g$ ;
\item $\tau_i\mapsto (\Phi_i)^{-1}R^{i,i+1}\Phi_i$ ; for all $1
    \leqslant i \leqslant n-1$ ;
\item $f_i\mapsto (\Phi_i)^{-1}F^{i,i+1}\Phi_i$, for all $1
    \leqslant i \leqslant n-1$ ;
    \item $f_n\mapsto \tilde F^{n-1,n}$,
\end{itemize}
where $A^{1,2\ldots n}\in\on{Aut}_{\tilde{Q}(n)}(p)$ is obtained from $A^{1,2}$,  
$F^{i,i+1}$ and $F^n$ are obtained from $F^{1,2}$ and 
$R^{i,i+1}\in\on{Aut}_{\tilde{Q}(n)}(p)$ is obtained from $R^{1,2}$ by some finite 
sequences of arrows involving the associator and the operadic module 
morphisms since the parenthesizations are unmarked.
\end{lemma}

In particular, by applying a finite sequence of associators one can show 
that the above lemma remains true for all possible choices of base points $p\in\tilde{Q}(n)$.

\begin{proof}
For simplicity, we omit the associativity 
constraints. One can show by induction that the image of $X^a_{i}:=\tau_{i-1}X^a_{i-1}\tau_{i-1}$ is 
$$
 R^{12...(i-1),i...n}A_a^{1,2...n} R^{i...n,12...(i-1)}
$$
therefore the image of $X^a_{1}
\cdots X^a_{i}$ is $A_a^{1...i,(i+1)...n}$. We will thus reduce to the cases $n=2,3$ in the rest of the proof.
\medskip

\noindent\underline{$\phi_n $ is a well-defined group morphism:} 
Let us first show that there is indeed such a group morphism. 
First of all, the braid relations are preserved as there are morphisms from 
$\on{B}_3$ to both groups (the first one is classic, the second one is 
induced by the fact that $\tilde{Q}$ is a $\B^f$-module).
Relation \eqref{eqn:FBG1} is preserved by naturality since $\tilde{Q}$ is a $\on{B}^f$-module.
Next, notice that, by removing the third braid in relation \eqref{eqn:gD} for $Z=A$, we obtain relation
$$
A_a^{1,2} R^{1,2}A_a^{2,1} R^{2,1}=A_a^{12}
$$
which can be depicted as follows:
\begin{center}
\begin{align}\label{D1bis}\tag{D1bis}
\begin{tikzpicture}[baseline=(current bounding box.center)]
\tikzstyle point=[circle, fill=black, inner sep=0.05cm] 
\node[point, label=above:$1$] at (1,1) {};
 \node[point, label=below:$1$] at (1,0) {};
 \node[point, label=above:$2$] at (1.5,1) {};
 \node[point, label=below:$2$] at (1.5,0) {};
 \draw[-,thick] (1,1) .. controls (1,0.5) and (1,0.5).. (0.5,0.5); 
 \draw[->,thick] (0.5,0.5) .. controls (1,0.5) and (1,0.5).. (1,0.05); 
 \draw[-,thick] (1.5,1) .. controls (1.5,0.5) and (1.5,0.5).. (1,0.5); 
 \draw[->,thick] (1,0.5) .. controls (1.5,0.5) and (1.5,0.5).. (1.5,0.05); 
\node[point, white, label=left:$A_a^{12}$] at (0.5,0.5) {};
\end{tikzpicture} 
\qquad = \qquad
\begin{tikzpicture}[baseline=(current bounding box.center)] 
\tikzstyle point=[circle, fill=black, inner sep=0.05cm]
 \node[point, label=above:$1$] at (1,4) {};
 \node[point, label=below:$1$] at (1,0) {};
 \node[point, label=above:$2$] at (1.5,4) {};
 \node[point, label=below:$2$] at (1.5,0) {};
 \draw[-,thick] (1,4) .. controls (1,3.75) and (1,3.75).. (1,3.5);
 \draw[-,thick] (1.5,4) .. controls (1.5,3.75) and (1.5,3.75).. (1.5,3.5);
 \draw[-,thick] (1,3.5) .. controls (1,3.25) and (1,3.25).. (0.5,3.25); 
 \draw[-,thick] (0.5,3.25) .. controls (1,3.25) and (1,3.25).. (1,3); 
\node[point, white, label=left:$A_a^{1,2}$] at (0.5,3.25) {};
 \draw[-,thick] (1.5,3.5) .. controls (1.5,3.25) and (1.5,3.25).. (1.5,3);
 \draw[-,thick] (1,3) .. controls (1,2.75) and (1,2.75).. (1,2.5);
 \draw[-,thick] (1.5,3) .. controls (1.5,2.75) and (1.5,2.75).. (1.5,2.5);
 \draw[-,thick] (1,2.5) .. controls (1,2.25) and (1.5,2.25).. (1.5,2);
 \node[point, ,white] at (1.25,2.25) {};
\draw[-,thick] (1.5,2.5) .. controls (1.5,2.25) and (1,2.25).. (1,2); 
 \draw[-,thick] (1,2) .. controls (1,1.75) and (1,1.75).. (1,1.5);
 \draw[-,thick] (1.5,2) .. controls (1.5,1.75) and (1.5,1.75).. (1.5,1.5);
 \draw[-,thick] (1,1.5) .. controls (1,1.25) and (1,1.25).. (0.5,1.25); 
 \draw[-,thick] (0.5,1.25) .. controls (1,1.25) and (1,1.25).. (1,1); 
\node[point, white, label=left:$A_a^{2,1}$] at (0.5,1.25) {};
 \draw[-,thick] (1.5,1.5) .. controls (1.5,1.25) and (1.5,1.25).. (1.5,1);
 \draw[-,thick] (1,1) .. controls (1,0.85) and (1,0.85).. (1,0.5);
 \draw[-,thick] (1.5,1) .. controls (1.5,0.85) and (1.5,0.85).. (1.5,0.5);
 \draw[->,thick] (1,0.5) .. controls (1,0.25) and (1.5,0.25).. (1.5,0.05); 
 \node[point, ,white] at (1.25,0.25) {};
 \draw[->,thick] (1.5,0.5) .. controls (1.5,0.25) and (1,0.25).. (1,0.05); 
\end{tikzpicture}
 \end{align}
\end{center}

Now, in Lemma \ref{LE1}, we prove that relation \eqref{eqn:gD} for $Z=A$ is equivalent to 
\begin{equation}\label{eqn:Nbis2}
A_a^{12, 3}= \Phi^{1, 2, 3} A_a^{1, 23} (\Phi^{1, 2, 3})^{- 1} R^{1, 2}
   \Phi^{2, 1, 3} A_a^{2, 13}  (\Phi^{2, 1, 3})^{- 1} R^{2, 1}.
\end{equation}

Again, by naturality since $\tilde{Q}$ is a $\on{B}^f$-module, we have
 $A_a^{12,3} R^{1,2}=R^{1,2}A_a^{21,3}$ so that replacing $A_a^{12,3}$ and 
 $A_a^{21,3}$ by the r.h.s. of equation \eqref{eqn:Nbis2}, this implies that
$$
(A_a^{1,23},R^{1,2}A_a^{2,13} R^{2,1} )=1.
$$
Removing the third strand in this equation implies that relation $(X^a_1,X^a_2)=1$ 
is preserved by $\phi_n$. The same reasoning applies to $\eqref{eqn:gD}$ for $Z=B$ and 
implies that $(Y^a_1,Y^a_2)=1$ is also preserved by $\phi_n$. In the same way, 
we obtain relation \eqref{eqn:FBG4} directly from \eqref{eqn:gN} since our operadic modules are 
pointed so we can remove
 the third strand.
Next, relation \eqref{eqn:FBG3} is satisfied as we can retrieve the 
third strand in \eqref{eqn:gE1} to obtain the desired relation.
Finally, relation \eqref{eqn:FBG5} is obtained directly from \eqref{eqn:gE2}.
Thus, we have a group morphism. Let us show that it is bijective.
\medskip

\noindent\underline{$\phi_n $ is surjective:}
The fact that the map $\phi_n $ is surjective is a consequence of the fact that all 
the defining relations in $\tilde{Q}(n)$ come from the defining relations of 
$\on{B}^f_{g,n}$ and the operadic module partial compositions.
\medskip

\noindent\underline{$\phi_n $ is injective:}
Let us now show the injectivity of this map. Let $\bar{Q}$ be the operad module 
with same objects as $\tilde{Q}$ and, for every object $p$ of $\bar{Q}(n)$, we 
define $\on{Aut}_{\bar{Q}(n)}(p):=B^f_{g,n}$. Next we have a map 
$\tilde{Q}\to \bar{Q}$ sending the generators $A_{a}^{1,2}$ to $X_{1}^{a}$ and 
$B_{a}^{1,2}$ to $Y_{1}^{a}$ in $B_{g,2}^f$. Indeed, the fact that both modules 
are pointed and since the following relations hold in $B_{g,2}^f$ and $B_{g,3}^f$
\begin{flalign*}
&(X_{1}^{a}\tau_{1}\tau_{2})^{3}=X_{123}^{a}, \\
&  (Y_{1}^{a}\tau_{1}\tau_{2})^{3}=Y_{123}^{a},\\ 
&  (\tau_{1}X_{1}^{a}\tau_{1},Y_{1}^{a}) = \tau_{1}^{2}, \\
& (X_1^a,X_2^b)=(X_1^a,Y_2^b)=(Y_1^a,X_2^b)=(Y_1^a,Y_2^b)=1, \\
& \prod_{a=1}^g ((X_{1}^{a})^{-1},Y_{1}^{a} )=\tau_1^2f_1^{2(g-1)}, &
\end{flalign*}
we show that relations \eqref{eqn:gR}, \eqref{eqn:gD}, 
\eqref{eqn:gN},  \eqref{eqn:gE1} and \eqref{eqn:gE2} are preserved.

Then, as $\PaB^f$ acts on both of these operadic modules we conclude 
that there is a map $\on{Aut}_{\tilde{Q}(n)}(p) \to \on{Aut}_{\bar{Q}(n)}(p) $. 
In order to prove the injectivity of $\phi$, we are left to prove that the composite 
$$
\on{B}^f_{g,n}\to \on{Aut}_{\tilde{Q}(n)}(p) \to \on{Aut}_{\bar{Q}(n)}(p) 
$$
is the identity morphism, which is true by construction of both maps.
\end{proof}

\textit{End of the proof of Theorem \ref{Thm:PaBfg}.} 

In the same way the collection $\{\on{PB}^f_{g,n}\}_{n\geq 1}$ of pure 
genus $g$ braids owns a non-symmetric module over the non-symmetric 
operad $\mathbf{PB}^f$, constructed in Section \ref{PBf}, denoted
 $\mathbf{PB}^f_g$.

Moreover, one the forgetful map $\mathbf{Op}\mathcal{C}\to\mathbf{NsOp}\mathcal{C}$ between 
the category of operads in $\mathcal{C}$ and the category of non-symmetric operads in $\mathcal{C}$ 
 induces a map $ Q\to\tilde{Q} $. Then, one has by constuction of $\tilde{Q}$ 
 that $\on{Aut}_{\mathcal Q(n)}(p)$ is the kernel of the map 
 $\on{Aut}_{\tilde{Q}(n)}([p]) \to \mathfrak{S}_n$. 
One can actually show that we have a commuting diagram 
$$
\xymatrix{
 \on{PB}^f_{g,n} \ar[r]^-{\simeq} \ar[d] &  \on{Aut}_{Q(n)}(p) \ar[r] \ar[d] & 
\pi_1\left(\overline{\textrm{Conf}}^f(\Sigma_g,n),p\right) \ar[d] & \ar[l]_-{\simeq}
 \pi_1\left(\textrm{Conf}^f(\Sigma_g,n),p_{reg}\right)\ar[d] \\ 
 \on{B}^f_{g,n} \ar[r]^-{\simeq} \ar[d] &  \on{Aut}_{\tilde{Q}(n)}([p]) \ar[r] \ar[d] & 
\pi_1\left(\overline{\textrm{Conf}}^f(\Sigma_g,n)/\mathfrak S_n,[p] \right)  \ar[d] & \ar[l]_-{\simeq}
\pi_1\left(\textrm{Conf}^f(\Sigma_g,n)/\mathfrak S_n,[p_{reg}]\right) \ar[d] \\
\mathfrak{S}_n \ar@{=}[r]& \mathfrak{S}_n\ar@{=}[r] & \mathfrak{S}_n  \ar@{=}[r]  & \mathfrak{S}_n 
}
$$
where all vertical sequences are short exact sequences. 
Thus, in order to show that the map $\on{Aut}_{\mathcal Q(n)}(p) \to 
\pi_1\left(\overline{\textrm{Conf}}^f(\Sigma_g,n),p\right)$ is an isomorphism, 
we are left to show that 
$$
 \phi:\on{B}^f_{g,n} \longrightarrow 
 \pi_1\left(\textrm{Conf}^f(\Sigma_g,n)/\mathfrak S_n,[p_{reg}]\right)
$$
is indeed an isomorphism. This is the case since the map $ \phi$ is exactly 
the isomorphism constructed in \cite[Theorem 13]{BeG}.
This completes the proof of Theorem \ref{Thm:PaBfg}.
\end{proof}

\subsection{An alternative presentation of $\PaB_g^f$}
\label{An alternative presentation}

In this subsection we exhibit an alternative presentation for the module $\PaB_g^f$. 
This will be then used in the particular case $g=1$ to show that the set of genus 
$1$ non-reduced associators over $\C$ is not empty.

Let $\PaB^{bis}_{g,f}$ be the $\PaB^f$-module having $\Pa$ as $\Pa$-module of objects and freely 
generated by morphisms $\tilde A_a^{1,2}$ and $\tilde B_a^{1,2}$ in arity 2 with relations
\begin{flalign}
& \tilde Z_a^{1,\emptyset}=\on{Id}_{1}\,, 
\label{eqn:gRbis} \\
& \tilde Z_a^{123} \tilde Z_a^{12,3}=\Phi^{1,2,3}\tilde Z_a^{1,23}(\Phi^{1,2,3})^{-1} 
(R^{2,1})^{-1}\Phi^{2,1,3}\tilde Z_a^{2,13}(\Phi^{2,1,3})^{-1}(R^{1,2})^{-1}\,, 
\label{eqn:gDbis} \\
& \on{Id}^{12,3} =\left(\tilde Z_a^{12,3}, (\tilde Z_b^{12,3})^{-1}\Phi^{1,2,3}\tilde Z_b^{1,23}R^{2,3}R^{3,2}(\Phi^{1,2,3})^{-1}   \right)\,,	 
\label{eqn:gNbis}  \\
& \Phi^{1,2,3}R^{2,3}R^{3,2}(\Phi^{1,2,3})^{-1}=\left((\tilde A_a^{12,3})^{-1}, 
\tilde B_a^{12,3}\Phi^{1,2,3}(\tilde B_a^{1,23})^{-1}(\Phi^{1,2,3})^{-1} \right)\, 
\label{eqn:gE1bis} \\
& R^{1,2}R^{2,1}(F^{1,2})^{2(g-1)}=\prod_{a=1}^g\left(\tilde A_a^{1,2},(\tilde B_a^{1,2})^{-1}\right). \label{eqn:gE2bis} &	
\end{flalign}

\begin{proposition}\label{PaB:g:f:bis}
As $\PaB$-modules in groupoids having $\Pa$ as $\Pa$-module of objects, 
$\PaB_g^f$ and $\PaB^{bis}_{g,f}$ are isomorphic.
\end{proposition}

\begin{proof}
In $\PaB_g^f$ there are morphisms $\tilde A^{1,2}=R^{1,2}(A^{2,1})^{-1}
(R^{1,2})^{-1},\tilde B^{1,2}=R^{1,2}(B^{2,1})^{-1}(R^{1,2})^{-1}$. These 
correspond topologically to moving the point indexed by 2 in the direction of the generating cycles of $\Sigma_g$.
The fact that the assignment $\Pa \mapsto \Pa$, $(\tilde A^{1,2},\tilde B^{1,2}) 
\mapsto (\tilde A^{1,2},\tilde B^{1,2})$ defines a morphism of $\PaB$-modules 
$\PaB^{bis}_{1} \longrightarrow \PaB_{1}$ is justified topologically by the fact 
that, in $\PaB_g^f$, the paths representing the l.h.s. and r.h.s. of relations 
\eqref{eqn:gRbis}, \eqref{eqn:gDbis}, \eqref{eqn:gNbis},  \eqref{eqn:gE1bis} and 
\eqref{eqn:gE2bis} are homotopic.

In order to prove that this morphism is an isomorphism, let us show that 
\eqref{eqn:gR}, \eqref{eqn:gD}, \eqref{eqn:gN},  \eqref{eqn:gE1} and 
\eqref{eqn:gE2} are equivalent to \eqref{eqn:gRbis}, \eqref{eqn:gDbis}, \eqref{eqn:gNbis},  \eqref{eqn:gE1bis} and 
\eqref{eqn:gE2bis}.

First of all, as we have $\tilde Z_a^{1,2}=R^{1,2}(Z_a^{2,1})^{-1}(R^{1,2})^{-1}$, 
then one can readily see that \eqref{eqn:gR} is equivalent to \eqref{eqn:gRbis}.

Second of all, let's prove that \eqref{eqn:gD} is equivalent to \eqref{eqn:gDbis}. 

On the one hand, in Lemma \ref{LE1}
we showed that relation \eqref{eqn:gD} is equivalent to 
\begin{equation}
Z_a^{12, 3}= \Phi^{1, 2, 3} Z_a^{1, 23} (\Phi^{1, 2, 3})^{- 1} R^{1, 2}
   \Phi^{2, 1, 3} Z_a^{2, 13}  (\Phi^{2, 1, 3})^{- 1} R^{2, 1}.
\end{equation}

On the other hand, doubling one of the strands, applying a permutation in the 
equality $\tilde Z_a^{1,2}=R^{1,2}(Z_a^{2,1})^{-1}(R^{1,2})^{-1}$ and plugging 
the result into relation \eqref{eqn:gDbis}, one obtains
$$
(Z_a^{123})^{-1}R^{12,3} (Z_a^{3,12})^{-1}(R^{12,3})^{-1}=\Phi^{1,2,3}R^{1,23} 
(Z_a^{23,1})^{-1}(R^{1,23})^{-1}(\Phi^{1,2,3})^{-1} 
$$
$$
(R^{2,1})^{-1}\Phi^{2,1,3}R^{2,13} (Z_a^{13,2})^{-1}(R^{2,13})^{-1}(\Phi^{2,1,3})^{-1}(R^{1,2})^{-1}.
$$
The inverse of this equation reads
$$
R^{12,3} Z_a^{3,12}(R^{12,3})^{-1}Z_a^{123}
$$
$$
=R^{1,2}\Phi^{2,1,3}R^{2,13}Z_a^{13,2}(R^{2,13})^{-1}(\Phi^{2,1,3})^{-1}R^{2,1}
\Phi^{1,2,3}R^{1,23}Z_a^{23,1}(R^{1,23})^{-1}(\Phi^{1,2,3})^{-1},
$$
which is equivalent to
$$
Z_a^{123}=R^{12,3} (Z_a^{3,12})^{-1}(R^{12,3})^{-1}R^{1,2}\Phi^{2,1,3}R^{2,13}Z_a^{13,2}(R^{2,13})^{-1}
$$
$$
(\Phi^{2,1,3})^{-1}R^{2,1}\Phi^{1,2,3}R^{1,23}Z_a^{23,1}(R^{1,23})^{-1}(\Phi^{1,2,3})^{-1}.
$$
Now, using 
\begin{itemize}
\item $R^{12,3} (Z_a^{3,12})^{-1}(R^{12,3})^{-1}R^{1,2}=R^{1,2}R^{21,3} (Z_a^{3,21})^{-1}(R^{21,3})^{-1}$,
\item $(Z_a^{3,12})^{-1} = R^{3,21} (Z_a^{213})^{-1}Z_a^{21,3}R^{21,3} $,
\item $Z_a^{13,2} = (R^{2,13})^{-1} (Z_a^{2,13})^{-1}Z_a^{213}(R^{13,2})^{-1}$,
\item $Z_a^{23,1} = (R^{1,23})^{-1} (Z_a^{1,23})^{-1}Z_a^{123}(R^{23,1})^{-1}$,
\end{itemize}
we deduce 
$$
Z_a^{123}=R^{1,2}R^{21,3}R^{3,21} (Z_a^{213})^{-1}Z_a^{21,3}\Phi^{2,1,3}(Z_a^{2,13})^{-1}
Z_a^{213}(R^{13,2})^{-1}(R^{2,13})^{-1}
$$
$$
(\Phi^{2,1,3})^{-1}R^{2,1}\Phi^{1,2,3}(Z_a^{1,23})^{-1}Z_a^{123}(R^{23,1})^{-1}(R^{1,23})^{-1}(\Phi^{1,2,3})^{-1}.
$$
Now, the elements $Z_a^{123}$ and $Z_a^{213}$ can, after switable permutations 
and conjugation with associators, be moved to the rightmost part of the r.h.s. 
of the above equation and cancel out. Using $R^{1,2}R^{21,3}R^{3,21} (Z_a^{213})^{-1}
=Z_a^{12,3}R^{1,2}R^{21,3}R^{3,21}$, we obtain the following equation
$$
Z_a^{12,3}=\Phi^{1, 2, 3} R^{1,23}R^{23,1} (\Phi^{1, 2, 3})^{- 1} \Phi^{1, 2, 3} Z_a^{1, 23}
 (\Phi^{1, 2, 3})^{- 1}(R^{2,1})^{-1}\Phi^{2,1,3}
$$
$$
R^{2,13}R^{13,2}Z_a^{2,13}(\Phi^{2,1,3})^{-1}(R^{1,2})^{-1}(R^{3,12})^{-1} (R^{12,3})^{-1}.
$$
Now by using $R^{2,13}R^{13,2}Z_a^{2,13}=Z_a^{2,13}R^{2,13}R^{13,2}$ and $\Phi^{1, 2, 3} 
(R^{23,1})^{-1}(R^{1,23})^{-1} (\Phi^{1, 2, 3})^{- 1}Z_a^{12,3}=Z_a^{12,3}\Phi^{1, 2, 3}
(R^{23,1})^{-1}(R^{1,23})^{-1}(\Phi^{1, 2, 3})^{- 1}$, we now  left to compute the expression
\begin{equation}\label{T123}
R^{2,13}R^{13,2}(\Phi^{2,1,3})^{-1}(R^{1,2})^{-1}(R^{3,12})^{-1} (R^{12,3})^{-1}
\Phi^{1, 2, 3} R^{1,23}R^{23,1} (\Phi^{1, 2, 3})^{- 1}.
\end{equation}
By using suitable hexagon relations we obtain
\begin{itemize}
\item $R^{2,13}R^{13,2} =(\Phi^{2,1,3})^{-1}R^{2,1}\Phi^{1, 2, 3}R^{2,3}R^{3,2}
(\Phi^{1, 2, 3})^{- 1}R^{1,2}\Phi^{2,1,3}$,
\item $(R^{3,12})^{-1} (R^{12,3})^{-1} =$
$$\Phi^{1, 2, 3} (R^{3,2})^{-1}(R^{2,3})^{-1}(\Phi^{1, 2, 3})^{- 1}(R^{2,1})^{-1}
\Phi^{2,1,3} (R^{3,1})^{-1}(R^{1,3})^{-1}(\Phi^{2,1,3})^{-1}R^{2,1},$$
\item $ \Phi^{1, 2, 3} R^{1,23}R^{23,1} (\Phi^{1, 2, 3})^{- 1} =R^{1,2}\Phi^{2,1,3}
R^{1,3}R^{3,1}(\Phi^{2,1,3})^{-1}R^{2,1} $.
\end{itemize}
By using the above relations we deduce that \eqref{T123} equals $
(\Phi^{2,1,3})^{-1}R^{2,1}R^{1,2}R^{2,1},
$ so that 
$
Z_a^{12,3}=\Phi^{1, 2, 3} Z_a^{1, 23} (\Phi^{1, 2, 3})^{- 1}(R^{2,1})^{-1}\Phi^{2,1,3}
Z_a^{2,13}(\Phi^{2,1,3})^{-1}R^{2,1}R^{1,2}R^{2,1}.
$
Now, in the automorphism group of $(21)3$, we have $$\Phi^{2,1,3}Z_a^{2,13}
(\Phi^{2,1,3})^{-1}R^{2,1}R^{1,2}=R^{2,1}R^{1,2}
\Phi^{2,1,3}Z_a^{2,13}(\Phi^{2,1,3})^{-1}.$$ In conclusion, we obtain 
$$
Z_a^{12,3}=\Phi^{1, 2, 3} Z_a^{1, 23} (\Phi^{1, 2, 3})^{- 1}R^{1,2}
\Phi^{2,1,3}Z_a^{2,13}(\Phi^{2,1,3})^{-1}R^{2,1},
$$
which is precisely equation \eqref{eqn:gD}. 

Third of all, let us assume relations \eqref{eqn:gRbis} and \eqref{eqn:gDbis} and let us 
prove that \eqref{eqn:gE1bis} is equivalent to \eqref{eqn:gE1} and that \eqref{eqn:gNbis} is equivalent to \eqref{eqn:gN}. 
Relation \eqref{eqn:gDbis} for $\tilde Z=\tilde B$ is equivalent to
\[ \Phi^{1, 2, 3}  (\tilde{B}^{1, 23})^{- 1}
   (\Phi^{1, 2, 3})^{- 1} \tilde{B}^{12, 3} = (R^{2,1})^{-1}\Phi^{2, 1, 3}  \tilde{B}^{2, 13} 
   (\Phi^{2, 1, 3})^{- 1} (R^{1,2})^{-1}(\tilde B^{123})^{-1}. \]
Thus, \eqref{eqn:gE1bis} is equivalent to 
\begin{eqnarray*}
\Phi R^{2,3}R^{3,2} \Phi^{- 1} & = & (\tilde A_a^{12,3})^{-1}(R^{2,1})^{-1}\Phi^{2, 1, 3}  \tilde{B}_a^{2, 13} 
   (\Phi^{2, 1, 3})^{- 1} (R^{1,2})^{-1}(\tilde B_a^{123})^{-1}\\
& & \tilde A_a^{12,3}R^{1,2}\Phi^{2, 1, 3}  (\tilde{B}_a^{2, 13} )^{-1}
   (\Phi^{2, 1, 3})^{- 1} R^{2,1}\tilde B_a^{123}
\end{eqnarray*}
Now, as $B_a^{123}$ commutes with all elements in this equation, we simplify it 
and, as we have $(R^{1,2})^{-1}\tilde A_a^{12,3}=\tilde A_a^{21, 3} (R^{1,2})^{-1}$, 
we deduce that \eqref{eqn:gE1bis} is equivalent to
$$
\Phi R^{2,3}R^{3,2} \Phi^{- 1}  =  (\tilde A_a^{12,3})^{-1} (R^{2,1})^{-1}\Phi^{2, 1, 3}
 \tilde{B}_a^{2, 13} (\Phi^{2, 1, 3})^{- 1} \tilde A_a^{21,3}\Phi^{2, 1, 3}   (\tilde{B}_a^{2, 13} )^{-1}
   (\Phi^{2, 1, 3})^{- 1} R^{2,1}.
$$
which is equivalent to
$$
 (\Phi^{2, 1, 3})^{- 1}\tilde A_a^{21,3} R^{2,1}\Phi R^{2,3}R^{3,2} \Phi^{- 1}(R^{2,1})^{-1} 
 \Phi^{2, 1, 3} = \tilde{B}_a^{2, 13} (\Phi^{2, 1, 3})^{- 1} \tilde A_a^{21,3}\Phi^{2, 1, 3}   (\tilde{B}_a^{2, 13} )^{-1}.
$$
Now, by plugging $ \tilde A_a^{21,3}=R^{21,3}(A_a^{3,21})^{-1}(R^{21,3})^{-1}$ and 
$\tilde{B}_a^{2, 13}=R^{2,13}(B_a^{13,2})^{-1}(R^{2,13})^{-1}$ in the above equation and using
\begin{itemize}
\item $  \Phi^{2, 1, 3}  =  R^{21,3}(\Phi^{3,2,1})^{-1}
(R^{2,3})^{-1} \Phi^{2, 3, 1} (R^{1,3})^{-1}  $,
\item $R^{2,1}\Phi R^{2,3}R^{3,2} \Phi^{- 1}(R^{2,1})^{-1} \Phi^{2, 1, 3}  =R^{21,3}(\Phi^{3,2,1})^{-1}
R^{3,2} \Phi^{2, 3, 1} (R^{1,3})^{-1} $,
\item $ B_a^{13,2}=(R^{2,13})^{-1}(B_a^{2,13})^{-1}B_a^{132} (R^{13,2})^{-1}$,
\end{itemize}
we obtain
$$
R^{1,3}(\Phi^{2, 3, 1})^{-1}R^{2,3}\Phi^{3,2,1}(A_a^{3,21})^{-1}(\Phi^{3,2,1})^{-1}
R^{3,2} \Phi^{2, 3, 1} (R^{1,3})^{-1}=R^{2,13}R^{13,2}B_a^{2,13}R^{1,3}(\Phi^{2, 3, 1})^{-1}
$$
$$
R^{2,3}\Phi^{3,2,1}
(A_a^{3,21})^{-1}(\Phi^{3,2,1})^{-1}
(R^{2,3})^{-1} \Phi^{2, 3, 1} (R^{1,3})^{-1} (B_a^{2,13})^{-1}(R^{13,2})^{-1}(R^{2,13})^{-1}.
$$
After suitable permutations of the indices and conjugations with associators, one 
can move the $R^{21,3}R^{13,2}$ term in the r.h.s of the above equation to the
rightmost part, where it cancels out.
We obtain, by performing the permutation $(123) \mapsto (312)$ that
$$
R^{3,2}(\Phi^{1,2,3})^{-1}R^{1,2}\Phi^{2,1,3}(A_a^{2,13})^{-1}(\Phi^{2,1,3})^{-1}
R^{2,1} \Phi^{1,2,3} (R^{3,2})^{-1}
$$
$$
=B_a^{1,32}R^{3,2}(\Phi^{1,2,3})^{-1}R^{1,2}\Phi^{2,1,3}
(A_a^{2,13})^{-1}(\Phi^{2,1,3})^{-1}
(R^{1,2})^{-1} \Phi^{1,2,3} (R^{3,2})^{-1} (B_a^{1,32})^{-1}.
$$
Now, using $B_a^{1,32}R^{3,2}=R^{3,2}B_a^{1,23}$ this equation can be rewritten
$$
R^{1,2}\Phi^{2,1,3}(A_a^{2,13})^{-1}(\Phi^{2,1,3})^{-1}
R^{2,1} 
$$
$$
=\Phi^{1,2,3}B_a^{1,23}(\Phi^{1,2,3})^{-1}R^{1,2}\Phi^{2,1,3}
(A_a^{2,13})^{-1}(\Phi^{2,1,3})^{-1}
(R^{1,2})^{-1} \Phi^{1,2,3}  (B_a^{1,23})^{-1}(\Phi^{1,2,3})^{-1}.
$$
This is equivalent to 
$$
R^{1,2}\Phi^{2,1,3}A_a^{2,13}(\Phi^{2,1,3})^{-1}(R^{1,2})^{-1} \Phi^{1,2,3}  
(B_a^{1,23})^{-1}(\Phi^{1,2,3})^{-1}R^{1,2}\Phi^{2,1,3}
$$
$$
(A_a^{2,13})^{-1}(\Phi^{2,1,3})^{-1}
R^{2,1} \Phi^{1,2,3}B_a^{1,23}(\Phi^{1,2,3})^{-1}=\on{Id}_{(12)3}.
$$
Taking the inverse of this relation and multiplying it by $R^{1,2}R^{2,1}$ in the 
rightmost part of each side we obtain
$$
R^{1,2}R^{2,1}=\Phi^{1,2,3}  (B_a^{1,23})^{-1}\Phi^{1,2,3})^{-1}\Phi^{2,1,3}A_a^{2,13}(\Phi^{2,1,3})^{-1}(R^{1,2})^{-1} \Phi^{1,2,3}B_a^{1,23}
$$
$$
(\Phi^{1,2,3})^{-1}R^{1,2}\Phi^{2,1,3}(A_a^{2,13})^{-1}(\Phi^{2,1,3})^{-1}
R^{2,1}.
$$
Finally, as $\Phi^{1,2,3}  (B_a^{1,23})^{-1}(\Phi^{1,2,3})^{-1}$ comutes with
 $R^{1,2}R^{2,1}$ in the automorphism group of $(12)3$, we obtain 
$$
R^{1,2}\Phi^{2,1,3}A_a^{2,13}(\Phi^{2,1,3})^{-1}R^{2,1}\Phi^{1,2,3}B_a^{1,23}(\Phi^{1,2,3})^{-1}
$$
$$
(R^{1,2})^{-1}(\Phi^{2,1,3})^{-1}(A_a^{2,13})^{-1}\Phi^{2,1,3}(R^{2,1})^{-1}\Phi^{1,2,3}
(B_a^{1,23})^{-1}(\Phi^{1,2,3})^{-1}=R^{1,2}R^{2,1},
$$
which is precisely equation \eqref{eqn:gE1}. 
One can then eventually obtain the equivalence between equations 
\eqref{eqn:gNbis} and \eqref{eqn:gN} using the same procedure

Last of all, after plugging $ (A_a^{1,2})^{-1}=(R^{2,1})^{-1} \tilde A_a^{2,1} R^{2,1} $ and 
$B_a^{1,2} =(R^{2,1})^{-1} (\tilde B_a^{2,1})^{-1} R^{2,1}$ into equation \eqref{eqn:gE2}, we obtain
\begin{eqnarray*}
R^{1,2}R^{2,1} & = & \left((R^{2,1})^{-1} \tilde A_a^{2,1} R^{2,1},(R^{2,1})^{-1} (\tilde B_a^{2,1})^{-1} R^{2,1}  \right)\\
& = & (R^{2,1})^{-1} \left( \tilde A_a^{2,1},  (\tilde B_a^{2,1})^{-1}  \right)R^{2,1}\\
R^{2,1}R^{1,2} & = & (R^{1,2})^{-1} \left( \tilde A_a^{1,2} , (\tilde B_a^{1,2})^{-1}  \right)R^{1,2}\\
R^{1,2}R^{2,1} & = & \left( \tilde A_a^{1,2},  (\tilde B_a^{1,2})^{-1}  \right)
\end{eqnarray*}
which is precisely equation \eqref{eqn:gE2bis}.
\end{proof}

\begin{remark}
One can also notice that for $g=1$, removing the third strand in relation \eqref{eqn:gE1} 
implies relation  \eqref{eqn:gE2bis} and removing the first strand in relation \eqref{eqn:gE1bis} implies relation \eqref{eqn:gE2}.
\end{remark}

\subsection{Genus $g$ Grothendieck--Teichm\"uller groups}

Let us finish this section by defining Grothendieck-Teichm\"uller groups in 
genus $g$ operadicly, and then making explicit descriptions of this groups.

\begin{definition}
The ($\kk$-prounipotent version of the) \textit{genus $g$ 
Grothendieck--Teichm\"uller group
} is defined as the group 
$$
\widehat\GT^f_g(\kk):=\on{Aut}_{\tmop{OpR}\mathbf{Grpd}_\kk}^+(\widehat\PaB^f(\KK),\widehat\PaB^f_g(\KK))
$$ 
 of couples $(F,G)$ where $F\in\widehat\GT^f(\kk)$ and $G$ is an automorphism of the $\widehat\PaB^f(\KK)$-module $\widehat\PaB^f_g(\KK)$ 
 which is the identity on objects and which is compatible with $G$.
\end{definition}

The presentation of $\PaB^f_g$ then implies the following: each 
automorphism $F$ of $\PaB^f_g$ compatible with an automorphism 
$G$ of $\PaB^f$ is uniquely determined by $(\lambda,f)\in \widehat {\on{GT}}(\KK)$ such that 
\begin{itemize}
\item $F(R^{1,2})=(R^{1,2}R^{2,1})^{ \nu}R^{1,2}$,
\item$ F(\Phi^{1,2,3})=f(x,y) \cdot \Phi^{1,2,3}  $,
\item $G(A_a^{1,2})=g_a^{1,2}(X_1^a,X_2^a,Y_1^a,Y_2^a,f_1, f_2 ; 1 \leq a \leq g)$,
\item $G(B_a^{1,2})=h_a^{1,2}(X_1^a,X_2^a,Y_1^a,Y_2^a,f_1, f_2 ; 1 \leq a \leq g)$,
\end{itemize}
where  $\nu=\frac{\lambda-1}{2}$ and $g_a^{1,2},h_a^{1,2}\in \widehat{\on{PB}}^f_{g,2}(\KK)$, for $1\leq a\leq g$. 
These elements satisfy the following relations, induced by \eqref{eqn:gR}, 
\eqref{eqn:gD}, \eqref{eqn:gN}, \eqref{eqn:gE1} and 
\eqref{eqn:gE2}:
\begin{equation} \label{def:GTg:0}
g_a^{\emptyset,1}=1 \quad h_a^{\emptyset,1}=1,
\end{equation}
\begin{equation} \label{def:GTg:1}
(f(\tau_1^2,\tau_2^2)
g_a^{1,2}(\tau_1\tau_2^2\tau_1)^{ \frac{\lambda-1}{2}}\tau_2\tau_1)^3=g_a^{(12)3}
\end{equation}
\begin{equation} \label{def:GTg:2}
(f(\tau_1^2,\tau_2^2)
h_a^{1,2}(\tau_1\tau_2^2\tau_1)^{ \frac{\lambda-1}{2}}\tau_2\tau_1)^3=h_a^{(12)3}
\end{equation}
\begin{equation} \label{def:GTg:3}
u^{2} = (ug_a^{1,2}u, h_a^{1,2}) 
\end{equation}
\begin{equation} \label{def:GTg:N}
1 = (g_a^{1,2}, ug_b^{1,2}u)=(g_a^{1,2}, uh_b^{1,2}u)=(h_a^{1,2}, uh_b^{1,2}u)
=(h_a^{1,2}, ug_b^{1,2}u)
\end{equation}
(identities in $\widehat{\on{B}}^f_{g,3}(\kk)$) where $u=f(\tau_1^2,
\tau_2^2)^{-1}\tau_1^{\lambda} f(\tau_1^2,\tau_2^2)$ and
\begin{equation} \label{def:GTg:4}
\tau_1^{2\lambda}f_1^{2\lambda(g-1)} = \prod_{a=1}^g ((g_a^{1,2})^{-1},h_a^{1,2}).
\end{equation}
The image of the composition in $\widehat\GT^f_g(\kk)$  is given by
\begin{equation*}
(\lambda_1, f_1)(\lambda_2, f_2)=
(\lambda_1\lambda_2, 
 f_1(x^{\lambda_2},f_2(x, y)y^{\lambda_2} f_2(x, y)^{-1}) f_2(x, y))
\end{equation*}
and, for $\nu:=\frac{\lambda-1}{2}$, by
\begin{eqnarray*}
g_a^{1,2} \cdot \tilde{g}_a^{1,2} &=&
g_a^{1,2}(\tilde{g}_1^{1,2},\tau_1^{\nu}\tilde{g}_1^{1,2}\tau_1^{\nu},\tilde{h}_2^{1,2},
\tau_1^{\nu}\tilde{h}_1^{1,2}\tau_1^{\nu},\ldots, \tilde{h}_g^{1,2},\tau_1^{\nu}
\tilde{h}_g^{1,2}\tau_1^{\nu}, \tau_1^{\lambda},\tau_2^{\lambda}).
\end{eqnarray*}

\subsection{Horizontal framed genus $g$ chord diagrams and genus $g$ associators}

\subsubsection{The $\CD^f(\kk)$-module of genus $g$ framed chord diagrams}

Let $\mathfrak{t}^f_{g,n}(\kk)$ denote the graded Lie algebra over $\kk$ generated by 
$t_{ij}$, $1\leq i, j\leq n$, $x_i^a,y_i^a$ for $1\leq i\leq n, 1\leq a\leq g$ with relations  
\eqref{eqn:fS},  \eqref{eqn:fL},  \eqref{eqn:f4T} and the following additional genus 
$g$  relations
\begin{flalign}
& [x_i^a,y_j^b]=\delta_{ab}t_{ij} \quad \text{ for all } i\neq j,   \tag{S$_g$} \label{eqn:Sg} \\
& [x_i^a,x_j^b]=0= [y_i^a,y_j^b] \quad \text{ for all } i\neq j,   \tag{N$_g$} \label{eqn:Ng} \\
& \underset{a=1}{\overset{g}{\sum}} [x_i^a,y_i^a]=-\sum_{j:j\neq i}t_{ij}-2(g-1)t_{i i}, \tag{FT$_g$} \label{eqn:FTg} \\
& [x_k^a,t_{ij}] = [y_k^a,t_{ij}]=0 \quad  \text{ if }\{i,j\} \cap \{k\}=\emptyset,    \tag{FL$_g$} \label{eqn:FLg} \\
& [x_i^a+x_j^a,t_{ij}] = [y_i^a+y_j^a,t_{ij}] =0 \quad  \text{ for all } i, j.    &    \tag{F4T$_g$} \label{eqn:F4Tg} 
\end{flalign}

The Lie algebra $\t^f_{g,n}$ is acted on by the symmetric group $\mathfrak{S}_n$. 
One can show that the $\mathfrak{S}$-module in $grLie_\kk$ 
$$
\mathfrak{t}^f_g (\kk):=\{\mathfrak{t}^f_{g,n} (\kk)\}_{n\geq0}
$$
is a $\t^f(\KK)$-module in $grLie_{\KK}$. Partial compositions are defined as follows: 
for $I,J$ two finite sets and $k\in I$, 

$$
\begin{array}{ccccccc}
\circ_k : & \t^f_{g,I}(\KK) \oplus \t^f_J(\KK)  & \longrightarrow & \t^f_{g,J\sqcup I-\{i\}}(\KK) \\
    & (0,t_{\alpha \beta}) & \longmapsto & t_{\alpha\beta} \\
 & (t_{ij},0) & \longmapsto & 
 \begin{cases}
  \begin{tabular}{ccccc}
  $t_{ij}$ & if & $ k\notin\{i,j\} $ \\
  $\sum\limits_{p\in J} t_{pj}$ & if & $k=i$ \\
  $\sum\limits_{p\in J} t_{ip}$ & if & $j=k$ 
  \end{tabular}
  \end{cases}\\
   & (x^a_i,0) & \longmapsto & 
 \begin{cases}
  \begin{tabular}{ccccc}
  $x^a_i$ & if & $ k \neq i $ \\
  $\sum\limits_{p\in J} x^a_{p}$ & if & $k=i$ 
  \end{tabular}
  \end{cases}\\
     & (y^a_i,0) & \longmapsto & 
 \begin{cases}
  \begin{tabular}{ccccc}
  $y^a_i$ & if & $ k \neq i $ \\
  $\sum\limits_{p\in J} y^a_{p}$ & if & $k=i$ 
  \end{tabular}
  \end{cases}
\end{array}
$$  

We call $\t^f_g({\KK})$ the module of \textit{infinitesimal genus $g$ framed braids}.

\medskip

We also define the $\CD^f(\kk)$-module $\CD^f_g(\kk) := \mathcal{\hat U}(\t^f_g(\kk))$ 
of \textit{genus $g$ framed chord diagrams} whose morphisms can be pictured as 
chords on $n$ vertical strands with extra chords correponding to the generators
 $x^a_i$ and $y^a_i$ as follows

\begin{center}
\tik{\node[point, label=above:$i$] at (1,0) {}; \tell{0}{1}{A_a^+} } 
$ \qquad \text{and} \qquad $ \tik{\node[point, label=above:$i$] at (1,0) {}; \tell{0}{1}{A_a^-} }
\end{center}
The relations introduced in the definition of $\t^f_{g,n}$ can then be pictured as follows: 
\begin{center}

\begin{align}\tag{\ref{eqn:Sg}}
\tik{\node[point, label=above:$i$] at (1,0) {};\node[point, label=above:$j$] at (2,0) {};
\tell{0}{1}{A_a^{-}} \tell[->]{1}{2}{A_a^{+}}\straight[->]{1}{1}\straight{2}{0}}-
\tik{\node[point, label=above:$i$] at (1,0) {};\node[point, label=above:$j$] at (2,0) {};
\tell{0}{2}{A_a^{+}} \tell[->]{1}{1}{A_a^{-}}\straight[->]{2}{1}\straight{1}{0}}
= \tik{\node[point, label=above:$i$] at (1,0) {};\node[point, label=above:$j$] at (2,0) {};
\tell{0}{1}{A_a^{+}} \tell[->]{1}{2}{A_a^{-}}\straight[->]{1}{1}\straight{2}{0}}-
\tik{\node[point, label=above:$i$] at (1,0) {};\node[point, label=above:$j$] at (2,0) {};
\tell{0}{2}{A_a^{-}} \tell[->]{1}{1}{A_a^{+}}\straight[->]{2}{1}\straight{1}{0}}
 =\ \tik{\node[point, label=above:$i$] at (0,-1) {}; \node[point, label=above:$j$] at (1,-1) {};
 \hori[->]{0}{1}{1}{}{}}
\end{align}

\begin{align}\tag{\ref{eqn:Ng}}
\tik{\tell{0}{1}{A_a^{\pm}} \tell[->]{1}{2}{A_a^{\pm}}\straight[->]{1}{1}\straight{2}{0}
\node[point, label=above:$i$] at (1,0) {};\node[point, label=above:$j$] at (2,0) {};}=
\tik{\tell{0}{2}{A_a^{\pm}} \tell[->]{1}{1}{A_a^{\pm}}\straight[->]{2}{1}\straight{1}{0}
\node[point, label=above:$i$] at (1,0) {};\node[point, label=above:$j$] at (2,0) {};}
\end{align}

\begin{align}\tag{\ref{eqn:FTg}}
\underset{a=1}{\overset{g}{\sum}}  \tik{\tell{0}{1}{A_a^{+}} \tell[->]{1}{1}{A_a^{-}}
\straight[->]{1}{1} \node[point, label=above:$i$] at (1,0) {};}-\tik{\tell{0}{1}{A_a^{-}} 
\tell[->]{1}{1}{A_a^{+}}\node[point, label=above:$i$] at (1,0) {};}
=\ - \sum_{j;j\neq i} \tik{ \hori[->]{0}{1}{1} \node[point, label=above:$i$] at (0,-1) {}; 
\node[point, label=above:$j$] at (1,-1) {};} -(2g-1) 
\tik{  \node[point, label=above:$i$] at (0,-1) {}; 
  \draw[->,thick] (0,-1) to (0,-2); 
  \draw[-,thick] (-0.1,-1.5) to (0.1,-1.5)  ;
  }
\end{align}

\begin{align*}\tag{\ref{eqn:FLg}}
\tik{\node[point, label=above:$k$] at (3,0) {};\node[point, label=above:$i$] 
at (1,0) {};\node[point, label=above:$j$] at (2,0) {};
\tell{0}{1}{A_a^{\pm}}
\draw[zell] (0,-1)--(0,-2); 
\hori[->]{2}{1}{1}
\straight[->]{1}{1}
\straight{2}{0}
\straight{3}{0}
}
=
\tik{\node[point, label=above:$k$] at (3,0) {};\node[point, label=above:$i$] 
at (1,0) {};\node[point, label=above:$j$] at (2,0) {};\tell[->]{1}{1}{A_a^{\pm}}
\draw[zell] (0,0)--(0,-1); \hori{2}{0}{1}\straight{1}{0}
\straight[->]{2}{1}\straight[->]{3}{1}} ;
\end{align*}
\begin{align*}
\tik{\node[point, label=above:$i$] at (1,0) {};
\node[point, label=above:$j$] at (2,0) {};
\tell{0}{1}{A_a^{\pm}}
\draw[zell] (0,-1)--(0,-2); 
\straight[->]{1}{1}
 \draw[->,thick] (2,0) to (2,-2); 
  \draw[-,thick] (1.9,-1.5) to (2.1,-1.5)  ;
}
=
\tik{\node[point, label=above:$i$] at (1,0) {};\node[point, label=above:$j$] at (2,0) {};
\tell[->]{1}{1}{A_a^{\pm}}
\draw[zell] (0,0)--(0,-1); 
\straight{1}{0}
\straight[->]{1}{1}
 \draw[->,thick] (2,0) to (2,-2); 
  \draw[-,thick] (1.9,-0.5) to (2.1,-0.5)  ;
} 
\end{align*}

\begin{align*}\tag{\ref{eqn:F4Tg}}
\tik{\node[point, label=above:$i$] at (1,0) {};\node[point, label=above:$j$] at 
(2,0) {};\tell{0}{1}{A_a^{\pm}} \draw[zell] (0,-1)--(0,-2);\hori[->]{1}{1}{1}
\straight[->]{1}{1}\straight{2}{0}}+\tik{\node[point, label=above:$i$] at (1,0) {};
\node[point, label=above:$j$] at (2,0) {};\tell{0}{2}{A_a^{\pm}} 
\draw[zell] (0,-1)--(0,-2);\hori[->]{1}{1}{1}\straight[->]{1}{1}\straight{2}{0}\straight{1}{0}}
= \tik{\node[point, label=above:$i$] at (1,0) {};\node[point, label=above:$j$] at
 (2,0) {};\tell[->]{1}{1}{A_a^{\pm}} \draw[zell] (0,0)--(0,-1);\hori{1}{0}{1} 
 \straight[->]{2}{1}}+\tik{\node[point, label=above:$i$] at (1,0) {};
 \node[point, label=above:$j$] at (2,0) {};\tell[->]{1}{2}{A_a^{\pm}} 
 \draw[zell] (0,0)--(0,-1);\hori{1}{0}{1} \straight[->]{2}{1}\straight[->]{1}{1}}
\end{align*}
\begin{align*}
\tik{\node[point, label=above:$i$] at (1,0) {};
\tell{0}{1}{A_a^{\pm}}
\draw[zell] (0,-1)--(0,-2); 
\straight[->]{1}{1}
  \draw[-,thick] (0.9,-1.5) to (1.1,-1.5)  ;
}
=
\tik{\node[point, label=above:$i$] at (1,0) {};
\tell[->]{1}{1}{A_a^{\pm}}
\draw[zell] (0,0)--(0,-1); 
\straight{1}{0}
  \draw[-,thick] (0.9,-0.5) to (1.1,-0.5)  ;
} 
\end{align*}
\end{center}

\subsubsection{The $\PaCD^f(\kk)$-module of parenthesized framed genus $g$ chord diagrams}

As in the framed genus 0 situation, the module of objects 
$\on{Ob}(\CD^f_g(\kk))$ of $\CD^f_g(\kk)$ is terminal. 
Thus, we have a morphism of modules 
$\omega_2:\Pa=\on{Ob}(\mathbf{Pa}(\kk)\to\on{Ob}(\CD^f_g(\kk))$ over the morphism 
of operads $\omega_1$ from \S\ref{sec-pacdf}, and thus we can define 
the $\PaCD^f(\kk)$-module
\[
\PaCD^f_g(\kk):=\omega_2^\star \CD^f_g(\kk)\,,
\]
in $\mathbf{Cat(CoAss_\KK)}$, of so-called \textit{parenthesized 
genus $g$ framed chord diagrams}. 
We have 
\begin{itemize}
\item $\on{Ob}(\PaCD_g^f(\KK)):=\Pa$,
\item $\on{Mor}_{\PaCD_g^f(\KK)(n)}(p,q):=\on{End}_{\CD_g^f(\KK)(n)}(pt)$.
\end{itemize}

\begin{example}[Notable arrows in  $\mathbf{PaCD}_g(\kk)$]
We have the following arrows $X_a$, $Y_a$ in $\mathbf{PaCD}_g(\kk)(1)$ 
\begin{center}
$X_a=x_1^a \cdot$
\begin{tikzpicture}[baseline=(current bounding box.center)]
\tikzstyle point=[circle, fill=black, inner sep=0.05cm]
 \node[point, label=above:$1$] at (1,1) {};
 \node[point, label=below:$1$] at (1,-0.25) {};
 \draw[->,thick] (1,1) .. controls (1,0) and (1,0).. (1,-0.20); 
\end{tikzpicture}
\qquad\qquad $Y_a=y_1^a \cdot$
\begin{tikzpicture}[baseline=(current bounding box.center)]
\tikzstyle point=[circle, fill=black, inner sep=0.05cm]
 \node[point, label=above:$1$] at (1,1) {};
 \node[point, label=below:$1$] at (1,-0.25) {};
 \draw[->,thick] (1,1) .. controls (1,0) and (1,0).. (1,-0.20); 
\end{tikzpicture}
\end{center} 
and $X^{1,2}_a$, $Y^{1,2}_a$ in $\mathbf{PaCD}_g(\kk)(2)$
\begin{center}
$X^{1,2}_a=x_1^a\cdot$
\begin{tikzpicture}[baseline=(current bounding box.center)]
\tikzstyle point=[circle, fill=black, inner sep=0.05cm]
 \node[point, label=above:$1$] at (1,1) {};
 \node[point, label=below:$1$] at (1,-0.25) {};
 \node[point, label=above:$2$] at (2,1) {};
 \node[point, label=below:$2$] at (2,-0.25) {};
 \draw[->,thick] (1,1) .. controls (1,0) and (1,0).. (1,-0.20); 
 \draw[->,thick] (2,1) .. controls (2,0.25) and (2,0.5).. (2,-0.20);
\end{tikzpicture}
\qquad\qquad $Y^{1,2}_a=y_1^a \cdot$
\begin{tikzpicture}[baseline=(current bounding box.center)]
\tikzstyle point=[circle, fill=black, inner sep=0.05cm]
 \node[point, label=above:$1$] at (1,1) {};
 \node[point, label=below:$1$] at (1,-0.25) {};
 \node[point, label=above:$2$] at (2,1) {};
 \node[point, label=below:$2$] at (2,-0.25) {};
 \draw[->,thick] (1,1) .. controls (1,0) and (1,0).. (1,-0.20); 
 \draw[->,thick] (2,1) .. controls (2,0.25) and (2,0.5).. (2,-0.20);
\end{tikzpicture}
\qquad\qquad $\tilde{X}^{1,2}_a=x_2^a \cdot$
\begin{tikzpicture}[baseline=(current bounding box.center)]
\tikzstyle point=[circle, fill=black, inner sep=0.05cm]
 \node[point, label=above:$1$] at (1,1) {};
 \node[point, label=below:$1$] at (1,-0.25) {};
 \node[point, label=above:$2$] at (2,1) {};
 \node[point, label=below:$2$] at (2,-0.25) {};
 \draw[->,thick] (1,1) .. controls (1,0) and (1,0).. (1,-0.20); 
 \draw[->,thick] (2,1) .. controls (2,0.25) and (2,0.5).. (2,-0.20);
\end{tikzpicture}\\
$\tilde{Y}^{1,2}_a=y_2^a \cdot$
\begin{tikzpicture}[baseline=(current bounding box.center)]
\tikzstyle point=[circle, fill=black, inner sep=0.05cm]
 \node[point, label=above:$1$] at (1,1) {};
 \node[point, label=below:$1$] at (1,-0.25) {};
 \node[point, label=above:$2$] at (2,1) {};
 \node[point, label=below:$2$] at (2,-0.25) {};
 \draw[->,thick] (1,1) .. controls (1,0) and (1,0).. (1,-0.20); 
 \draw[->,thick] (2,1) .. controls (2,0.25) and (2,0.5).. (2,-0.20);
\end{tikzpicture}
\qquad\qquad $X^{12}_a=(x_1^a+ x_2^a) \cdot$
\begin{tikzpicture}[baseline=(current bounding box.center)]
\tikzstyle point=[circle, fill=black, inner sep=0.05cm]
 \node[point, label=above:$1$] at (1,1) {};
 \node[point, label=below:$1$] at (1,-0.25) {};
 \node[point, label=above:$2$] at (2,1) {};
 \node[point, label=below:$2$] at (2,-0.25) {};
 \draw[->,thick] (1,1) .. controls (1,0) and (1,0).. (1,-0.20); 
 \draw[->,thick] (2,1) .. controls (2,0.25) and (2,0.5).. (2,-0.20);
\end{tikzpicture}
\qquad\qquad $X^{12}_a=(y_1^a+ y_2^a) \cdot$
\begin{tikzpicture}[baseline=(current bounding box.center)]
\tikzstyle point=[circle, fill=black, inner sep=0.05cm]
 \node[point, label=above:$1$] at (1,1) {};
 \node[point, label=below:$1$] at (1,-0.25) {};
 \node[point, label=above:$2$] at (2,1) {};
 \node[point, label=below:$2$] at (2,-0.25) {};
 \draw[->,thick] (1,1) .. controls (1,0) and (1,0).. (1,-0.20); 
 \draw[->,thick] (2,1) .. controls (2,0.25) and (2,0.5).. (2,-0.20);
\end{tikzpicture}
\end{center}
\end{example}

\begin{remark}
One can write the elements $X^{12}_a,Y^{12}_a , \tilde{X}^{1,2}_a$ and 
$\tilde{Y}^{1,2}_a$ in terms of $X^{1,2}_a$ and $Y^{1,2}_a$ by means of the 
following relations:
\begin{itemize}
\item $\tilde{X}^{1,2}_a=X^{12}_a -X^{1,2}_a$, $\tilde{Y}^{1,2}_a=Y^{12}_a -Y^{1,2}_a$;
\item $X^{12}_a =X^{12, \emptyset}_a$, $Y^{12}_a =Y^{12, \emptyset}_a$.
\end{itemize}
\end{remark}

\begin{remark}\label{PaCD:fg:rel} 
There is a map of $\mathfrak{S}$-modules $\PaCD^f(\kk) \longrightarrow \PaCD^f_g(\kk)$ 
and we abusively denote $P^{1,2}$, $X^{1,2}$, $H^{1,2}$ and $a^{1,2,3}$ the images 
in $\PaCD^f_g(\kk)$ of the corresponding arrows in $\PaCD^f(\kk)$.  
 The elements $X^{1,2}_a$ and $, Y^{1,2}_a$ are generators of the 
 ${\PaCD}^f(\kk)$-module $\mathbf{PaCD}^f_g(\kk)$
 and satisfy the following relations for all $1\leq a \leq g$:
\begin{itemize}
\item $X_a^{2,1}=(X^{1,2})^{-1}X_a^{1,2}X^{1,2}$, $Y_a^{2,1}=(X^{1,2})^{-1}Y_a^{1,2}X^{1,2}$,
\item $X_a^{\emptyset,2}=Y_a^{\emptyset,2}=0$, $X_a^{1, \emptyset}=
X_a$, $Y_a^{1, \emptyset}=Y_a$,
\item $\tilde{X}_a^{12,3}+a^{1,2,3}X^{1,23}\tilde{X}_a^{23,1}(a^{1,2,3}
X^{1,23})^{-1}+X^{12,3}(a^{3,1,2})^{-1}\tilde{X}_a^{31,2}(X^{12,3}(a^{3,1,2})^{-1})^{-1}=
X_a^{(12)3}$,
\item $\tilde{X}_a^{12,3}+a^{1,2,3}X^{1,23}\tilde{X}_a^{23,1}
(a^{1,2,3}X^{1,23})^{-1}+X^{12,3}(a^{3,1,2})^{-1}
\tilde{X}_a^{31,2}(X^{12,3}(a^{3,1,2})^{-1})^{-1}=Y_a^{(12)3}$,
\item $H^{1,2}=  \left[a^{1,2,3}X_a^{1,23} (a^{1,2,3})^{-1}, X^{1,2}a^{2,1,3}
Y_a^{2,13}(a^{2,1,3})^{-1}
X^{2,1}\right],$
\item $H^{1,2}+(P^1)^{2(g-1)}= \sum_{a=1}^g \left[Y_a^{1,2},X_a^{1,2}
 \right].$
\end{itemize}
\end{remark}

\subsubsection{Genus $g$ associators}

\begin{definition}
  A genus $g$ associator over $\kk$ is couple $(F,G)$ where
  $F\in\Ass^f(\kk)$ is a $\kk$-associator and $G$ is an isomorphism between
  the $\widehat{\PaB}^f(\kk)$-module
  $\widehat{\PaB}^f_g(\kk)$ and the
  $G\mathbf{P}\mathbf{a}\mathbf{C}\mathbf{D}^f(\kk)$-module
  $G\mathbf{P}\mathbf{a}\mathbf{C}\mathbf{D}^f_g(\kk)$ which is the
  identity on objects and which is compatible with $F$. We denote its set
  by
  \[ \mathbf{A}\mathbf{s}\mathbf{s}_g (\kk) \assign 
  \on{Iso}^+_{\tmop{OpR}\mathbf{Grpd}_\kk}\Big((\widehat{\PaB}^f(\kk),
  \widehat{\PaB}^f_g(\kk)),(G\mathbf{P}\mathbf{a}\mathbf{C}\mathbf{D}^f(\kk),
    G\mathbf{P}\mathbf{a}\mathbf{C}\mathbf{D}^f_g(\kk))\Big). \]
\end{definition}

We have a morphism of short exact 
sequences
\begin{equation}
\xymatrix{
1 \ar[r] & \kk^n \ar[r]\ar[d] & \widehat{\on{PB}}^f_{g,n}(\kk) \ar[d] \ar[r] 
& \widehat{\on{PB}}_{g,n}(\kk) \ar[d] \ar[r] & 1 \\
1 \ar[r] & \kk^n \ar[r] & \on{exp}(\hat{\t}^f_{g,n}(\kk)) \ar[r] 
& \on{exp}(\hat{\t}_{g,n}(\kk))  \ar[r] & 1
}
\end{equation}
where the right vertical arrow was constructed in \cite{En4}. This shows 
that the map $ \widehat{\on{PB}}^f_{g,n}(\kk) \to \on{exp}(\hat{\t}^f_{g,n}(\kk))$ 
is a $\kk$-pro-unipotent group isomorphism. We will derive this result from the 
flatness of a connection defined over $\Conff (\Sigma_g,n)$ in a future work.

\begin{theorem}\label{Assg}
  There is a one-to-one correspondence between elements of
  $\mathbf{A}\mathbf{s}\mathbf{s}_g (\kk)$ and elements of the 
  set $\on{Ass}_g(\kk)$ consisting on tuples $(\mu,
 \varphi,A^{1,2}_{1,\pm}, \ldots, A^{1,2}_{g,\pm})$ where $(\mu,
 \varphi) \in \on{Ass}(\kk)$ and $ A^{1,2}_{a,\pm} \in \exp
  (\hat{\mathfrak{t}}^f_{g, 2})$, for $a=1,...,g$, satisfying
   the following equations in $\exp
  (\hat{\mathfrak{t}}^f_{g, 1}(\kk))$:
    \begin{equation} \label{def:g:ass:0}
A_{a,\pm}^{\emptyset,1}=1
\end{equation}
  the following equations in $\exp
  (\hat{\mathfrak{t}}^f_{g, 3}(\kk))$:
  \begin{equation} \label{def:g:ass:1}
\alpha_a^{1,2,3}\alpha_a^{2,3,1} \alpha_a^{3,1,2}=A_{a,\pm} ^{(12)3}, 
\text{\ where\ }   \alpha_a =\varphi^{1,2,3}
A_{a,\pm} ^{1,23} e^{\mu(t_{12}+t_{13})/2}, 
\end{equation}
\begin{equation}\label{def:g:ass:2}
e^{\mu t_{12}} =  \big(e^{\mu t_{12}/2}\varphi^{2,1,3} A_{a,+}^{2,13}
(\varphi^{2,1,3})^{-1}e^{\mu t_{12}/2}, \varphi^{1,2,3}A_{a,-}^{1,23}
(\varphi^{1,2,3})^{-1}\big),
\end{equation}
for all $1\leq a \leq g$ 
\begin{equation}\label{def:g:ass:3}
1=  \big( \varphi^{1,2,3}A_{a,\pm}^{1,23}(\varphi^{1,2,3})^{-1},
e^{\mu t_{12}/2}\varphi^{2,1,3} A_{a,\pm}^{2,13}(\varphi^{2,1,3})^{-1}
e^{\mu t_{12}/2}\big),
\end{equation}
for all $1\leq b<a \leq g$ and the following equation in $\exp
  (\hat{\mathfrak{t}}^f_{g, 2}(\kk))$:
\begin{equation} \label{def:g:ass:4}
e^{\mu( t_{12}+2(g-1) t_1)} =\sum_{a=1}^g \left((A^{1,2}_{a,+} )^{-1},A^{1,2}_{a,-}
 \right).  
\end{equation}
\end{theorem}

\begin{proof}
Let $\tilde{F}$ be a framed $\kk$-associator 
$\widehat{\PaB}^f(\KK) \longrightarrow G \PaCD^f(\KK)$ and let $\tilde{G}$ 
be an isomorphism $$\widehat{\PaB}_g^f(\KK) \longrightarrow G \PaCD_g^f(\KK)$$ of 
$(\widehat{\PaB}^f(\KK),G \PaCD^f(\KK))$-modules 
which is the identity on objects and which is compatible 
with $\tilde{F}$. It corresponds to a unique morphism 
$G:\PaB_g^f\longrightarrow G \PaCD_g^f(\KK)$. 
From the presentation of $\PaB_g^f$, we know that $G$ is 
uniquely determined by the images 
of $A_{a,+}^{1},B_a^{1}\in  \on{Hom}_{\PaB_g^f(\KK)(1)}(1)$ 
and $A_{a,+}^{1,2},B_a^{1,2}\in  \on{Hom}_{\PaB_g^f(\KK)(2)}(12)$, 
for all $1\leq a\leq g$ at the morphisms level. 
Thus, there are elements $ A^{1,2}_{a,\pm}\in \exp
  (\hat{\mathfrak{t}}^f_{g, 2})$, for $a=1,...,g$, such that 
\begin{itemize}
\item $G(A_a^{1,2})=A^{1,2}_{a,+} \cdot X_a^{1,2}$, 
\item $G(B_a^{1,2})=A^{1,2}_{a,-}\cdot Y_a^{1,2}$. 
\end{itemize}
These elements must satisfy the following relations 
\eqref{def:g:ass:0}, \eqref{def:g:ass:1}, \eqref{def:g:ass:2}, \eqref{def:g:ass:3} and 
\eqref{def:g:ass:4}, which are the images of relations \eqref{eqn:gR}, \eqref{eqn:gD}, \eqref{eqn:gN}, 
\eqref{eqn:gE1} and \eqref{eqn:gE2}.
\end{proof}

\subsection{Graded genus $g$ Grothendieck--Teichm\"uller groups}

\begin{definition}
The graded genus $g$ Grothendieck-Teichm\"uller group is the 
group $$\GRT_g(\KK):=\on{Aut}^+_{\tmop{OpR}\mathbf{Grpd}_\kk}(G\PaCD^f(\kk),G\PaCD^f_g(\KK))$$ 
of automorphisms of 
the $G\PaCD^f(\KK)$-module $G\PaCD^f_g(\KK)$ which are the 
identity on objects.  
\end{definition}
Notice that there is an isomorphism 
$$
\on{Aut}^+_{\tmop{OpR}\mathbf{Cat}(\mathbf{CoAlg}_\kk)}(\PaCD^f(\kk),\PaCD^f_g(\KK))\simeq 
\on{Aut}^+_{\tmop{OpR}\mathbf{Grpd}_\kk}(G\PaCD^f(\kk),G\PaCD^f_g(\KK)).
$$
To any element $(F,G)$ in $\GRT_g(\KK)$ one can associate tuples 
$(\mu,g,u^{1,2}_{1,\pm},\ldots,u^{1,2}_{g,\pm})$, such that $(\mu,g)\in\on{GRT}(\kk)$ and
\begin{itemize}
\item $G(X_a^{1,2})=u^{1,2}_{a,+} \cdot X_a^{1,2}$, 
\item $G(Y_a^{1,2})=u^{1,2}_{a,-}\cdot Y_a^{1,2}$. 
\end{itemize}
Here $u^{1,2}_{1,\pm},
\ldots,u^{1,2}_{g,\pm}\in\hat{\t}^f_{g,2}(\kk)$ 
satisfy, for $1 \leq a \leq g$, 
    \begin{equation} \label{def:grt:g:0}
u_{a,\pm} ^{1,\emptyset}=u_{a,\pm} ^1, \quad u_{a,\pm}^{\emptyset,1}=1
\end{equation}
\begin{equation} \label{def:grt:g:1}
\on{Ad}(g^{1,2,3})((u^a_\pm)^{1,23})+
\on{Ad}(g^{2,1,3})((u^a_\pm)^{2,13})+(u^a_\pm)^{3,12}=
x_1^{a,\pm}+x_2^{a,\pm}+x_3^{a,\pm}, 
\end{equation}
\begin{equation} \label{def:grt:g:2}
[\on{Ad}(g^{1,2,3})((u^a_\pm)^{1,23}),(u^a_\pm)^{3,12}]=0, 
\end{equation}
\begin{equation} \label{def:grt:g:N}
[\on{Ad}(g^{2,1,3})((u^b_\pm)^{2,13}), 
\on{Ad}(g^{1,2,3})((u^a_\pm)^{1,23})]]=0, 
\end{equation}
\begin{equation} \label{def:grt:g:3}
[\on{Ad}(g^{2,1,3})((u^a_+)^{2,13}), 
\on{Ad}(g^{1,2,3})((u^a_-)^{1,23})]=t_{12},
\end{equation}
as relations in $\hat{\t}^f_{g,3}({\kk})$ and
\begin{equation} \label{def:grt:g:4}
\sum_{a=1}^g [u^a_+,
u^a_-]=t_{12}+2(g-1)t_1.
\end{equation}
Let us denote by $\on{GRT}_1^{g}(\kk)$ the set of such tuples.
Set $(\mu,g,u^{1,2}_{1,\pm},\ldots,u^{1,2}_{g,\pm})*(\tilde 
\mu,\tilde g,\tilde u^{1,2}_{1,\pm},\ldots,\tilde u^{1,2}_{g,\pm}):= 
(\tilde{\tilde{\mu}},\tilde{\tilde{g}},\tilde{\tilde{u}}^{1,2}_{1,\pm},
\ldots,\tilde{\tilde{u}}^{1,2}_{g,\pm})$, 
where $\tilde{\tilde{\mu}},\tilde{\tilde{g}}$ are as in
subsection \ref{Rators} and, for all $1\leq a \leq g$,
\begin{eqnarray}\label{pdt:upm:2}
\tilde{\tilde{u}}^{1,2}_{a,\pm} & := &  u^{1,2}_{a,\pm}
\Big(\tilde u^{1,2}_{1,\pm}(x_1^a,x^a_2,y_1^a,y^a_2;1\leq a \leq g),
\tilde u^{2,1}_{1,\pm}(x_1^a,x^a_2,y_1^a,y^a_2;1\leq a \leq g)) \\
\nonumber & &  \ldots ,\tilde u^{1,2}_{g,\pm}(x_1^a,x^a_2,y_1^a,
y^a_2;1\leq a \leq g),\tilde u^{2,1}_{g,\pm}
(x_1^a,x^a_2,y_1^a,y^a_2;1\leq a \leq g))\Big).
\end{eqnarray}

The group $\kk^\times$ acts on $\on{GRT}_1^{g}(\kk)$ by rescaling. 
We then set 
$\on{GRT}_{g}(\kk):= \on{GRT}_1^{g}(\kk)\rtimes \kk^\times$. 

\subsubsection{The non framed case for chord diagrams}

Let us consider $g > 0$ and $n \geq 0$ and define $\mathfrak{t}_{g, n}
(\kk)$ as the $\kk$-Lie algebra with generators $x_i^a,y_i^a,t_{ij}$ for $i\neq j\in [n], 1\leq a \leq g$ 
satisfying relations \eqref{eqn:S}, \eqref{eqn:L}, \eqref{eqn:4T} and
\begin{flalign}
& [x_i^a,y_j^b]=\delta_{ab}t_{ij} \quad \text{ for all } i\neq j,   \tag{S$_g$} \label{eqn:Sg1} \\
& [x_i^a,x_j^b]=0= [y_i^a,y_j^b] \quad \text{ for all } i\neq j,   \tag{N$_g$} \label{eqn:Ng1} \\
& \underset{a=1}{\overset{g}{\sum}} [x_i^a,y_i^a]= - \sum_{j:j\neq i}t_{ij}, \tag{T$_g$} \label{eqn:Tg} \\
& [x_k^a,t_{ij}]= [y_k^a,t_{ij}]=0 \quad \textrm{if }\#\{i,j,k\}=3\,,   \tag{L$_g$} \label{eqn:Lg} \\
& [y_i^a+y_j^a,t_{ij}] = [x_i^a+x_j^a,t_{ij}] =0\quad\textrm{for }i\neq j\,.   &   \tag{4T$_g$} \label{eqn:4Tg}
\end{flalign}

The Lie algebra $\mathfrak{t}_{g,
n} (\kk)$ is equipped with a grading given by $deg(x^a_i)  = (1, 0)$,
$deg(y^a_i ) = (0, 1)$. The total degree defines a positive grading on
$\mathfrak{t}_{g, n} (\kk)$; we denote by $\hat{\mathfrak{t}}_{g, n}
(\kk)$ the corresponding completion. If $\kk=\mathbbm{C}$,
we will denote $\mathfrak{t}_{g, n} (\kk) \assign \mathfrak{t}_{g,
n}$.

The Lie algebra $\t_{g, n}(\KK)$ is acted on by the symmetric group $\mathfrak{S}_n$, 
and one can show that the $\mathfrak{S}$-module in $grLie_\kk$ 
$$
\t_{g}({\KK}):=\{\t_{g, n}(\KK) \}_{n \geq 0}
$$
is a $\t(\KK)$-module in $grLie_{\KK}$. 

The collection of the Lie algebras $\t_{g, n}(\KK)$, for $n \geq 1$ is provided with 
the structure of an module over
  the operad $\mathfrak{t}$ in (positively graded finite dimensional) 
Lie algebras over $\kk$, denoted $\t_{g,}(\KK)$. Partial compositions are defined 
as follows: for $I$ a finite set and $i\in I$, 
  $$
  \begin{array}{cccccc}
\circ_k : & \t_{g,I}(\KK) \oplus \t_J(\KK)  & \longrightarrow & \t_{g,J\sqcup I-\{i\}}(\KK) \\
    & (0,t_{\alpha \beta}) & \longmapsto & t_{\alpha\beta} \\
 & (t_{ij},0) & \longmapsto & 
 \begin{cases}
  \begin{tabular}{ccccc}
  $t_{ij}$ & if & $ k\notin\{i,j\} $ \\
  $\sum\limits_{p\in J} t_{pj}$ & if & $k=i$ \\
  $\sum\limits_{p\in J} t_{ip}$ & if & $j=k$ 
  \end{tabular}
  \end{cases}\\
   & (x^a_i,0) & \longmapsto & 
 \begin{cases}
  \begin{tabular}{ccccc}
  $x^a_i$ & if & $ k \neq i $ \\
  $\sum\limits_{p\in J} x^a_{p}$ & if & $k=i$ 
  \end{tabular}
  \end{cases}\\
     & (y^a_i,0) & \longmapsto & 
 \begin{cases}
  \begin{tabular}{ccccc}
  $y^a_i$ & if & $ k \neq i $ \\
  $\sum\limits_{p\in J} y^a_{p}$ & if & $k=i$ 
  \end{tabular}
  \end{cases}
\end{array}
$$ 

Since we are in possession of operad modules $\mathbf{Pa}(\kk)$ and $\widehat{\CD}_g(\kk)$ in 
$\bf{Cat(CoAss_\KK)}$ and of an operad module morphism $f:\Pa\to\on{Ob}(\widehat{\CD}_g(\kk))$, 
we are ready to define the $\PaCD(\kk)$-module 
$$\PaCD_g(\KK):=f^\star \widehat{\CD}_g(\kk)$$ in $\mathbf{Cat(CoAss_\KK)}$ of 
parenthesized genus $g$ chord diagrams. We have $\on{Ob}(\PaCD_g(\KK)):=\Pa$ 
and $\on{Mor}_{\PaCD_g(\KK)(n)}(p,q):=\hat{\mathcal{U}}(\hat\t_{g,n}({\kk}))$.

\subsection{Towards the genus $g$ KZB associator}
Recall the following result.
\begin{theorem}
  \label{formality:thm}(Bezrukavnikov, Enriquez) There is a monodromy morphism
  $\tmop{PB}_{g, n} \to \exp (\hat{\mathfrak{t}}_{g, n})$ inducing an
  isomorphism of Lie algebras $\tmop{Lie} (\tmop{PB}_{g, n})^{\mathbbm{C}}
  \overset{\sim}{\to} \hat{\mathfrak{t}}_{g, n}$.
\end{theorem}

Let us recall the construction from \cite{En4} of the universal genus $g$ KZB 
connection (defined over the configuration spaces). Endow the surface 
$\Sigma_g$ with a complex structure and denote $C$ the resulting smooth 
closed complex curve. For any $z\in C$, the fundamental group of $C$ 
based at $z$ is isomorphic to the group generated by $X^a,Y^a,1\leq a \leq g$, 
such that $\prod_{a=1}^g (X^a,Y^a)=1$ and
$
\on{PB}_{g,n}:=\pi_1(\on{Conf}(C,n),\zz)
$
where $\zz:=(z_1,\ldots,z_n)\in \on{Conf}(C,n)$.

Define a map $\rho : \on{PB}_{g,n} \longrightarrow \on{exp}(\hat\f_g^{\oplus n}) $ 
by means of the following composite
$$\on{PB}_{g,n}\to \pi_1(C^n,\zz)\to \pi_g^n \to F_g^n \to \on{exp}(\hat\f_g)^n, $$ where 
\begin{itemize}
\item $F_g$ is the free group with generators $\gamma_a, 1\leq a \leq g$, 
\item $\pi_g\to F_g$ is the composite
$$
\pi_g\to \pi_g/N \to F_g
$$
where $\pi_g\to \pi_g/N$ is the quotient morphism, 
where $N$ is the normal subgroup generated by the $X^a$, $1\leq a \leq g$,
\item $\pi_g/N\to F_g$, $\bar Y^a\mapsto \gamma_a$ is the isomorphism induced from 
the presentation of $\pi_g/N$, where $F_g\to \on{exp}(\hat\f_g)$ is the assignment
$\gamma_a\mapsto \on{exp}(x_a)$.
\end{itemize}

The principal $\on{exp}(\hat\t_{g,n})$-bundle with flat connection on $\on{Conf}(C,n)$ 
corresponding 
to $\rho_0$  is then 
$i^*({\mathcal{P}}_n)$, where $i : \on{Conf}(C,n)\to C^n$ is the inclusion and 
$$({\mathcal{P}}_n\to C^n) = ({\mathcal{P}}^0_1\to C)^n\times_{\on{exp}(\hat \f_g)^n}
\on{exp}(\hat\t_{g,n}),$$ where $({\mathcal{P}}^0_1\to C)$ is the principal 
$\on{exp}(\hat\f_g)$-bundle with flat connection corresponding to the above 
morphism $\pi_g\to F_g\to\on{exp}(\hat\f_g)$. 

Denote the set of flat connections of degree 1 by
$$
\cF_1 = \{\alpha\in\Omega^1(C^n - \on{Diag}, {\mathcal{P}}_n\times_{\on{ad}}
\hat\t_{g,n}[1]) | d\alpha = \alpha\wedge\alpha = 0\}
$$
and denote its subset of holomorphic flat connections by
$$
\cF_1^{hol} = \{\alpha\in H^0(C^n, \Omega^{1,0}_{C^n}\otimes 
( {\mathcal{P}}_n \times_{\on{ad}}\hat\t_{g,n}[1])(*\on{Diag})) | d\alpha = \alpha\wedge\alpha = 0\}
$$
with $\on{Diag} = \sum_{i<j}\on{Diag}_{ij}$ and $\on{Diag}_{ij}\subset C^n$ is the diagonal
corresponding to $z_i=z_j$. Then, Enriquez showed that
there is an element $\alpha_{KZ}\in \cF_1^{hol}$ given by
\begin{equation}
\alpha^{\on{KZB}}_{g,n} = \sum_{i=1}^n \alpha_i,  
\end{equation}
where $\alpha_i \in H^0(C,K_C^{(i)} \otimes ( {\mathcal{P}}_n \times_{\on{ad}}\hat\t_{g,n}[1])
(\sum_{j:j\neq i}\Delta_{ij}))$ expands as $\alpha_i \equiv \sum_{1\leq a \leq g} \omega_a^{(i)}y_a^i$ 
modulo $\hat\oplus_{q\geq 2}\t_{g,n}[1,q]$.

As in \cite{En4},  $K_C^{(i)} = \cO_C^{\boxtimes i-1}
\boxtimes K_C \boxtimes \cO_C^{\boxtimes n-i}$, $\omega_a^{(i)} = 1^{\otimes i-1}
\otimes \omega_a \otimes 1^{\otimes n-i}$, where $(\omega_a)_{1\leq i \leq g}$ are the holomorphic 
differentials such that $\int_{\cA_a}\omega_b = \delta_{ab}$ and the images of $X^a$ 
and $Y^a$ under $\pi_g \to \pi_g^{ab} \simeq H_1(C,\Z)$ are $\mathcal{X}^a$ and 
$\mathcal{Y}^a$ respectively.

Consider integers $(g,n)$ in hyperbolic position (i.e. $2-2g-n<0$) and let $S$ be a 
genus $g$ topological compact oriented surface, $x_1,...,x_n$ $n$ marked points 
on it. Now let $X$ be a Riemann surface modeled on $S$ with genus $g$ and $n$ 
marked points. As $X$ is hyperbolic, the Uniformisation Theorem says that $X$ is
 isomorphic to a quotient $\h/\Gamma$ of the Poincaré half-plane $\h$ by a discrete 
 subgroup $\Gamma$ of $\on{PSL}(2,\RR)$. Fix $\tau\in\h$ and consider a 
 uniformization $\Sigma_{g}$ of $X$. This corresponds to a point $\kappa$ in the 
 moduli space $\mathcal{M}_{g,n}$. Such a point can be described by $3g+n-3$ 
 parameters. Enriquez chowed that, under this uniformization, the one form 
 $\alpha_{KZ}$ induces a flat connection 
$$
\nabla^{\on{KZB}}_{g,n,\kappa}:=\on{d}-\alpha^{\on{KZB}}_{g,n,\kappa}
$$
over $\on{Conf}(\Sigma_{g,\kappa},n)$.
Now, the fundamental group $\pi_1(\Sigma^\times_{g,\kappa},\zz_0)$ of 
$\Sigma^\times_{g,\kappa}:=\Sigma_{g,\kappa}-{p_0}$, where $p_0$ is a puncture 
in the curve, is the nothing but the free group $F(x^1,y^1,x^2,y^2,...,x^g,y^g)$ on 
$2g$ generators. Now choose a non-zero tangent vector $\overrightarrow{v}_0$ of 
$\Sigma_{g,\kappa}$ at $0$. Then, flatness of $\nabla^{KZB}_{g,n,\kappa}$ implies 
the existence of a $\Q$-algebra map
\begin{eqnarray*}
T^{g,\on{KZB}}_{-\overrightarrow{v}_0,\overrightarrow{v}_0} : 
\Q[\pi_1(\Sigma^\times_{g,\tau},-\overrightarrow{v}_0,\overrightarrow{v}_0) ] 
& \longrightarrow & \Q\llangle x^1,y^1,x^2,y^2,...,x^g,y^g \rrangle \\
\gamma & \longmapsto & T^{\on{g,KZB}}_{-\overrightarrow{v}_0,\overrightarrow{v}_0}(\gamma)
:= \sum_{k=0}^{\infty} Reg \int_{\gamma} \alpha^{\on{KZB}}_{g,n,\kappa}
\end{eqnarray*}
\begin{definition}
The non-framed genus $g$ KZB associator is the tuple $$e_g(\kappa)\assign 
(A_1(\kappa),B_1(\kappa),\ldots,A_g(\kappa),B_g(\kappa))$$ where 
\begin{eqnarray*}
A_a(\kappa) & := &  T^{g,\on{KZB}}_{-\overrightarrow{v}_0,\overrightarrow{v}_0}(X_1^a)\\
B_a(\kappa) &: = & T^{g,\on{KZB}}_{-\overrightarrow{v}_0,\overrightarrow{v}_0}(Y_1^a)
\end{eqnarray*}
where $X_1^a$ and $Y_1^a$, $1\leq a \leq g$ are the generating elements in 
$\pi_1(\on{Conf}^f(\Sigma_{g,\kappa},[2]))$.
\end{definition}

We do not know what kind of monodromy relations these associators may have. 
In particular, if we want to relate them to our operadic definition of genus $g$ 
associators we need to extend the universal KZB connection to its framed version.
In this direction, we propose the following:
\begin{conjecture}\label{Conj1}
  There is a flat universal framed KZB connection $\nabla_{g,n,\kappa}^{f\on{KZB}}$ 
  defined on the principal $\on{exp}(\hat{\t}_{g,n}^f)$-bundle over $\on{Conf}^f(C,n)$ 
  constructed as above such that
\begin{itemize}
\item its pullback of $\nabla_{g,n,\kappa}^{f\on{KZB}}$ to the associated 
$\on{exp}(\hat{\t}_{g,n}^f)$-bundle over $C^n$ is 
$$
\nabla_{g,n,\kappa}^{f\on{KZB}}:= \on{d} - \alpha^{f\on{KZB}}_{g,n}
$$
where 
$$
\alpha^{f\on{KZB}}_{g,n}:= \alpha^{\on{KZB}}_{g,n}+ \underset{1 \leqslant i  \leqslant n}{\sum} t_i \on{d}log(\lambda_i);
$$
\item  the 1-form $\alpha^{f\on{KZB}}_{g,n}$ is $(\C^\times)^n$-basic and the 
induced connection on the $\on{exp}(\hat{\t}_{g,n})$-bundle over $\on{Conf}(C,n)$ 
given above coincides with the universal genus $g$ KZB connection defined by Enriquez.
\end{itemize}  
\end{conjecture}

Let $\kappa$ represent a point in the moduli space $\mathcal{M}_{g,n}$. 
If conjecture \ref{Conj1} holds, then the monodromy of the connection 
$\nabla_{g,n,\kappa}^{f\on{KZB}}$ induces a tuple $$e_g(\kappa)\assign 
(A^f_1(\kappa),B^f_1(\kappa),\ldots,A^f_g(\kappa),B^f_g(\kappa))$$ where 
\begin{eqnarray*}
A^f_a(\kappa) & := &  T^{g,\on{KZB}}_{-\overrightarrow{v}_0,\overrightarrow{v}_0}(X_1^a)\\
B^f_a(\kappa) &: = & T^{g,\on{KZB}}_{-\overrightarrow{v}_0,\overrightarrow{v}_0}(Y_1^a)
\end{eqnarray*}
where $X_1^a$ and $Y_1^a$ are the inverse generating loops in $\pi_1(\on{Conf}^f(\Sigma_{g},2))$.
Let $(2 i \pi, \Phi^f_{\tmop{KZ}})$ be the framed KZ associator coming 
from the framed universal KZ connection defined above.
\begin{conjecture}
The data $(2 i \pi, \Phi^f_{\tmop{KZ}}, e^f_g (\kappa))$, where $$e^f_g (\kappa) = (A^f_1
  (\kappa), B^f_1 (\kappa), \ldots, A^f_g (\kappa), B^f_g (\kappa))$$ is a genus $g$ $\C$-associator.
\end{conjecture}

\section{Torsor comparisons in the elliptic case}
\label{genus 1 associators}

\subsection{Four modules of genus 1 parenthesized braidings}
\label{The genus 1 case}

Since our base space $\mathbb{T}$ is parallelizable and has a translation action, there are four variants of the module of parenthesized elliptic braids corresponding to the framed/unframed and the reduced/non-reduced situations. This subsection is devoted into 
comparing these four operadic modules. 

On the one hand, the above subsection applied to $g=1$ gives a $\PaB^f$-module
\begin{eqnarray*}
\PaB^f_1& :=& \pi_1(\overline{\textrm{Conf}}^f(\mathbb{T},-), \Pa).
\end{eqnarray*}
As $\overline{\textrm{Conf}}(\mathbb{T},-)$ is a module over $\overline{\textrm{C}}(\C,-)$,
we obtain a $\PaB$-module
$$\PaB_1:=\pi_1(\overline{\textrm{Conf}}(\mathbb{T},-), \Pa).$$ 
The operadic pointings are chosen 
to be the unit of of $\PaB^f_1(1)$ and $\PaB_1(1)$ respectively.

On the other hand, as constructed in \cite{CaGo2}, to any finite set $I$ we associate the  ASFM 
 compactification $\overline{\textrm{C}}(\mathbb{S}^1,I)$ of the reduced configuration space $\textrm{C}(\mathbb{S}^1,I):=\textrm{Conf}(\mathbb{S}^1,I)/\mathbb{S}^1$ 
 of $\mathbb{S}^1$.
The inclusion of boundary components provide $\overline{\rm C}(\mathbb{S}^1,-)$ with the structure of a module over the operad 
$\overline{\rm C}(\mathbb{R},-)$ in $\mathbf{Top}$. 

Thus, we can construct a $\PaB^f$-module
$$
\PaB^f_{e\ell\ell}:=\pi_1(\overline{\textrm{Conf}}^f(\mathbb{T},-)/\mathbb{T}, \Pa),
$$
and a $\PaB$-module
$$
\PaB_{e\ell\ell}:=\pi_1(\overline{\textrm{Conf}}(\mathbb{T},-)/\mathbb{T}, \Pa),
$$
Here, the action of $\mathbb{T}$ on 
the configuration space is given by global translation of the marked points.

In \cite{CaGo2}, we showed that, as a $\mathbf{PaB}$-module in groupoids having 
$\mathbf{Pa}$ as $\mathbf{Pa}$-module of objects, $\mathbf{PaB}_{e\ell\ell}$ is freely 
generated by morphisms two morphisms satisfying certain relations.
For each $n\geq 1$ and each $p\in \PaB_{e\ell\ell}(n)$, the group 
$\on{Aut}_{\PaB_{e\ell\ell}(n)}(p)$ is isomorphic to the reduced pure 
braid group $\overline{\on{PB}}_{1,n}$ with $n$ strands on the torus. In \cite{CaGo2}, 
we give a presentation of this group conjugated to the one we use in here 
(i.e. \cite[Definition 5]{BeG}).

As a $\PaB$-module in
    groupoids having $ \textbf{\tmop{\Pa}}$ as
    $ \textbf{\tmop{\Pa}}$-module of objects, $\PaB_1$ is
  isomorphic freely generated by 
$A^{1, 2}$ and $B^{1, 2}$ in arity 2,
together with relations
\begin{flalign}
& A^{\emptyset,2}=\on{Id}_{1}, B^{\emptyset,2}=\on{Id}_{1},		\label{eqn:Pab:1:0}		 \\
& \Phi^{1,2,3}A^{1,23}R^{1,23}\Phi^{2,3,1}A^{2,31}R^{2,31}\Phi^{3,1,2}A^{3,12}R^{3,12} =A^{(12)3},		\label{eqn:Pab:1:1}		\\
& \Phi^{1,2,3}B^{1,23}R^{1,23}\Phi^{2,3,1}B^{2,31}R^{2,31}\Phi^{3,1,2}B^{3,12}R^{3,12} =B^{(12)3},			\label{eqn:Pab:1:2} \\		
& R^{1,2}R^{2,1}= \left(R^{1,2}\Phi^{2,1,3}A^{2,13}(\Phi^{2,1,3})^{-1}R^{2,1},\Phi^{1,2,3}B^{1,23}(\Phi^{1,2,3})^{-1}\right) \label{eqn:Pab:1:3}	\\
&R^{1,2}R^{2,1}=\left((A^{1,2})^{-1},B^{1,2}\right). \label{eqn:Pab:1:4} &	
\end{flalign}

The quotient map $ \on{PB}_{1,n} \to \overline{\on{PB}}_{1,n}=\on{PB}_{1,n}/(X_1
\ldots X_n,Y_1\ldots Y_n)$ induces a unique $\PaB$--module morphism 
$$
F: \PaB_1 \longrightarrow \PaB_{e\ell\ell}
$$
given by the identity on objects and
\begin{itemize}
\item $F(A^{1,2})= (A^{1,2})^{-1}$, $F(B^{1,2})=(B^{1,2})^{-1}$.
\end{itemize}

Recall from \cite{Ho2} that we have a short exact sequence of operads in groupoids
\begin{equation}\label{f}
1 \longrightarrow \PaB \longrightarrow \PaB^f \longrightarrow \mathbf{Z} \longrightarrow 1,
\end{equation}
where $\mathbf{Z}$ is viewed as the operad in groupoids with a single object in each 
arity $n$ and $\Z^n$ as endomorphism of the object. One can also show that we have 
an isomorphism of operads $\PaB^f \simeq \PaB \rtimes \mathbf{Z}$ (see \cite{Wahl} 
for more details).
Then the inclusion $\PaB_1 \to \PaB_1^f$ (which topologically sends the marked points
 to the same marked points with all framings attached to them aligned to the right on the real line) induces a short exact sequence 
\begin{equation}\label{f:1}
1 \longrightarrow \PaB_1 \longrightarrow \PaB_1^f \longrightarrow \Z \longrightarrow 1
\end{equation}
of modules over \eqref{f}.
Then, we have an isomorphism
\begin{eqnarray*}
\PaB^f_1 \simeq & \PaB_1  \rtimes  \mathbf{Z}.
\end{eqnarray*}
This leads to an isomorphism of Drinfeld torsors between the framed and non 
framed genus $1$ non-reduced situations as we will see below.

 Then, one can eventually obtain a further short exact sequence 
\begin{equation}\label{f:ell}
1 \longrightarrow \PaB_{e\ell\ell} \longrightarrow \PaB_{e\ell\ell}^f \longrightarrow \bar\Z \longrightarrow 1
\end{equation}
of modules over \eqref{f}. Here $\bar\Z$ is viewed as the operad in groupoids 
with a single object in each arity $n$ and $\Z^n/\Z$ as endomorphism of the 
object, with diagonal action of $\Z$ on $\Z^n$.

\subsection{Reminders on elliptic associators}

In the genus 1 case, the $\sum_i x_i$ and $\sum_i y_i$ are central 
in $\t_{1,n}({\KK})$, and we also consider the quotient 
$$
\bar\t_{1,n}({\KK}):=\t_{1,n}({\KK})/(\sum_i x_i,\sum_i y_i)\,.
$$ 
In particular, 
$\bar\t_{1,2}(\KK)$ is equal to the free Lie $\kk$-algebra $\f_2(\kk)$ 
on two generators $x=x_1$ and $y=y_2$. 
The Lie algebra $\bar\t_{1,n}$ is acted on by the symmetric group 
$\mathfrak{S}_n$, and one can show that the $\mathfrak{S}$-module 
in $grLie_\kk$ 
$$
\bar\t_{e\ell\ell}({\KK}):=\{\bar\t_{1,n}(\KK) \}_{n \geq 0}
$$
actually is a $\t(\KK)$-module in $grLie_{\KK}$. 

The same formula defines a $\bar\t(\kk)$-module structure on $\bar\t_{e\ell\ell}(\kk)$.
We call $\bar{\t}_{e\ell\ell}({\KK})$ the module of \textit{infinitesimal reduced elliptic braids} and we define the $\CD(\kk)$-module $\CD_{e\ell\ell}(\kk) := 
\mathcal{\hat U}(\bar\t_{e\ell\ell}(\kk))$ of \textit{elliptic chord diagrams}. 
As in the genus zero case, the module of objects 
$\on{Ob}(\CD_{e\ell\ell}(\kk))$ of $\CD_{e\ell\ell}(\kk)$ is terminal. 
Hence we have a morphism of modules $\omega_2:\Pa=\on{Ob}
(\mathbf{Pa}(\kk))\to\on{Ob}(\CD_{e\ell\ell}(\kk))$ over the morphism 
of operads $\omega_1$, and thus we can define the $\PaCD(\kk)$-module
\[
\PaCD_{e\ell\ell}(\kk):=\omega_2^\star \CD_{e\ell\ell}(\kk)\,,
\]
in $\mathbf{Cat(CoAss_\KK)}$, of so-called 
\textit{parenthesized elliptic chord diagrams}. 
There is a map of $\mathfrak{S}$-modules $\PaCD(\kk) 
\longrightarrow \PaCD_{e\ell\ell}(\kk)$ and we abusively denote 
$X^{1,2}$, $H^{1,2}$ and $a^{1,2,3}$ the images in 
$\PaCD_{e\ell\ell}(\kk)$ of the corresponding arrows in $\PaCD(\kk)$. 
We have elements $X^{1,2}_{e\ell\ell}$, $Y^{1,2}_{e\ell\ell}$ in 
$\mathbf{PaCD}_{e\ell\ell}(\kk)(2)$
which are generators of the ${\PaCD}(\kk)$-module 
$\mathbf{PaCD}_{e\ell\ell}(\kk)$ and satisfy a certain number of relations.
The elliptic Drinfeld torsor over $\kk$ is the torsor
$(\widehat{\GT}_{e\ell\ell}(\kk),\Ell(\kk),\GRT_{e\ell\ell}(\kk) )$ defined by 
\begin{eqnarray*}
\Ell(\kk)&:=&\on{Iso}^+_{\tmop{OpR}\mathbf{Grpd}_\kk}
\Big(\big(\widehat{\PaB}(\KK),\widehat{\PaB}_{e\ell\ell}(\KK)\big),\big(G \PaCD(\KK),G \PaCD_{e\ell\ell}(\KK)\big)\Big)\\
\widehat{\GT}_{e\ell\ell}(\kk)&:= & \on{Aut}_{\tmop{OpR}\mathbf{Grpd}_\kk}^+
\big({\widehat\PaB(\KK)},\widehat\PaB_{e\ell\ell}(\KK)\big) \\
\GRT_{e\ell\ell}(\kk)&:=&\on{Aut}^+_{\tmop{OpR}\mathbf{Cat}(\mathbf{CoAlg}_\kk)}\big(\PaCD(\kk),\PaCD_{e\ell\ell}(\KK)\big).
\end{eqnarray*}

There is a torsor isomorphism
\begin{equation}\label{bitorsor:ell}
(\widehat{\GT}_{e\ell\ell}(\kk),\Ell(\kk),\GRT_{e\ell\ell}(\kk)) \longrightarrow 
(\widehat{\on{GT}}_{e\ell\ell}(\kk),\on{Ell}(\kk),\on{GRT}_{e\ell\ell}(\kk)),
\end{equation}
where $(\widehat{\on{GT}}_{e\ell\ell}(\kk),\on{Ell}(\kk),\on{GRT}_{e\ell\ell}(\kk))$ 
is the torsor constructed in \cite[Definition 3.12, Definition 4.1, Subsection 5.2]{En2}.

\subsection{Torsor comparisons}

Let $(\widehat{\GT}_1(\kk),\tilde{\Ell}(\kk),\GRT_{1}(\kk))$ be the Drinfeld 
$\kk$-torsor associated to $\PaB_1$ and $\PaCD_1(\kk)$. 

As we saw before, the genus $g$ Drinfeld torsor is independent of the 
framing data so there are obvious torsor isomorphisms.
\begin{eqnarray}
(\widehat{\GT}^f_1(\kk),\tilde{\Ell}^f(\kk),\GRT^f_{1}(\kk))& \to &(\widehat{\GT}_1(\kk),\tilde{\Ell}(\kk),\GRT_{1}(\kk)) \\
(\widehat{\GT}^f_{e\ell\ell}(\kk),\Ell^f(\kk),\GRT^f_{e\ell\ell}(\kk))&\to & (\widehat{\GT}_{e\ell\ell}(\kk),\Ell(\kk),\GRT_{e\ell\ell}(\kk)).
\end{eqnarray}

It remains to compare 
the reduced and non-reduced versions of the genus 1 Drinfeld torsor.

 There is a one-to-one correspondence between elements of
  $\tilde{\Ell}(\kk)$ and elements of the set 
  $\on{Ass}_1(\kk)$ consisting on tuples $(\mu,
 \varphi, A^{1,2}_{\pm})$ where $(\mu,
 \varphi) \in \on{Ass}(\kk)$ and $ A^{1,2}_{\pm} \in \exp
  (\hat{\mathfrak{t}}_{1, 2})$, satisfying
   the following equations in $\exp
  (\hat{\mathfrak{t}}_{1,1}(\kk))$:
    \begin{equation} \label{def:1:ass:0}
 \quad A_{\pm}^{\emptyset,1}=1
\end{equation}
  the following equations in $\exp
  (\hat{\mathfrak{t}}_{1, 3}(\kk))$:
  \begin{equation} \label{def:1:ass:1}
\alpha^{1,2,3}\alpha^{2,3,1} \alpha^{3,1,2}=A_{\pm} ^{(12)3}, 
\text{\ where\ }   \alpha =\varphi^{1,2,3}
A_{\pm} ^{1,23} e^{\mu(t_{12}+t_{13})/2}, 
\end{equation}
\begin{equation}\label{def:1:ass:2}
e^{\mu t_{12}} =  \big(e^{\mu t_{12}/2}\varphi^{2,1,3} A_{+}^{2,13}
(\varphi^{2,1,3})^{-1}e^{\mu t_{12}/2}, \varphi^{1,2,3}A_{-}^{1,23}(\varphi^{1,2,3})^{-1}\big),
\end{equation}
and the following equation in $\exp
  (\hat{\mathfrak{t}}_{1, 2}(\kk))$:
\begin{equation} \label{def:1:ass:3}
e^{\mu t_{12}} = \left((A^{1,2}_{+} )^{-1},A^{1,2}_{-}
 \right).  
\end{equation}

\begin{proof}
This is a straightforward application of Theorem \ref{Assg} for $g=1$ and forgetting about the framing.
\end{proof}

\begin{proposition}\label{Pres1:1}
The set $\on{Ass}_1(\kk)$ is isomorphic to the set consisting on tuples $(\mu,
 \varphi,\tilde A^{1}_{\pm}, \tilde A^{1,2}_{\pm})$ where $(\mu,
 \varphi) \in \on{Ass}(\kk)$and $ \tilde A^{1,2}_{\pm} \in \exp
  (\hat{\mathfrak{t}}_{1, 2})$, satisfying
   the following equations in $\exp
  (\hat{\mathfrak{t}}_{1,1}(\kk))$:
\begin{equation}\label{AS1}
\tilde A_{\pm}^{1,\emptyset}=1,
\end{equation}
\begin{equation}\label{AS2}
  \tilde A_{\pm}^{12, 3}  \tilde A_{\pm}^{123} =  
  \varphi^{1, 2, 3}   \tilde A_{\pm}^{1, 23}(\varphi^{1, 2, 3})^{- 1} e^{-\mu t_{12}/2}
  \varphi^{2, 1, 3}   \tilde A_{\pm}^{2, 13}  (\varphi^{2, 1, 3})^{- 1} e^{-\mu t_{12}/2},
\end{equation}
\begin{equation}\label{AS3}
( (\tilde A_+^{12,3})^{-1}, (\tilde A_-^{12,3})^{-1}\varphi^{-1} \tilde A_-^{1,23}\varphi) =
 \varphi e^{\mu t_{23}}\varphi^{-1},
\end{equation}
and the following equation in $\exp
  (\hat{\mathfrak{t}}_{1, 2}(\kk))$:
\begin{equation}\label{AS4}
e^{\mu t_{12}} = \left(\tilde A^{1,2}_{+} ,(\tilde A^{1,2}_{-})^{-1}
 \right).  
\end{equation}
\end{proposition}

\begin{proof}
By using proposition \ref{PaB:g:f:bis} applied to $g=1$ and forgetting about the framing, one can show that $\PaB_1$ has the following alternative presentation
\begin{flalign}
& \tilde Z^{1,\emptyset}=\on{Id}_{1}\,, 
\label{eqn:Rbis} \\
& \tilde Z^{123} \tilde Z^{12,3}=\Phi^{1,2,3}\tilde Z^{1,23}(\Phi^{1,2,3})^{-1} 
(R^{2,1})^{-1}\Phi^{2,1,3}\tilde Z^{2,13}(\Phi^{2,1,3})^{-1}(R^{1,2})^{-1}\,, 
\label{eqn:Dbis} \\
& \Phi^{1,2,3}R^{2,3}R^{3,2}(\Phi^{1,2,3})^{-1}=\left((\tilde A^{12,3})^{-1}, 
\tilde B^{12,3}\Phi^{1,2,3}(\tilde B^{1,23})^{-1}(\Phi^{1,2,3})^{-1} \right)\, 
\label{eqn:E1bis} \\
& R^{1,2}R^{2,1}=\left(\tilde A^{1,2},(\tilde B^{1,2})^{-1}\right). \label{eqn:E2bis} &	
\end{flalign}
Then the equivalence between equations \eqref{def:1:ass:0}--\eqref{def:1:ass:3} and equations \eqref{AS1}--\eqref{AS4} follow straightforwardly.
\end{proof}

One can prove 
using the monodromy of the non-reduced version of the universal elliptic 
KZB connection contained in \cite{CEE} that $\tilde{\Ell}(\C)$ is not empty.  

\begin{theorem}
The set $\tilde{\Ell}(\C)$ is not empty.
\end{theorem}

\begin{proof}
The proof goes along the same lines as the one in \cite{CEE} but since our 
conventions for the fundamental group generators and monodromy actions 
differ from it, we give the proof in full detail.
Recall our conventions for monodromy actions from \cite[Appendix A.]{CG}.
In \cite{CEE}, it was shown that there is a flat universal elliptic KZB connection 
over $\Conf(E_\tau,n)$, where $\tau\in\h$ and $E_\tau$ is a normalized 
elliptic curve extending to the non reduced moduli space of marked elliptic 
curves.
For $\tau\in\mathfrak H$, let $U_{\tau,n}\subset\C^n-{\rm Diag}_{\tau,n}$ be the 
open subset of all $\zz=(z_1,\dots,z_n)$ of the form $z_i=a_i+\tau b_i$, where $0<a_1<\cdots<a_n<1$ 
and $0<b_n<\cdots<b_1<1$. If $\zz_0\in U_{\tau,n}$, then it defines a point both in the ordered and 
unordered configuration spaces $\textrm{Conf}(E_{\tau},n)$ and $\textrm{Conf}(E_{\tau},[n])$.
As $U_{\tau,n}$
is simply connected, a solution of the elliptic KZB system in 
\cite[Subsection 4.1]{CEE} on this domain is then unique. Then there is a unique solution $F^{(n)}(\zz)$ 
with the prescribed expansion of \cite[Subsection 4.1]{CEE} to this system 
on $U_{\tau,n}$. The domains $H_n:= \{\zz\in\C^n | z_i = a_i + b_i \tau, 
a_i,b_i\in\RR, 0<a_1<a_2<...<a_n<1\}$ and $D_n := \{\zz\in
\C^n | z_i = a_i + b_i \tau, a_i,b_i\in\RR, 
0<b_n<\cdots<b_1<1\}$ are also simply connected and invariant, 
and we denote by $F^H(\zz)$ and $F^V(\zz)$ the prolongations
of $F^{(n)}(\zz)$ to these domains.
Then we have solutions of the elliptic (non-reduced)
KZB system on $H_n$ and $D_n$ given by $\zz\mapsto F^H(\zz + \sum_{j=i}^n \delta_i)$
and $\zz\mapsto e^{2\pi\i( x_i + ... + x_n)}F^V(\zz + 
\tau(\sum_{j=i}^n \delta_i))$ respectively. We define $A_i^F,B_i^F
\in \on{exp}(\hat{{\t}}_{1,n})$ by 
$$
F^H(\zz)=
A_i^F F^H\left(\zz + \sum_{j=i}^n \delta_j\right),
$$
$$
F^V(\zz)=
B_i^Fe^{2\pi\i( x_i + ... +x_n)}F^V\left(\zz + \tau(\sum_{j=i}^n \delta_j)\right).
$$
Let denote by $\gamma_n : \on{B}_{1,n} \to \on{exp}(\hat{{\t}}_{1,n})\rtimes \mathfrak{S}_{n}$ 
the unique morphism induced by the solution $F^{(n)}(\zz)$ and, 
for all $1\leq i \leq n$, denote $A_i$ for the class of the projection of the path 
$[0,1]\ni t\mapsto (\zz_0 + t\sum_{j=i}^n \delta_j)$, 
$B_i$ for the class of the projection of $[0,1]\ni t \mapsto 
(\zz_0 + t\tau \sum_{j=i}^n \delta_j)$. The paths $A_i, B_i$ are generators 
of $\on{B}_{1,n}$.

Let us also denote $\tilde A^{1}_{KZB} := \gamma_1(A_1)$, $\tilde B_{KZB}^{1} :=
 \gamma_1(B_1)$  $\tilde A_{KZB}^{1,2} := \gamma_2(A_2)$, $\tilde B_{KZB}^{1,2} :=
  \gamma_2(B_2)$ and denote by $ \varphi_{KZ}$ the KZ associator. 
It is then clear that 
$$
\tilde Z^{1,\emptyset}=1, \quad \tilde Z^{\emptyset,1}=Z^1.
$$
for $ Z$ any of the letters $\tilde A_{KZB},\tilde B_{KZB}$. Then, our presentation 
of $\on{B}_{1,n}$ implies that we have relation $A_{3}^{-1} A_{2}= 
\sigma_{1}A_{2}^{-1}\sigma_{1}A_1$ and 
$B_{3}^{-1}B_{2} = \sigma_{1}B_{2}^{-1}\sigma_{1}B_1$. 
Then the image by $\gamma_3$ of these relations yield 
\begin{equation}\label{eqn:D1}
  Z^{12, 3}Z^{123}  =  
  \varphi_{KZ}^{1, 2, 3}   Z^{1, 23}(\varphi_{KZ}^{1, 2, 3})^{- 1} e^{-\pi\i t_{12}}
  \varphi_{KZ}^{2, 1, 3} Z^{2, 13}  (\varphi_{KZ}^{2, 1, 3})^{- 1} e^{-\pi\i t_{12}}
\end{equation}
for $ Z$ any of the letters $\tilde A_{KZB},\tilde B_{KZB}$. 
Next, the image by $\gamma_{3}$
of $(A^{-1}_{3},B^{-1}_{3}B_{2})= P_{23}$ then gives 
\begin{equation} \label{eqnE1}
((\tilde A_{KZB}^{12,3})^{-1},(\tilde B_{KZB}^{12,3})^{-1}
\varphi_{KZ}^{-1}\tilde B_{KZB}^{1,23}\varphi_{KZ}) =
 \varphi_{KZ} e^{2\pi\i t_{23}}(\varphi_{KZ})^{-1}.
\end{equation}
Finally,  the image by $\gamma_{2}$
of $(A_{2},B_2^{-1})= P_{12}$ then gives 
$$
e^{\mu t_{12}} = \left(\tilde A_{KZB}^{1,2} ,(\tilde B_{KZB}^{1,2})^{-1}
 \right).  
$$
By Proposition \ref{Pres1:1}, we obtain that $(2\pi\i, \varphi_{KZ},  A_{KZB},  B_{KZB})$, 
where $A_{KZB}=\gamma_2(X_1),B_{KZB}=\gamma_2(Y_1)$ satisfy relations 
\eqref{def:1:ass:0}, \eqref{def:1:ass:1}, \eqref{def:1:ass:2} and \eqref{def:1:ass:3} so that it is in $\on{Ass}_1(\C)$.
\end{proof}
We then get a surjective map of torsors 
$$
(\widehat{\GT}_1(\kk),\tilde{\Ell}(\kk),\GRT_{1}(\kk))\longrightarrow 
(\widehat{\GT}_{e\ell\ell}(\kk),\Ell(\kk),\GRT_{e\ell\ell}(\kk)).
$$
By functoriality and by the fact that the operadic pointing both in $\PaB_1$ and 
$\PaB_{e\ell\ell}$ have been chosen to be the unit, we retrieve maps of torsors  
$$
(\widehat{\GT}(\kk),\Ass(\kk),\GRT(\kk))\to
(\widehat{\GT}_1(\kk),\tilde{\Ell}(\kk),\GRT_{1}(\kk))\to (\widehat{\GT}_{e\ell\ell}(\kk),
\Ell(\kk),\GRT_{e\ell\ell}(\kk)).
$$
The composite of these maps is an operadic version of the map first 
constructed in \cite{En2}.

\end{document}